\documentclass{article}

\usepackage[dvipsnames]{xcolor}
\pdfoutput=1
\usepackage[total={5.6in, 8.5in}]{geometry}
\usepackage{amsfonts,amssymb}
\usepackage{float}
\usepackage{hyperref}
\usepackage{enumerate}
\usepackage{amsmath, amsthm, amsfonts, mathabx}
\usepackage[capitalise]{cleveref}
\usepackage[toc]{appendix}
\usepackage{titling}
\usepackage{authblk}

\usepackage{threeparttable}
\usepackage{tikz,pgf}
\usetikzlibrary{shapes,decorations,arrows,calc,arrows.meta,fit,positioning}
\usetikzlibrary{decorations.markings}
\usetikzlibrary{arrows, automata}
\usetikzlibrary{decorations.pathreplacing}

\usepackage[official]{eurosym}

\usepackage{nomencl}
\usepackage{bm}
\usepackage{amsfonts}
\usepackage{fancyhdr}
\usepackage{amscd}
\usepackage[english]{babel}
\usepackage{xcolor}
\usepackage[capitalise]{cleveref}
\usepackage{mathtools}

\usepackage{array}
\newcolumntype{C}[1]{>{\centering\let\newline\\\arraybackslash\hspace{0pt}}m{#1}}

\usepackage{caption}
\newcommand{\draft}[1]{}

\pretitle{\begin{center} \LARGE \bf}
\posttitle{\par\end{center}\vskip 0.5em}
\preauthor{
	\begin{center}
		\normalsize \lineskip 0.5em%
		\begin{tabular}[t]{c}
		}
	\postauthor{\end{tabular}\par\end{center}}
	
	\predate{\begin{center}\large}
	\postdate{\par\end{center}}

	\preauthor{
		\begin{center}
			\normalsize \lineskip 0.5em%
			\begin{tabular}[t]{c}
			}
		\postauthor{\end{tabular}\par\end{center}}
		
		\predate{\begin{center}\large}
		\postdate{\par\end{center}}

		\newtheorem{thm}{Theorem}[section]
		\newtheorem{theoremletter}{Theorem}
		\newtheorem{corollaryletter}[theoremletter]{Corollary}
		
		\newtheorem{cor}[thm]{\scshape{Corollary}}
		\newtheorem{prop}[thm]{\scshape{Proposition}}
		\newtheorem{lemma}[thm]{\scshape{Lemma}}
		\newtheorem{claim}[thm]{\scshape{Claim}}
		\newtheorem{question}[thm]{\scshape{Question}}
		\newtheorem{questionletter}[theoremletter]{Question}
		
		\newtheorem{conj}[thm]{\scshape{Conjecture}}
		\newtheorem{remark}[thm]{\scshape{Remark}}
		\newtheorem{obs}[thm]{\scshape{Observation}}
		
		\newtheorem{example}[thm]{\scshape{Example}}
		\newtheorem*{thm*}{Theorem}
		\newtheorem*{cor*}{Corollary}
		
		\theoremstyle{definition}
		\newtheorem{definition}{Definition}[thm]

		\def\={\cong}

		\def\mc{\mathcal}
		\def\mbb{\mathbb}

		\DeclarePairedDelimiter{\ceil}{\lceil}{\rceil}

  \def\={\cong}

  \def\mc{\mathcal}
  \def\mbb{\mathbb}

		\newcommand{\subg}[1]{\langle #1 \rangle}
	\newcommand{\enum}[2]{\begin{enumerate}[\hspace*{0.5cm}#1] #2 \end{enumerate}}
		
		\newcommand{\frp}[0]{\ast}
		\newcommand{\F}[0]{\mathbb{F}}
		\newcommand{\Z}[0]{\mathbb{Z}}

		\newcommand{\nin}[0]{\notin}

		\newcommand{\bvee}[3]{ \bigvee\limits_{#1}^{#3}#2}
		\newcommand{\bwedge}[3]{\bigwedge\limits_{#1}^{#3}#2}

		\newcommand{\dis}[1]{d(*, #1 *)}

		\usepackage{verbatim}
		\usepackage{amsmath}  \usepackage{scalefnt}

		\title{On the positive theory of groups acting on trees}

		\author{Montserrat Casals-Ruiz\footnote{Ikerbasque - Basque Foundation for Science and Matematika Saila,  UPV/EHU,  Sarriena s/n, 48940, Leioa - Bizkaia, Spain; Contact: \texttt{montsecasals@gmail.com}},\quad \quad Albert Garreta\footnote{\small{Matematika Saila,  UPV/EHU,  Sarriena s/n, 48940, Leioa - Bizkaia, Spain;} Contact: \texttt{garreta.a@gmail.com}},\quad \quad Javier de la Nuez Gonz\'alez\footnote{Matematika Saila,  UPV/EHU,  Sarriena s/n, 48940, Leioa - Bizkaia, Spain; Contact: \texttt{javier.delanuez@ehu.eus}\newline
	
	The authors were supported by the ERC grant PCG-336983, the Basque Government Grant IT974-16 and the Spanish Government grant MTM2017-86802-P.}
		}

		\date{}
		
\begin{document}
		\maketitle
		\tikzset{
			middlearrow/.style={
			decoration={
			markings, mark= at position 0.5 with {\arrow{#1}} ,
			},
			postaction={decorate}
		}
	}
	\tikzset{
		endarrow/.style={
		decoration={
		markings, mark= at position 1 with {\arrow{#1}} ,
		},
		postaction={decorate}
	}
}
\tikzset{->-/.style={decoration={
  markings,
  mark=at position .5 with {\arrow{>}}},postaction={decorate}}}

\begin{abstract}
In this paper we study the positive theory of groups acting on trees and show that under the presence of weak small cancellation elements, the positive theory of the group is trivial, i.e. coincides with the positive theory of a non-abelian free group. Our results apply to a wide class of groups, including non-virtually solvable fundamental groups of 3-manifold groups, generalised Baumslag-Solitar groups and almost all one-relator groups and graph products of groups. It follows that groups in the class satisfy a number of algebraic properties: for instance, their verbal subgroups have infinite width and, although some groups in the class are simple, they cannot be boundedly simple. 

In order to prove these results we describe a uniform way for constructing (weak) small cancellation tuples from (weakly) stable elements. This result of interest in its own is fundamental to obtain corollaries of general nature such as a quantifier reduction for positive sentences or the preservation of the non-trivial positive theory under extensions of groups.
\end{abstract}

\tableofcontents

\newgeometry{total={5.5in, 8.5in}}

\section{Introduction}
  
  Laws (identities), surjectivity of word maps, finite verbal width and bounded simplicity are examples of classical algebraic properties which are widely studied in groups and can actually be addressed in the common framework of the positive theory of the group.
  
  The positive theory of a group is the fragment of the elementary theory composed by the set of first-order sentences that do not contain negations. In other words, it is the set of sentences of the form:
  $$
  \forall x_1 \exists y_1 \dots \forall x_n \exists y_n \ \bigvee\limits_{i=1, \dots, r} \left( \bigwedge\limits_{j=1, \dots, s} \ \Sigma_{i,j}(x_1, y_1, \dots, x_n,y_n)=1 \right)
  $$
  satisfied by the group, where $\Sigma_{i,j} \in F(x_1, y_1, \dots, x_n,y_n)$.

  During the last decades, the model-theoretic study of groups has proven to be the motor for developing techniques and establishing new connection between different areas. One of the most remarkable achievements in the model-theoretic study of groups is the solution to Tarski problems for free groups. The first key step in the study of the elementary theory of a free group was precisely the study of its positive theory. In \cite{merzlyakov1966positive}, Merzlyakov proved an implicit function theorem for positive sentences: if a free group satisfies a positive sentence $\forall x_1 \exists y_1, \dots, \forall x_n \exists y_n \ \Sigma(x_1, \dots, x_n, y_n)=1$, then one can find words in variables $x_i$ which witness the truth of the sentence, that is there exist algebraic expressions $w_i(x_1, \dots, x_i)$, $i=1, \dots, n$, called formal solutions, such that
  $$
  \Sigma(x_1, w_1(x_1), \dots, x_n, w_n(x_1, \dots, x_n))=1
  $$ in the free group $F \ast F(x_1, \dots, x_n)$. As a consequence, all non-abelian free groups have the same positive theory and every positive sentence satisfied by a non-abelian free group is satisfied by every group. In view of this, we say that a group $G$ has \emph{trivial positive theory} if its positive theory coincides with the one of non-abelian free groups, i.e. $Th^+(G)=Th^+(F_2)$.
  
  For the general theory of a non-abelian free groups, Sela showed in \cite{sela2006diophantineV} that this theory admits quantifier elimination to the Boolean algebra of $\forall\exists$-formulas and furthermore, Kharlampovich-Miasnikov proved that this quantifier elimination can be performed algorithmically, see \cite{kharlampovich2006elementary}. However, for general $\forall \exists$-sentences one does not have an implicit function theorem and this makes the validation process for $\forall \exists$-sentences in the free group substantially more complicated than for positive sentences.
  
  Merzlyakov's result on the positive theory of non-abelian free groups have been generalised from free groups to free products of groups, \cite{sacerdote1973almost,sela2010diophantine}, hyperbolic groups \cite{sela2009diophantineVII, heilthesis}, right-angled Artin groups \cite{casals2010elements,diekert2004existential}, graph products \cite{diekert2004existential}, HNN-extensions and free products with amalgamation over finite groups \cite{lohrey2006positive} and for torsion-free acylindrically hyperbolic groups \cite{jonathanthesis}.
  
  \bigskip
  
  In this paper we study the positive theory of groups acting on trees. More precisely, we prove the following theorem.
  
  \begin{theoremletter}[Theorem \ref{thm:characterisation}]
  	If a group $G$ acts minimally on a simplicial tree and the action on the boundary neither fixes a point nor is 2-transitive, then its positive theory is trivial, that is the group $G$ satisfies exactly the same set of positive sentences as a non-abelian free group. In particular, the positive theory is decidable: there is an algorithm that given a positive sentence in the language of groups determines whether or not the sentence is satisfied in the group.
  \end{theoremletter}
  
  This result applies to many classes of groups for which we can provide a precise criterion for the triviality of their positive theory.
  
  \begin{corollaryletter}[Corollary \ref{cor:summaryex}]
  	Non-virtually solvable fundamental groups of closed, orientable, irreducible 3-manifolds; non-solvable generalised Baumslag-Solitar groups; (almost all) non-solvable one-relator groups; and graph products of groups whose underlying graph is not complete have trivial positive theory.
  \end{corollaryletter}

  As we mentioned, the triviality of the positive theory has many nice algebraic implications.
  
  \begin{corollaryletter}[Corollary \ref{cor:summaryex}]
  The following families of groups do not have   verbal subgroups of finite width, and they are not boundedly simple:
  	non-virtually solvable fundamental groups of closed, orientable, irreducible 3-manifolds; non-solvable generalised Baumslag-Solitar groups; (almost all) non-solvable one-relator groups; and graph products of groups whose underlying graph is not complete.
  \end{corollaryletter}
  
  Related results can be found in
\cite{bestvina2019verbal,myasnikov2014verbal} and references thereof.
  
  \bigskip
  
  The key ingredient in the proof of Mezlyakov's result on the existence of formal solutions for free groups and for the later generalisations to other groups is small cancellation.
  
  Small cancellation has been extensively used in geometric group theory: to prove embeddability results, for instance that every countable group can be embedded into a 2-generated simple group \cite{schupp1976embeddings}; to produce examples of groups with exotic properties: infinite Burnside groups, Tarski and Ol'shanskii monsters, hyperbolic groups with wild subgroups given by the Rips' construction and finitely presented groups not admitting a uniform embedding into a Hilbert space \cite{gromov2003random}; to provide a model for random finitely presented groups; to find non-trivial quasi-morphisms, etc. 
  
  In our work, we establish a new relation between a weaker version of small cancellation\footnote{Our version of weak small cancellation is closely related, and in some settings equivalent, to the ones introduced in \cite{iozzi2014characterising, bestvina2015constructing}).} and the triviality of the positive theory. This new approach allows us to address a wider class of groups which do not necessarily contain small cancellation elements (for instance, Baumslag-Solitar groups). 
  
  Informally speaking, a tuple is small cancellation if the minimal tree $M$ spanned by the elements of the tuple is ``almost" independent from all its translates by elements of the group, i.e. it intersects its coset $hM$ by the action of any element $h$ in a ``small" segment; in this way, a tuple is weak small cancellation if the minimal tree $M$ is almost ``dynamically independent" from all its translates, i.e. if it intersects its coset $hM$ by the action of any element $h$ in a ``long" segment, then the segment is the same and the element $h$ acts as the identity on it (we do not allow for interval exchange transformations on the dynamics of the minimal tree), see Definition \ref{d: small cancellation}.
  
  The relation between weak small cancellation and the positive theory is established in the following theorem. 
  
  \begin{theoremletter}[Theorem \ref{thm: wscitp}]
  	If $G$ acts irreducibly and minimally on a tree and contains weak small cancellation elements, then its positive theory is trivial.
  \end{theoremletter}
  
  \medskip
  
  The connection between Theorem A and Theorem D is established in \cite{iozzi2014characterising}, where the authors show that if a group $G$ acts minimally on a simplicial tree with vertex degree at least 3 and the action on the boundary neither fixes a point nor is 2-transitive, then the group $G$ admits what they call a ``good labeling"; they then proceed by showing that in this case, there are segments with small cancellation features from which they build non-trivial quasi-morphisms and prove that the second bounded cohomology of the group is infinite dimensional. In our case, we show that the existence of a good labeling and more precisely the existence of segments with small cancellation features is equivalent to having weak small cancellation elements. The relation between these two very different notions, triviality of the positive theory and infinite dimensional second bounded cohomology, seems to go further. In \cite{bestvina2015constructing}, the authors introduce the notion of WWPD element and show that for groups acting on quasi-trees, the existence of WWPD elements implies that the group has infinite dimensional second bounded cohomology. In the context of groups acting on trees, we prove that the existence of a WWPD element, which in our context coincides with the notion of weakly stable element, see Definition \ref{denf:weakly stable}, also implies that the group has weak small cancellation elements and so trivial positive theory. In a forthcoming work, we intend to keep exploring this connection and generalise our results on the positive theory to groups acting on more general spaces under the presence of WWPD elements.
  
  In the case of groups acting on hyperbolic spaces, the existence of a WPD element is equivalent to the existence of a small-cancellation tuple, \cite{hull2013small}. We show that, in fact, this relation can be established in a uniform way, i.e. given a stable element, one can uniformly describe a small cancellation tuple dominating an arbitrary tuple $c$, see Definition \ref{d: small cancellation}. However, in general, we do not know if the equivalence between stable element and small cancellation holds when one weakens these two properties, see Section \ref{sec:openquestions}. In any case, we prove that given a weakly stable element, there is again a uniform way to determine a weakly small cancellation tuple.
  
  \begin{theoremletter}[Corollary \ref{cor:uniformWSC}]
  	For all $N,m \geq 1$, there exist two $m$-tuple of words $w^{sc}, w^{sc}{}'\in\F(x,y)^{m} $ with the following property. Suppose we are given an action of a group $G$ on a tree $T$, two elements $g,h$ such that the subgroup $\subg{g,h}$ acts irreducibly on $T$, $h$ is hyperbolic and stable (resp. weakly stable) and $d(A(h),A(g))< tl(h)$. Then either $w(g,h)^{sc}$ or $w(g,h)^{sc}{}'$ is $N$-small cancellation (respectively, weakly $N$-small cancellation). 
  
  \end{theoremletter}
  
  \medskip
  
  We believe that these uniformity results, although technical, are interesting in their own and have proven to be useful tools to obtain other results.  Keeping in mind this further applications, the results are presented in a bit more generality than strictly needed for this paper. One of the results we obtained is an effective quantifier reduction for positive sentence in all groups. This reduction states that in order to check if a group has trivial positive theory, it is sufficient to determine it for positives $\forall \exists$-sentences. More precisely, we prove the following theorem.
  
  \begin{theoremletter}[Theorem \ref{l: quantifier reduction}]
  	Given any non-trivial positive sentence $\phi$, one can effectively describe a non-trivial positive $\forall\exists$-sentence $\phi'$ such that $\phi$ implies $\phi'$ in the theory of groups, that is for any group $G$, if $G \models \phi$, then $G \models \phi'$. In particular, if a group has non-trivial positive theory, it must satisfy some non-trivial positive $\forall\exists$-sentence.
  \end{theoremletter}
  In the same way, uniformity results are key in proving that the property of having non-trivial positive theory is closed under extensions.
  
  \begin{theoremletter}[Theorem \ref{lem:extensionpreservation}]
  	The class $\mathcal C$ of groups with non-trivial positive theory is closed under extensions, i.e. if $N,Q \in \mathcal C$ and there is an exact sequence $1 \to N \to G \to Q \to 1$ for $G$, then $G\in \mathcal C$.
  \end{theoremletter}
  
  Another corollary from the uniformity is that (most) groups acting on trees have uniform exponential growth, recovering results from \cite{de2000free}.
  
  Last but not least, in the forthcoming work \cite{graphproductspreprint} we use these results in as essential way to prove the interpretability of the graph and the vertex groups in a graph product of groups with vertex groups having non-trivial positive theory.
  
  \bigskip
  
  It is an easy observation that if a group $G$ quotients onto a group $H$, then the positive theory of $G$ is contained in the positive theory of $H$. In particular, all large groups, that is, groups for which a finite index subgroup possess a homomorphism onto a non-abelian free group, have trivial positive theory since they project to virtually free groups (non virtually abelian). In view of this, one may wonder if there is a relation between having trivial positive theory and being large or more, generally, SQ-universal. As we already pointed out, if the group is large, then its positive theory is trivial. However, the question for SQ-universal groups is not known. More precisely, we ask the following
  
  \begin{questionletter}
  	Does every SQ-universal group have trivial positive theory?
  \end{questionletter}
  
  On the hand, the converse is far from being true. Indeed, in the preprint \cite{simplegrouppreprint}, the authors provide with an uncountable family of finitely generated simple groups with trivial positive theory and so with all verbal subgroups of infinite width. This result generalised Muranov's construction of a finitely generated simple group with infinite commutator width, \cite{muranov2010finitely}. Furthermore, it also exhibits a clear contrast with the behaviour of the class of finite simple groups for which for each word there is a uniform bound on the width of the corresponding verbal subgroup.
  
  As we mentioned, weak small cancellation is the key tool to prove that a group has trivial positive theory but also to show the existence of non-trivial quasi-morphisms. This bring us to the following question:
  
  \begin{questionletter}
  	Is there are group with trivial positive theory and finite-dimensional second bounded cohomology? Is there a group with infinite dimensional second bounded cohomology and non-trivial positive theory? 
  \end{questionletter}
  
  Although it would be really surprising if indeed there is a relation between these two very different properties, the construction of an example that satisfies one but not the other one would require new techniques, which we believe could be of interest in their own.
  
  Another consequence of using weak small cancellation elements to prove that the positive theory is trivial is that all these groups contain free groups and so have exponential growth. Hence, we ask the following:
  
  \begin{questionletter}
  	Does every group with trivial positive theory have exponential growth?
  \end{questionletter}
  
  In Section \ref{sec:openquestions}, we discuss these and some other open questions.

\section{Preliminary results on actions on trees}\label{s: preliminaries}

  In this subsection we fix some terminology and provide some basic results regarding isometries and group actions on trees. The reader is referred to Chapter 3 of \cite{chiswell2001introduction} for a rigorous and detailed presentation of the subject at hand. Unless stated otherwise, throughout the whole paper, except in Section \ref{sec: characterisation}, $T$ denotes a real tree with distance metric $d$.

 \subsubsection*{Basics}
 
  An isometry $g$ of a simplicial tree $T$ that does not invert edges is either \emph{elliptic} if it fixes a vertex (also called \emph{point}), or \emph{hyperbolic} otherwise. In the first case there exists a unique maximal subtree $A(g)$ of $T$ that is fixed point-wise by $g$, and $g$ may be seen as `rotating' around $A(g)$. In the second case there exists an infinite line $A(g)$ that is preserved by $g$, and $g$ acts on $T$ as a translation along $A(g)$. Regardless of whether $g$ is elliptic or hyperbolic we call $A(g)$ the \emph{axis} of $g$. The \emph{translation length} of an isometry $g$, denoted $tl(g)$, is  defined as the infimum of the displacements by $g$ of a point, i.e.\ $tl(g)=\inf_{v\in T} d(v, gv)$.
  
   In this paper we work with groups $G$ that cat on trees by isometries on a tree $T$, i.e.\ groups that embed in the group of isometries of $T$.  In this case, given a vertex $v$ and two elements from a group $g,h$ we wrote $gh \cdot v$ (or simply $ghv$) to refer to the vertex $g\cdot(h\cdot v)$.

   We next present several basic results regarding tree isometries. In order to match the notation introduced above, we denote the composition of two isometries $g,h$ by $gh$, and we agree that $h$ acts before $g$. 
  
  The following basic lemma is fundamental for our paper. It will be used extensively without referring to it.
  
  \begin{lemma}[Lemma 1.7, Chapter 3, \cite{chiswell2001introduction}]\label{l: chiswell1}
  	If $g,h$ are both isometries of $T$, then
  	\begin{enumerate}
  		\item $A(g^h) = h^{-1}A(g)$.
  		\item $A(g^{-1})= A(g)$.
  		\item If $n$ is an integer then $tl(g^n) = n tl(g)$ and $A(g^n)\subseteq A(g)$. If $n\neq 0$ and $tl(g)>0$ then $A(g^n) =A(g)$.
  	\end{enumerate}
  \end{lemma}

  The following  remark can be verified in a straightforward way with the help of Item 1 of Lemma \ref{l: chiswell1}.
  
  \begin{remark}\label{r: alb_remark_1}
  	Let $g,h$ be two isometries  such that $A(g)$ and $A(h)$ are disjoint. Then
  	\begin{enumerate}
  		\item both sets are mutually disjoint from $A(h^{g})$,
  		\item the path from $A(h^{g})$ to $A(h)$ has nonempty intersection with $A(g)$. 
  	
  	\end{enumerate}
  \end{remark}
  
  \subsubsection*{The axis of the composition of two isometries}

  The following  lemmas and their corresponding diagrams provide valuable information regarding the product of two isometries. 
  
  By \emph{bridge} between two axes $A(g)$ and $A(h)$ we mean the shortest path connecting them.

  \begin{lemma}[Lemma 2.2 of Chapter 3 \cite{chiswell2001introduction}]\label{l: chiswell2}
  	Let $g,h$ be two elements acting on a tree which are not inversions. Suppose $A(g) \cap A(h) =\emptyset$. Let $[p,q]$ be the bridge between $A(g)$ and $A(h)$, with $p \in A(g)$ and $q \in A(h)$. Then $[p,q]\subseteq A(gh)$, and
  	$$
  	tl(gh) = tl(g) + tl(h) + 2d(A(g), A(h)).
  	$$
  	Moreover, $A(gh) \cap A(h) = [q, h^{-1}q]  = tl(h)$ and $A(gh)\cap A(g) =[p,gp] = tl(g)$, segments of length $tl(h)$ and $tl(g)$, respectively. See Figure \ref{ffff2} where we have depicted $A(gh)$ in case both $g$ and $h$ are hyperbolic.
  \end{lemma}

  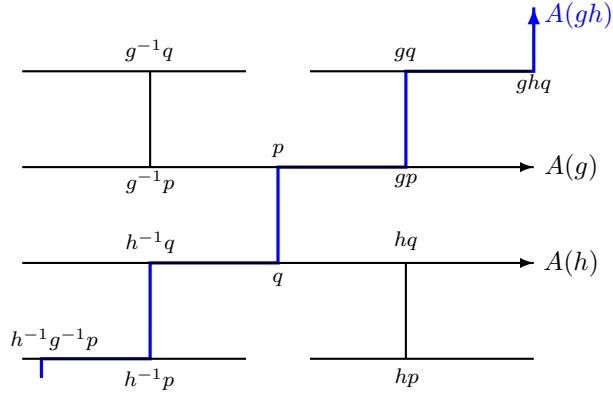
\begin{figure}[ht]
  	\begin{center}
  		\begin{tikzpicture}[scale=0.85]
  			\tikzset{near start abs/.style={xshift=1cm}}

  			\draw[color=blue,very thick,endarrow={latex}] (0.3, -0.3)-- (0.3,0) -- (2,0)-- (2,1.5) --(4,1.5) -- (4,3) -- (6,3)-- (6,4.5) -- (8,4.5)-- (8,5.5) node[xshift=0.6cm, yshift=-0.1cm] {$A(gh)$};
  			
  			\draw[thick,endarrow={latex}] (0, 1.5) -- (8,1.5) node[xshift=0.5cm] {$A(h)$} node[midway, below] {\footnotesize{$q$}};
  			\draw[thick,endarrow={latex}, yshift=1.5cm] (0, 1.5) -- (8,1.5) node[xshift=0.5cm] {$A(g)$} node[midway, above] {\footnotesize{$p$}};
  			
  			\node at (0.5, 0.3) {\footnotesize{$h^{-1}g^{-1}p$}};
  			\draw[thick] (0, 0)-- (3.5,0) node[midway, below, xshift=0.2cm] {\footnotesize{$h^{-1}p$}};
  			\draw[thick, xshift=4.5cm] (0, 0)-- (3.5,0) node[midway, below, xshift=-0.2cm] {\footnotesize{$hp$}};
  			\draw[thick, yshift=4.5cm] (0, 0)-- (3.5,0) node[midway, above, xshift=0.2cm] {\footnotesize{$g^{-1}q$}};
  			\draw[thick, yshift=4.5cm, xshift=4.5cm] (0, 0)-- (3.45,0) node[midway, above, xshift=-0.2cm] {\footnotesize{$gq$}};
  			
  			\draw[thick] (6, 0)-- (6,1.5) node[yshift=0.3cm] {\footnotesize{$hq$}} node[yshift=1.1cm] {\footnotesize{$gp$}};
  			\draw[thick] (2, 4.5)-- (2,3) node[yshift=-0.2cm] {\footnotesize{$g^{-1}p$}} node[yshift=-1cm] {\footnotesize{$h^{-1}q$}};
  			
  			\node at (8,4.3) {\footnotesize $ghq$};
  			
  		\end{tikzpicture}
  	\end{center}
  	\caption{The axis of $gh$ under the assumptions of Lemma \ref{l: chiswell2}, assuming $g$ and $h$ are hyperbolic. We have also depicted the image of the points $p$, $q$ under different  action compositions of the isometries $g^{\pm 1}, h^{\pm 1}$.} \label{ffff2}
  \end{figure}

  Two isometries $g, h$ are said to meet \emph{coherently} if $A(g) \cap A(h) \neq \emptyset$ and if both are hyperbolic then the translation direction of $g$ and $h$ coincides on $A(g) \cap A(h)$ (see page 100 in \cite{chiswell2001introduction}). The following is the analogue of Lemma \ref{l: chiswell2} for the case when $g$ and $h$ meet coherently.
  
  \begin{lemma}[Lemma 3.1 of Chapter 3 \cite{chiswell2001introduction}]\label{l: chiswell3}
  	Assume $g,h$ are hyperbolic isometries which meet coherently.  Then $gh$ meets both $g$ and $h$ coherently, $tl(gh)=tl(g) + tl(h)$,  $|A(gh)\cap A(h)| = |A(g) \cap A(h)| +tl(h)$, and  $|A(gh)\cap A(g)| = |A(g) \cap A(h)| + tl(g)$. See Figure \ref{f: fff} for a depiction of $A(gh)$.
  
  \end{lemma}

  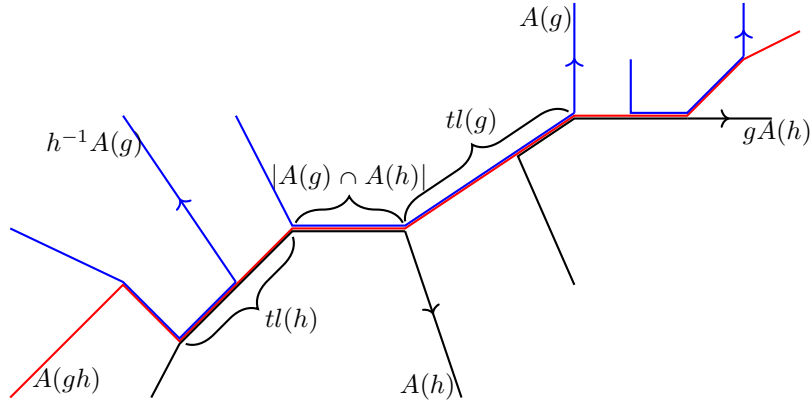
\begin{figure}[ht]
  	\begin{center}
  		\begin{tikzpicture}[scale=0.75]
  			\tikzset{near start abs/.style={xshift=1cm}}
  			
  			\draw[color=red, thick] (0,0) -- (2,2) -- (3,1) -- (5,3) -- (7,3) --(10,5)--(12,5)--(13,6)--(14,6.5);
  			\node (xx) at (1,0.3) {\color{red}{$A(gh)$}};

  			\draw[color=blue,thick] (0,3) -- (2,2.05)-- (3,1.05) -- (4,2.05);
  			\draw[color=blue,thick,middlearrow={>}] (4,2.05) -- (2,5);
  			\node (xx) at (1.5,4.5) {\color{blue}{$h^{-1}A(g)$}};

  			\draw[thick] (2.5,0) -- (3,0.95)-- (5,2.95) --(7,2.95);
  			\draw[thick,middlearrow={>}] (7,2.95) -- (8,0);
  			\node (xx) at (7.4,0.2) {{$A(h)$}};
  			
  			\draw [thick,decorate,decoration={brace,amplitude=10pt},xshift=-4pt,yshift=0pt]
  			(5.15,2.85) -- (3.25,0.95) node [black,midway,xshift=0.7cm, yshift=-0.4cm]
  			{$tl(h)$};
  			\draw [thick,decorate,decoration={brace,amplitude=10pt},xshift=-4pt,yshift=0pt]
  			(5.2,3.1) -- (7.1,3.1) node [black,midway,yshift=0.6cm] {$|A(g)\cap A(h)|$};
  			
  			\draw[color=blue,thick] (4,5) -- (5,3.05)-- (7,3.05) --(10,5.05);
  			\draw[color=blue,thick,middlearrow={>}] (10,5.05) -- (10,7);
  			\node (xx) at (9.5,6.7) {\color{blue}{$A(g)$}};

  			\draw [thick,decorate,decoration={brace,amplitude=10pt},xshift=0pt,yshift=5pt]
  			(7,2.95) -- (9.9,4.9) node [black,midway,xshift=-0.2cm, yshift=0.6cm]
  			{$tl(g)$};
  			
  			\draw[color=blue,thick] (11,6) -- (11,5.05)-- (12,5.05) --(13,6.05);
  			\draw[color=blue,thick,middlearrow={>}] (13,6.05) -- (13,7);

  			\draw[thick] (10,2) -- (9,4.28)-- (10,4.95) --(12,4.95);
  			\draw[thick,middlearrow={>}] (12,4.95) -- (13.5,4.95);
  			\node (xx) at (13.6,4.7) {{$gA(h)$}};

  		\end{tikzpicture}
  	\end{center}
  	\caption{The axis of $gh$ (in red) under the assumptions of Lemma \ref{l: chiswell3}.  The black and blue lines are translates of $A(h)$ and of $A(g)$, respectively.} \label{f: fff}
  \end{figure}

  For convenience, in Figure \ref{f: hyp_elliptic} we depict the axis of $A(gh)$ for  $g$ and $h$ two isometries of the tree $T$ such that $g$ is elliptic and $h$ is hyperbolic,  $A(g)\cap A(h)\neq \emptyset$, and additionally $g$ does not invert a subsegment of $A(h)$ and  $g$ does not pointwise fix $A(h)$.
  
    	\begin{figure}[ht]
  		\begin{center}
  			\begin{tikzpicture}[scale=0.78]
  				\tikzset{near start abs/.style={xshift=1cm}}
  				
  			   \draw[thick] (-3,0)-- (0,0) -- (2,0) -- (6,0);
  			   \draw[->, thick] (2,0) -- (6,0);
  			   \draw[thick] (4,0) -- (4, 5);
  			   \draw[thick] (-1,0) -- (-1,5);
  			   \draw[thick] (4,4) -- (1,4);
  			   \draw[thick] (3, 4) -- (3, 2);
  			   \draw[thick] (2, 4) -- ( 2, 2);
  			    \node at (3.5, 2.1) {\footnotesize $kA(h)$};
  			    \draw[thick] (4,1.5) -- (1,1.5);
  			    \node at (2.5,1.1) {$(ghg)^{-1}A(h)$};
  			    \node[color=blue] at (0.3,0.3) {$A(gh)$};
  			    \node at (6,-0.5) {$A(h)$};
  			    \node at (6,-0.5) {$A(h)$};
  			    \node at (5,4.5) {$g^{-1}A(h)$};
  			   \draw[color=blue, very thick] (-2,-1) -- (-2,0) -- (4,0) -- (4,4) -- (2, 4) -- (2,2);
  			   \draw[color=blue, very thick, middlearrow={>}] (2,4) -- (2,2);
  			  \draw [thick,decorate,decoration={brace,amplitude=7pt},yshift=-1pt] (4,0) -- (-1,0) node [midway,yshift=-0.6cm] {$A(g)$};

  			\end{tikzpicture}
  		\end{center}
  		
  		\caption{The axis of $A(gh)$ (in blue) when $g$ is hyperbolic, $h$ is elliptic, $A(g)\cap A(h)\neq \emptyset$, and $g$ neither fixes $A(h)$ nor inverts an edge in $A(h)$. Here $k$ denotes the element $(hg)^{-2}g^{-1}$.} \label{f: hyp_elliptic}
  	\end{figure}

\medskip
  We will also need the following observation.
  
  \begin{remark}\label{r: axis_of_product}
  	Let $g$ and $h$ be two hyperbolic elements.
  	If $A(g) \cap A(h) = \emptyset$ then $A(gh)$ contains the path connecting $A(g)$ and $A(h)$. Otherwise $A(gh)$ contains $A(g) \cap A(h)$. In both cases $|A(g) \cap A(gh)| = tl(g)$.
  	
  	If $g$ is hyperbolic and $h$ is elliptic then $A(g) \cap A(gh)$ is finite.
  \end{remark}

	 We next provide information regarding the translation length of the product of two isometries.

     \begin{remark}\label{r: alb_remark_2} 
     	Let $g,h$ be two isometries of $T$. Then
     	\begin{enumerate}
     		\item $tl(xy) = tl(x) + tl(y) + 2d(A(x), A(y))$ if $A(x)$ and $A(y)$ do not intersect.
     		\item $tl(xy)= tl(x) + tl(y)$ if $A(x)$ and $A(y)$ intersect coherently.
     		\item $tl(xy) = tl(x) + tl(y) - 2|A(x) \cap A(y)|$ if $0< |A(x)\cap A(y)|\leq \min\{tl(x), tl(y)\}$, and  $A(x), A(y)$ do not intersect coherently.
     	\end{enumerate}
     \end{remark}

 \subsubsection*{Action dynamics for a set of isometries}
 
The action of a group $G$ on $T$ is called \emph{abelian} if $tl(gh)\leq tl(g) + tl(h)$ for all $g,h\in G$. It is called \emph{dihedral} if it is not abelian and $tl(gh)\leq tl(g) + tl(h)$ for all hyperbolic pair of elements $g, h\in G$. Finally, the action is said to be \emph{irreducible} if it is neither abelian nor dihedral. We refer to \cite[Chapter 3]{chiswell2001introduction} for further details and alternative characterisations of these notions.
   
  \begin{lemma}\label{l: lema dihedral action}
  	 Let $G$ be a group generated by two isometries $g,h$ of the tree $T$. Assume that $A(g)\cap A(h)=\emptyset$.
  	Then the action of $G$ on $T$ is dihedral if and only if both $g$ and $h$ are elliptic and both $g^2$ and $h^2$ fix
  	the segment $[p,q]$ between $A(g)$ and $A(h)$. Otherwise, the action is irreducible.
  \end{lemma}
  \begin{proof}
  First note that if the action on $T$ is not irreducible, then both $g,h$ are elliptic. Indeed, if for example $g$ is hyperbolic, then $g$ and $gh$ are two hyperbolic elements whose axes have finite intersection (see Remark \ref{r: axis_of_product}), which forces the action of $\langle g, h\rangle$ on $T$ to be irreducible.
  	
  Moreover, since $[p,q]$ is contained in  $A(gh)$ (by Remark \ref{r: axis_of_product}) but it is not contained in  $A(g)$ nor $A(h)$ by hypothesis, we have that $A(g)$ and $A(h)$ do not contain $A(gh)$. This means the only possibility left is that $g$ and $h$ both act as  on the bi-infinite line $A(gh)$ by inversion. In particular, $g^2$ and $h^2$ fix $[p,q]$.
  	
  	Vice-versa, if $g^2$ and $h^2$ fix $[p,q]$, then it is easy to check that the union of the points in the orbit of $[p,q]$ under the action of $\langle h,g \rangle$ is a bi-infinite line on which both $g$ and $h$ act by a rotation.
  \end{proof}

  \begin{lemma}
  	\label{c: irreducibility of powers}Let $p$ and $q$ be coprime integers and assume we are given a two-generated group $\subg{g,h}$ acting irreducibly
  	on  $T$. Then there is $(s,t)\in\{p,q\}\times \{p,q\}$ such that the subgroup $\subg{g^{s},h^{t}}$ also acts irreducibly on $T$.
  \end{lemma}
  \begin{proof}
  	If both $g$ and $h$ act hyperbolically on $T$ the result is clear. Suppose now that, for instance, $g$ is elliptic. The fact that $p$ and $q$ are coprime implies that $A(g^{p})\cap A(g^{q})=Fix(g)$. If $h$ is hyperbolic, the latter implies that for some $s\in\{p,q\}$ and any $t\in\{p,q\}$ we have that $A(g^{s})\nsupseteq A(h^{t})=A(h)$, which implies that $\subg{g^{s},h^{t}}$ acts irreducibly on $T$. Suppose now $h$ is elliptic as well. If we had $A(g^{s})\cap A(h^{t})\neq\emptyset$ for all $(s,t)\in\{p,q\}^{2}$, then every pair of sets in the family $\{A(g^{q}),A(g^{p}),A(h^{q}),A(h^{p})\}$ has a non-empty intersection. Since they are all convex, Helly's theorem implies that the intersection of all of them must be convex as well, but
  	\begin{align*}
  		A(g^{p})\cap A(g^{q})\cap A(h^{p})\cap A(h^{q})=(A(g^{p})\cap A(g^{q}))\cap(A(h^{p})\cap A(h^{q}))=A(g)\cap A(h)
  	\end{align*}
  	so in that case the action of $\subg{g,h}$ on $T$ would have a fixed vertex: a contradiction.
  \end{proof}

  \newcommand{\lc}[0]{of smaller complexity }

  We end this preliminary section with some miscellaneous results which will be needed during the paper. The first one of them can be proved by using Remark \ref{l: chiswell1}.

  Given a convex subset $A\subset T$ we denote by
  $pr_{A}$ the map that sends each point $x\in T$ to
  the unique closest point $a\in A$. Notice that given another convex set $B\subset T$ either
  $d(A,B)>0$ and $pr_{A}(B)$ is a single point or 
  $pr_{A}(B)=A\cap B$. 
  
  \begin{remark}\label{r: alb_remark_11}
  	Let $g,h,k$ be isometries such that $A(h)\cap A(k)=\emptyset$ and at least one of the following conditions is satisfied:
  	\begin{enumerate}
  	    \item \label{option1} $A(g)\cap A(k)=\emptyset$ and $A(g)$ and $A(h)$ project to the same point of $A(k)$
  	    \item \label{option2} The diameter of $pr_{A(k)}(A(g)\cup A(g))$ is strictly less than $tl(k)$.
  	\end{enumerate}
   Then $A(g^k)\cap A(h) = \emptyset$ and the path between $A(h)$ and $A(g^k)$ intersects $A(k)$.
	A particular case setting is illustrated in Figure \ref{f: f7}.

  	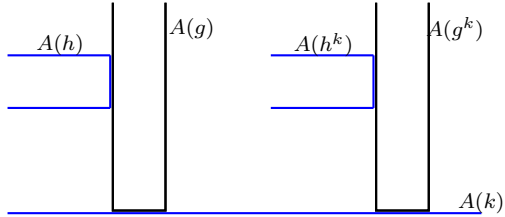
\begin{figure}[ht]
  		\begin{center}
  			\begin{tikzpicture}[scale=0.7]
  				\tikzset{near start abs/.style={xshift=1cm}}
  				
  				\draw[color=blue, thick] (0,0) -- (9,0);
  				
  				\draw[color=blue,thick] (0,2) -- (1.95,2);
  				\draw[color=blue,thick] (1.95,2) -- (1.95,3);
  				\draw[color=blue,thick] (0,3) -- (1.95,3);
  				
  				\draw[color=blue,thick] (5,2) -- (6.95,2);
  				\draw[color=blue,thick] (6.95,2) -- (6.95,3);
  				\draw[color=blue,thick] (5,3) -- (6.95,3);

  				\draw[thick] (2,4) -- (2,0.05) -- (3,0.05) -- (3,4);
  				\draw[thick] (2,4) -- (2,0.05) -- (3,0.05) -- (3,4);
  				
  				\draw[thick] (7,4) -- (7,0.05) -- (8,0.05) -- (8,4);
  				\draw[thick] (7,4) -- (7,0.05) -- (8,0.05) -- (8,4);
  				
  				\node (xx) at (1,3.2) {\color{blue}{\footnotesize $A(h)$}};
  				\node (yy) at (6,3.2) {\color{blue}{\footnotesize $A(h^k)$}};
  				\node (zz) at (9,0.2) {\color{blue}{\footnotesize $A(k)$}};
  				
  				\node (tt) at (3.5,3.5) {\footnotesize $A(g)$};
  				\node (tt) at (8.5,3.5) {\footnotesize $A(g^k)$};
  				
  			\end{tikzpicture}
  		\end{center}
  		\caption{Scenario described in Remark \ref{r: alb_remark_11}, assuming both $g$ and $h$ are hyperbolic.} \label{f: f7}
  	\end{figure}
 
  \end{remark}
  
  \subsubsection*{(Weak) stability}
  
  The next definition plays a key role in this paper. We refer to the introduction and to Section \ref{s: section 4} for context and motivation. 
  
  Given a hyperbolic element $h\in G$, denote the collection of all elements $g\in G$ preserving $A(h)$ set-wise and point-wise by $E(h)$ and $K(h)$, respectively.

  \begin{definition}\label{denf:weakly stable}
  	We say that the hyperbolic element $h$ is \emph{weakly stable} (\emph{weakly} $\lambda$-\emph{stable}) if for all $g\in G$ if $|Ax(h)\cap gAx(h)|> tl(h)$ ( $>\lambda\cdot tl(h)$ ), then $g\in E(h)$. 
  	
  	We say that the hyperbolic element $h$ is ($\lambda$-)\emph{stable} if it is weakly ($\lambda$-)stable and $K(h)=\{1\}$.

  \end{definition}

  The properties defined in the next lemma will only be used at one specific technical point later on.
  
  \begin{lemma}
  	\label{l: forms of stability} Let $G$ be a group acting on a simplicial tree $T$. Consider the following three conditions, parametrized by $\lambda>0$, on a hyperbolic element $h$.
  	\enum{i)}{
  		\item [$FS(\lambda)$] If an element fixes a subsegment of $A(h)$ of length bigger than $\lambda tl(h)$ then it has to fix the whole $A(h)$.
  		\item [$AI(\lambda)$] Given another hyperbolic element $g$ with $tl(g)\leq tl(h)$
  		either $|A(g)\cap A(h)|\leq\lambda tl(h)$ or $g\in E(h)$.
  		\item [$WS(\lambda)$] $h$ is weakly $\lambda$-stable.
  	}
  	The following implications hold:
  	\begin{align*}
  		FS(\lambda)&\Rightarrow AI(\lambda+2), WS(\lambda+2)\\
  		WS(\lambda)&\Rightarrow FS(\lambda)
  	\end{align*}

  \end{lemma}
  \begin{proof}
  	Let us start with $FS(\lambda)\Rightarrow AI(\lambda+2)$.
  	Pick $g$ hyperbolic with $tl(g)\leq tl(h)$ and assume that $|A(g)\cap A(h)|>(\lambda+2) tl(h)$.

  	Recall the well-known fact valid for any Bass-Serre action on a tree, see \cite{chiswell2001introduction}: For any two hyperbolic elements $h,g$ with $|A(g)\cap A(h)|\geq tl(g)+tl(h)$ the commutator $[g,h]$ must fix a subsegment of $A(g)\cap A(h)$ of length $|A(g)\cap A(h)|-tl(g)-tl(h)$
  	and no other points on either $A(g)$ or $A(h)$. This fact implies that the commutator $[g,h]$ fixes a segment of length $(\lambda+2)tl(h)-tl(g)=tl(h)$ and since by assumption $tl(g)\le tl(h)$, the commutator fixes a segment of length at least $\lambda tl(h)$. Since by assumption $h$ satisfies the condition $FS(\lambda)$, it implies that $[g,h]$ fixes the whole $A(h)$, which is a contradiction (since the commutator only fixes points in the intersection $A(g) \cap A(h)$).
  	
  	Let us now see $FS(\lambda)\Rightarrow WS(\lambda+2)$. Suppose that we are given $g$ sending a subsegment $J$ of length at least $(\lambda+2)tl(h)$ of $A(h)$ into $A(h)$. Assume first that $g$ does not revert the orientation of $J$. Then for some value $k\in\Z$ the element $g'=h^{k}g$ displaces $J$ by a distance of at most $tl(h)$ along $A(h)$, which implies that $tl(g')\leq tl(h)$ and
  	$A(g')\cap A(h)\geq\lambda tl(h)$. Properties $FS(\lambda)$ and $AI(\lambda+2)$ then imply that $g'$ preserves $A(h)$ set-wise and, consequently, so does $g$.
  	
  	Assume now that the partial map of $A(h)$ restricting the action of $g$ inverts the orientation. Then, for some value $k\in\Z$ there exists a sub-segment of length at least $\lambda\cdot tl(h)$ such that the two elements $g'=gh^{k}$ and $hg'h$ act on $J'$ as a flip over the same vertex. This implies that $h^{g'}h=g'^{-1}(hg'h)$ must fix said $J'$ point-wise and hence the whole $A(h)$, by $FS(\lambda)$. This means that the hyperbolic element $h^{g'}$ has to preserve $A(h)$ set-wise,
  	which can only occur if $g'$ already does, as $A(h^{g'})=g'^{-1}A(h)$.
  	
  	The second implication is clear.
  \end{proof}

    We will also need the following well-known fact. We include a proof for completeness:
  \begin{lemma}
  	\label{l: one basepoint}Let $G$ be a group acting on a real tree $T$ and assume we are given elements $c_{1},c_{2},\dots, c_{k}$ of $G$ and $K>0$ with the property that for all $1\leq j<j'\leq k$ we have $tl(c_{j}),tl(c_{j'}),tl(c_{j}c_{j'})\leq K$.
  	Then there is a point $*\in T$ such that $d(c_{j}*,*)\leq 2K$ for all $1\leq j\leq k$.
  \end{lemma}
  \begin{proof}
  For each $1\leq j\leq k$ let $E_{j}$ be the $\frac{K}{2}$-neighbourhood of $A(c_{j})$ in $T$. Since
  $tl(c_{j})\leq K$ we have $d(c_{j}*,*)\leq 2K$ for any $*\in E_{j}$. We claim that the intersection $E_{i}\cap E_{j}$ is non-empty for any $1\leq i<j\leq k\}$. Indeed, otherwise $d(A(c_i),A(c_j))> K$, which implies that $tl(c_{i}c_{j})>2K$.
  By Helly's theorem the intersection $\bigcap_{j=1}^{k}E_{j}$ is non-empty and so it suffices to take the point $*$ in this intersection. 
  \end{proof}

  \newcommand{\sgg}[0]{subgraph group } 
  \newcommand{\sggs}[0]{subgraph group } 
  \newcommand{\FF}[0]{\mathcal{F}}
  \newcommand{\HF}[0]{\hat{\mathcal{F}}}
  \newcommand{\hf}[0]{\hat{\mathcal{F}}}
  \newcommand{\hh}[0]{\mathcal{H}}
  \newcommand{\HH}[0]{\mathcal{H}}
  \newcommand{\wpar}[0]{\mathcal{W}^{part}}
  \newcommand{\wdom}[0]{\mathcal{W}^{dom}}
  \newcommand{\U}[0]{\mathcal{U}}
  \newcommand{\V}[0]{\mathcal{V}}
  \newcommand{\W}[0]{\mathcal{W}}
  \newcommand{\ff}[0]{\mathcal{F}}

\section{Weakly small-cancellation and formal solutions} \label{sec:formal solutions}
  
  \newcommand{\traj}[0]{t} 
  \newcommand{\ttraj}[0]{\bar{t}} 
  \newcommand{\Traj}[0]{Traj} 
  \newcommand{\push}[0]{p} 
  \newcommand{\Push}[0]{P} 
  \newcommand{\Pushes}[0]{Pushes}
  \newcommand{\bad}[0]{Br}
  \newcommand{\Hull}[0]{Hull}

  The goal of this section is, given a group  admitting an irreducible action on a tree, to establish a relation between having weak cancellation elements and the triviality of its positive theory. Recall that that we say that the positive theory of a group is trivial if it coincides with the positive theory of a non-abelian free group. This terminology is motivated by the fact that the positive theory of any group contains at least the positive theory of a free non-abelian group.
  
  The triviality of the positive theory will follow from the existence of formal solutions for positive sentences, see Definition \ref{defn:formal solution}. Roughly speaking, given a sentence $\phi$ of the form $\forall x \exists y \Sigma(x,y)=1$, where $\Sigma$ is a system of equations, a formal solution is an algebraic description of the variables $y$'s as words in the variables $x$, i.e. $y=w(x)$, such that the words $\Sigma(x,w(x))$ are trivial in the free group $F(x)$. This formal description of the variables $y$ in function of the variables $x$ provides a ``generic" witness for the validity of the sentence it implies its triviality.
  
 After defining formal solutions, we introduce the other key notion of the section: weak small cancellation tuples. Roughly speaking, a tuple of elements is $N$-small cancellation if the translation length of each element in the tuple is long (at least $N$), all elements of the tuple have similar translation lengths and the minimal tree they span intersects any of its proper translates in a segment of length of the order $\frac{1}{N}$, see Definition \ref{d: small cancellation}. A weak small cancellation tuple allows for longer intersections between the minimal tree $T$ and some of its translates $gT$, but only as long as $g$ acts as the identity on the intersection. We say that a group has weak small cancellation elements if there are tuples of weak $N$-small cancellation elements for arbitrary big $N$.
  
  The first goal of the section, see Theorem \ref{l: small cancellation solutions}, is to show that if a group $G$ acts on a tree satisfies a positive $\forall\exists$-sentence $\phi$ and has weakly small cancellation elements, then there exist formal solutions for the sentence and therefore, in particular, $\phi$ is trivial. As a corollary, we obtain that if a group acting on a tree has weak small cancellation elements, then the only $\forall \exists$-sentences that it satisfies are the trivial ones.
  
  Later in Section \ref{sec:quantifier reduction}, we prove that if a group satisfies a non-trivial positive sentence, then it actually satisfies a non-trivial $\forall\exists$-sentence, see Theorem \ref{l: quantifier reduction}. Combining this fact with the  first main result of this section on the triviality of positive $\forall\exists$-sentences, we deduce that groups acting on trees with weakly small cancellation elements have trivial positive theory.
  
  The second main goal of the present section is to prove a parametric version of the aforementioned result for positive formulas, that is, given a positive $\forall \exists$-formula $\phi(z)$ with free variables $z$, there exist formal solutions (relative to a finite set of diophantine conditions), see Definition \ref{defn:formal solution}, such that if a group $G$ acting on a tree with small cancellation elements satisfies the sentence $\phi(c)$, for some $c\in G^{|z|}$, then the group satisfies one of the diophantine conditions and so the formula $\phi$ admits a formal solution relative to this diophantine condition, see Theorem \ref{l: small_cancellation_parameters}. As a corollary, we deduce that groups acting on trees with small cancellation elements have trivial positive theory, see Corollary \ref{c: iterated formal solutions}. This Corollary will be the key tool to prove the quantifier reduction in Theorem \ref{l: small_cancellation_parameters}.
    
  \bigskip    
  \begin{definition}[Formal solution relative to a diophantine condition]\label{defn:formal solution}
  
  Let $\psi(w)$ be a positive formula in the language of groups \emph{without constants} (all formulas in this paper use no constants ---besides the identity element):
  \begin{align*}
  \psi(w)\equiv\forall x^{1}\exists y^{1}\forall x^{2}\dots\forall x^{m}\exists y^{m}\,\bvee{j=1}{\Sigma_{j}(w,x^{1},y^{1},x^{2}, \dots, x^{m},y^{m})=1}{k},
  \end{align*}
  where $w$, $x^i$, and $y^i$ are tuples of variables for all $1\leq i\leq m$ and $\Sigma_j$ is a conjunction of atomic terms (i.e.\ a conjunction of words), for all $1\leq j \leq k$. Let $\phi(w)$ be a diophantine condition (i.e.\ a positive existential formula) of the form $\exists v \Pi(v,w)=1$ for some system of equations $\Pi(v,w)$.
  
  We next introduce the notion of formal solution of $\psi(w)$ relative to $\phi(w)$. Intuitively, this is a substitution of the $y$'s by words on $v$, $w$, and the $x$'s, so that under this substitution some  word in $\Sigma_j$ (for all $j=1, \dots, k$) is trivial if all words in $\Pi$ are. Formally,
  
  \medskip
  
  By \emph{formal solution} to $\psi(w)$ \emph{relative to} $\phi(w)$ we
  mean a tuple $\alpha=(\alpha^{1},\alpha^{2},\dots,\alpha^{m})$, where
  $\alpha^{l}\in\F(v,w,x^{1},x^{2},\dots, x^l)^{|y^{l}|}$ for all $1 \leq l \leq m$ such that for some $1\leq j\leq k$ all words in the tuple
  \begin{align*}
  \Sigma_{j}(w,x^{1},\alpha^{1}(v,w,x^{1}),x^{2},\alpha^{2}(v,w,x^{1},x^{2}), \dots,\alpha^{m}(v,w,x^{1},x^{2}, \dots, x^{m}))
  \end{align*}
  are in the normal closure of the words of $\Pi(v,w)$ in
  $\F(v,w)\frp\F(x^{1},x^{2},\cdots x^{m})$.

  One can understand formal solutions in terms of homomorphisms between formal groups.
  Consider the groups given by the following presentations:
  \begin{align*}
  G_{\Pi}(v,w)=&\subg{v,w\,|\,\Pi(v,w)=1} \\
  G_{\Sigma_{j}}(w)=&\subg{v,w,x^{1},y^{1},x^{2},y^{2},\dots, x^{m},y^{m}\,|\,\Sigma_{j}(w,x^{1},y^{1},x^{2},y^{2},\dots, x^{m},y^{m})=1}
  \end{align*}
  
  Formal solutions relative to $\phi$ are equivalent to homomorphisms $f:G_{\Sigma}\to G_\Pi *\subg{x^1,x^2, \dots, x^m}$ that restrict to the identity on each of the elements named by a variable from the tuple $w$ or from any of the tuples $x^{1}, \dots, x^m$, such that $f(y^{l})$ is contained in the group generated by $v,w,x^{1},x^{2}, \dots, x^{l}$, for all $1\leq l \leq m$.
   \end{definition}
  
  Notice that the existence of such $\alpha$ implies that $G\models\psi(a)$ for any group $G$ and any tuple $a\in G^{|w|}$ such that $G\models\phi(a)$. The classical notion of formal solution  \cite{merzlyakov1966positive,sela2006diophantineII} is recovered by taking $\Pi = 1$, in which case we will speak of $\alpha$ simply as a \emph{formal solution}.
  
  \newcommand{\br}[0]{\mathcal{P}^{*}}
  \newcommand{\hl}[0]{\hat{A}}
  
  We will denote the collection of all variables in the tuples $x^1, \dots, x^{m}$ simply by $x$. A similar notation will be used for $y^1, \dots, y^m$ and $y$.
  
  We next introduce one of the key concepts of the paper - weak small cancellation elements.

  \begin{definition}
  \label{d: small cancellation section3}
  Fix an action of a group $G$ on a simplicial tree $T$. Consider a tuple of elements $a=(a_{1},\cdots, a_{m})\in G^{m}$.
  Given $N>0$, we say that $a$ is \emph{weakly $N$-small cancellation} (in $T$) if the following holds for some base point $*\in VT$ 
  \enum{(a)}{
  \item \label{SCA} $\dis{a_{i}}>N$ for all $i$,
  \item \label{SCB} $\dis{a_{j}}\leq\frac{N+1}{N}\min_i tl(a_{i})$ for all $i, j$,  and
  \item \label{SCC}
  For all $g\in G,\epsilon\in\{1,-1\}$ and $1\leq i,j\leq m$  the condition $|[*,a_{i}*]\cap g[*,a_{j}*]|\geq\frac{1}{N}\min_i tl(a_{i})$ cannot hold unless $i=j$, in which case $g$ acts like the identity on  $[*,a_{i}*]\cap g[*,a_{i}*]$.
  } 
  Given $c_{1},c_{2},\cdots c_{k}\in G$, we say that a tuple $a$ is weakly $N$-small cancellation \emph{over} $c_{1},c_{2},\cdots, c_{k}$ if $*$ can be chosen in such a way that
  $\min_{i}\dis{a_{i}}>N \cdot \dis{c_{j}}$, for all $1\leq j\leq k$. We say that $(a_{1},a_{2},\dots a_{m})$ is \emph{$N$-small cancellation} (over $c$)
  if all the above applies and we additionally require that $g$ can only be equal to the identity in the last alternative of Item (\ref{SCC}).
  \end{definition}

  We shall prove the following two results.
  
  \begin{thm}[Non-parametric version]
  \label{l: small cancellation solutions}\
  Given an $\forall\exists$ positive sentence $\psi \equiv \forall x \exists y \, \Sigma(x,y)=1$, there is some $N>0$ (depending only on the syntactic length of $\psi$) such that if $G \models \Sigma(a,b)=1$ and $a\in G^{|x|}$ is weakly $N$-small cancellation with respect to the action of $G$ on a tree $T$, then there exists a formal solution $\alpha(x)$ for $\psi$. In particular, $G$ does not satisfy any non-trivial positive $\forall\exists$-sentence.
  \end{thm}
  
  \begin{thm}[Parametric version]\label{l: small_cancellation_parameters} 
  Given an $\forall\exists$ positive formula $\psi(w)=\forall x\exists y\,\Sigma(w,x,y)=1$, there is some $N>0$ and a finite collection $\mathcal{D}$ of diophantine conditions on free variable $w$ with the following properties:
  \enum{(i)}{
  \item For each $\Delta\in \mathcal{D}$ there exists a formal solution $\alpha_{\Delta}$ to $\psi(w)$ relative to $\Delta$.
  \item \label{item witness} For any action of a group $G$ on a tree $T$ and tuples $c\in G^{|w|},a\in G^{|x|},b\in G^{|y|}$ such that $a$ is $N$-small cancellation over $c$ and $\Sigma(c,a,b)=1$ there is a diophantine condition $\exists v\Theta(w,v)$ in $\mathcal{D}$ and $d\subset G$ such that $\Theta(c,d)=1$ holds in $G$ .
  }
  \end{thm}
  
  Notice that, unlike the above statement, Theorem \ref{l: small cancellation solutions} only requires weak small cancellation in its hypotheses.

  The rest of the section, except the last Corollary \ref{c: iterated formal solutions}, is devoted to the proof of these two theorems. 
   The  proof has two main steps. The goal of the first one, Corollary \ref{cl: main Merzlyakov claim}, is the following.  The property of a tuple being small cancellation is a generalization of the classical notion of being a tuple of Merzlyakov words. With this in mind, given an action of a group $G$ on a tree $T$ and tuples $a,b,c$ where $a$ is (weak) small-cancellation, we construct a band complex where the underlying skeleton segments are identified with segments from $T$. Then for each $a_i \in a$ we find a subinterval $\hat{A_i}$ of $A_i=[*,a_i*]$ (where $*$ is a fixed base-point), such that when iteratively sliding it across the band complex, $\hat{A}_i$ spans a simplicial band sub-complex with constant horizontal width which never intersects the subspaces spanned by $C_j=[*,c_j*]$ or $\hat{A}_k$, for all $j$ and all $k\neq i$.
  
  On the second step, we use the interval $\hat{A_i}$ to construct a pair $(\Theta(v,w),\alpha(x,v,w))$, where  $\alpha(x,w,v)$ is a formal solution to $\psi(w)$ relative to the diophantine condition $\exists v\,\,\Theta(v,w)$. Furthermore, we prove the existence of a bound $M$ depending only on the positive sentence $\psi(w)$ such that the size of $\Theta(v,w)$ may be assumed to be bounded by $M$ and so there are finitely many possible such diophantine conditions.
  
  \medskip

  \newcommand{\ww}[0]{S^{*}}
  
  First note that if in the formula $\psi(w)$ we replace $w$ by a tuple $1$ of identity elements, then any diophantine condition $\exists v \Theta(1,v)
  \in \mathcal D$ is trivially satisfied in any group, since it suffices to take $v$ to be a tuple of identity elements, and clearly $\Theta(c,d)=1$. This observation allows us to consider the  statement of Theorem \ref{l: small cancellation solutions} as a particular case of Theorem \ref{l: small_cancellation_parameters} under the weaker condition on weakly $N$-small cancellation and the stronger condition that the tuple $c$ is trivial. Hence we prove Theorems \ref{l: small cancellation solutions} and  \ref{l: small_cancellation_parameters} assuming the weak small cancellation condition and a generic tuple $c$, and  in some key parts which concern only the latter we shall assume non-weak small cancellation.

  We fix the following notation. Given a tuple of elements $v=(v_1, \dots, v_n)$,  we shall denote $v^{\pm 1} = (v_1, \dots, v_n, v_1^{-1}, \dots, v_{n}^{-1})$, and furthermore we define $v_{n + i} =_{def} v_{i}^{-1}$, $i=1, \dots, n$. Under this notation we have $v^{\pm 1}= (v_1, \dots, v_{2n})$. Note that an expression such as $v_j\in v^{\pm 1}$ now refers to an element from $v$ or from $v^{-1}$.   We say that $v_j \in v^{\pm 1}$ is \emph{positive} if $j\leq n$ and otherwise we say $v_j$ is \emph{negative}. 
  
  Throughout this section we let $T$ be a $G$-tree and we. fix a base point $*\in T$.
   For simplicity, we will assume that the system of equations $\Sigma$ consists of a unique equation $S=1$ (see Remark \ref{r: more than one eq} for an explanation of the general case), where
  $$S(w,x,y)=Z_{1}\cdots Z_{|S|},$$ with $Z_{r}\in x^{\pm 1}\cup y^{\pm 1}\cup w^{\pm 1}$.   For $1\leq r \leq |S|$ define $$S_{[r]}= Z_{1}Z_{2}\cdots Z_{r-1}Z_{r}^{\tau_r},$$ where  $\tau_r= 0$ if $Z_{r}$ is positive (so $S_{[r]}= Z_{1}Z_{2}\cdots Z_{r-1}$ in this case), and otherwise $\tau_r= 1$. We set  $S_{[0]}=1$ and further let $v_{r}=Z_{1}Z_2 \dots Z_r(c,a,b)\cdot*$, for $0\leq r\leq |S|$. Of course we have $v_{|S|} = *$.   By $Hull(v_0, \dots, v_{|S|}) = \Hull(\mc{V})$ we denote the convex hull in $T$ of the vertices $v_i$. 
  
  We fix $a,b,c$ to be three tuples of elements such that $\Sigma(c,a,b)=1$ and $a$ is $N$-small cancellation or weakly $N$-small cancellation over $c$, where  $N$ is a sufficiently large positive integer that will be determined during the proof.

  For each \emph{positive} $a_i \in a$, $b_j\in b$ and $c_k\in c$, let $A_{i}$, $B_{j}$, and $C_{k}$ be the segments $[*,a_{i}\cdot*]$, $[*,b_{j}\cdot*]$, and $[*,c_{k}\cdot *]$, respectively.
  For each $1\leq r \leq |S|$ by $Seg(Z_r)$ we denote the segment $A_i, B_j$, or $C_k$ depending on whether $Z_{r}\in x, Z_r\in y,$ or $Z_r \in z$, respectively. 
  We denote  $\mathcal{ABC}$ the collection of all segments $A_i, B_j, C_k$.
  For each $1\leq r\leq |S|$ we denote by $\push_{r}$
  the action of the element $S_{[r]}(c,a,b)$  
  on $T$. By $\push_{r}^{*}$ we denote the restriction of $\push_{r}$ on $Seg(Z_r)$. 
  Notice that the image of $\push_r^{*}$ is the segment $[v_{r-1}, v_{r}]$. We orient $[v_{r-1}, v_{r}]$ from $v_{r-1}$ to $v_r$ if $Z_r$ is positive, and vice-versa if $Z_r$ is negative.  This is motivated by the fact that in the first case we have $v_{r-1} = S_{[r]}(c,a,b)*, v_r =S_r Z_r(c,a,b)*$, and in the second case $v_{r-1} = S_{[r]}(c,a,b)Z_r(c,a,b)^{-1}*, v_r =S_r*$.

  \newcommand{\bbr}[0]{\mathcal{L}}
  
  Next consider the set of segments $\mathcal{ABC}\cup p^*(\mathcal{ABC})$, where $p^*(\mathcal{ABC})$ is defined as  $\{\push_r^*(Seg(Z_r)) \mid r= 1, \dots, |S|\}$. We construct a band complex $\mathcal{B}$ by attaching, for each $1\leq r\leq |S|$, a band $\mathcal{B}_r$  from $Seg(Z_{r})$ to $\push_r(Seg(Z_{r}))$. All bands are attached so that $*$ can be pushed vertically to $S_[r]*$ and the other endpoint $(r=1, \dots, |S|$).  (see Figure \ref{dibu2019-4}).
  Then the maps $\push_{r}^*$ and $\push_r^{*}{}^{-1}$  can be seen as maps that push segments vertically along the band, in the sense that a segment $s$ is pushed vertically along the band into the segment $S_r(c,a,b)(s)$. We call each map $\push_r^*$ a ``push'' (hence its notation).  See Figure \ref{dibu2019-4}. 
  
  \begin{figure}[ht]
  \begin{center}
  \begin{tikzpicture}[scale = 1.15]
  	\tikzset{near start abs/.style={xshift=1cm}}
  	
  	\draw[thick] (4.5,1) -- (6.5,1);
  	\draw[thick] (4.5,1) -- (2.5,1);
  	
  	\draw[thick,color=red, middlearrow={>}]  (2.5,1) -- (0,4);
  	\draw[thick, color=blue, middlearrow={>}] (2.5,1)--(4.5,4);
  	\draw[thick, color=red, middlearrow={>}] (4.5,1) -- (2,4);
  	\draw[ thick, color=blue, middlearrow={>}]  (4.5,1)--(6.5,4);
  	\draw[thick] (-0.5,4)--(1,4) -- (3,4) -- (4,4)--(6,4)-- (7.5,4);
  	  	\filldraw (7.5,4) circle (0pt) node[align=left, above, xshift=0.3cm] {$T$};
  	
  	\filldraw (4.5,1) circle (1pt) node[align=left, below, xshift=0.1cm] {$*$};
  	\filldraw (4.5,4) circle (1pt) node[align=left, above, xshift=0.3cm] {$S_{[s]}a_i*$};
  	\filldraw (0,4) circle (1pt) node[align=left, above, xshift=-0.3cm] {$S_{[r]}a_i*$};
  	\filldraw (6.5,4) circle (1pt) node[align=left, above, xshift=0.2cm] { $S_{[s]}*$};
  	\filldraw (2,4) circle (1pt) node[align=left, above, xshift=-0.3cm] {$S_{[r]}*$};
  	\filldraw (2.5,1) circle (1pt) node[align=left, left, xshift=0cm] {$a_i*$};
  	\filldraw (6.5,1) circle (1pt) node[align=left, below, xshift=-0.2cm] {$b_j*$};

  	\draw [thick,decorate,decoration={brace,amplitude=5pt},xshift=0pt,yshift=-0.1cm]
  	(4.4,1) -- (2.6,1) node [black,midway,xshift=-0.6cm, yshift=-0.4cm]
  	{\footnotesize $A_i=Seg(Z_r)=Seg(Z_s)$};
  	
  	\draw [color=red, thick,decorate,decoration={brace,amplitude=5pt},xshift=-0.2pt,yshift=-0.05cm] (1.9,4) -- (0.2,4);
  	\node[color=red] at (1.6,3.55){\footnotesize $p_r^*(Seg(Z_r))$};
  	\draw [color=blue, thick,decorate,decoration={brace,amplitude=5pt},xshift=0pt,yshift=-0.05cm]
  	(6.4,4) -- (4.6,4) node [midway,xshift=-0.2cm, yshift=-0.4cm] {\footnotesize $p_s^*(Seg(Z_s))$};
  	
  \end{tikzpicture}
  \end{center}
  \caption{Two bands $\mc{B}_r$ and $\mc{B}_s$. Here we have assumed that $Z_r, Z_s \in  \{a_i,a_i^{-1}\}$.}
  \label{dibu2019-4}
  \end{figure}
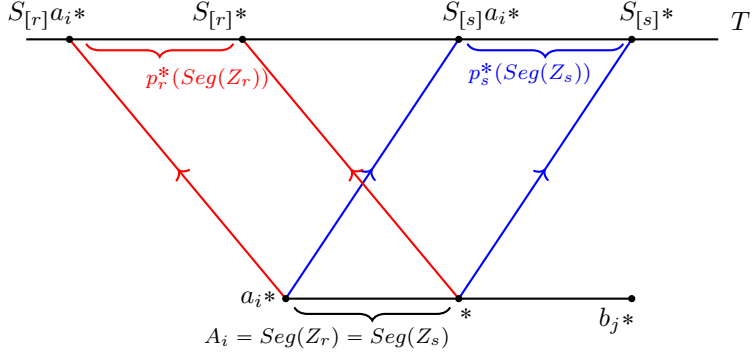

  We are interested in those points of the segments $A_i$ that can be pushed vertically along bands in a way that each time they encounter a segment from $\mathcal{ABC}$, this is of the form $B_j$ for some $b_j\in b$. 
  To formalize this idea we introduce the notion of \emph{trajectory}. Given $x_i \in x$, let $\Traj(x_i)$ be the collection of tuples of the form
  $(\traj_{1}, \traj_{2},\dots, \traj_{2q+1})$ for some $q\geq 0$, where $Z_{\traj_{1}}=x_i$ or $Z_{\traj_1}=x_i^{-1}$, and for any $1\leq k \leq q$ we have $Z_{\traj_{2k}}=Z_{\traj_{2k+1}}$ or $Z_{\traj_{2k}}=Z_{\traj_{2k+1}}^{-1}$, and $Z_{\traj_{2k}}, Z_{\traj_{2k+1}}\in y$. Note that $\Traj(x_i) = \Traj(x_i^{-1})$.
  For such a tuple $\bar{\traj}$ let
  \begin{align}
  \Push_{\bar{\traj}}=(\push_{\traj_{2q+1}}, \push_{\traj_{2q}}^{-1}, \push_{\traj_{2q-1}}, \dots, \push_{\traj_3},\push_{\traj_{2}}^{-1}, \push_{\traj_{1}}),\nonumber \\
  \Push_{\bar{\traj}}^*=(\push^{*}_{\traj_{2q+1}}, (\push_{\traj_{2q}}^*)^{-1}, \push_{\traj_{2q-1}}^*,\dots, \push_{\traj_{3}}^*,(\push^{*}_{\traj_{2}})^{-1},\push^{*}_{\traj_{1}}),\nonumber\\
  S_{[\ttraj]} =  (S_{[\traj_{2q+1}]}) (S_{[\traj_{2q}]})^{-1}(S_{[\traj_{2q-1}]})\dots (S_{[\traj_{3}]}) (S_{[\traj_{2}]})^{-1} (S_{[\traj_{1} ]}),\label{e: 4_oct_19}
  \end{align}
  Each $\Push_{\ttraj}$ and $\Push_{\ttraj}^*$  as above induces a  map and a partial map defined on the band complex, namely  $\push_{\traj_{2q+1}} \circ \push_{\traj_{2q}}^{-1} \circ \dots \circ \push_{\traj_{2}}^{-1}\circ \push_{\traj_{1}}$, and similarly for $\Push_{\ttraj}^*$. These maps are determined by the action of $S_{[\ttraj]}$ on $T$.  Given a point
  $v$ in the band complex, we write $\Push_{\ttraj}(v)$ or $\Push_{\ttraj}^*(v)$ to mean the image of $v$ by $\Push_{\ttraj}$ or $\Push_{\ttraj}^*$ (the latter only whenever defined). Notice that for two different $\ttraj, \ttraj' \in \Traj(x_i)$ we always have  $P_{\ttraj} \neq P_{\ttraj'}$, even though the maps induced by $P_{\ttraj}$ and $P_{\ttraj'}$ may be the same (this is the reason why we defined $P_{\ttraj}$ as a tuple of pushes, rather than as the composition of the pushes itself). Abusing the notation, we treat $P_{\ttraj}$ indistinctly as a tuple of pushes and as its underlying  map. The same observation applies for $P_{\ttraj}^*$.
  
  We let $dom(\Push_{\bar{\traj}}^*)$ be the domain of the partial map induced by $\Push_{\bar{\traj}}^*$, i.e.\ the set of points $v$ in the band complex $\mc{B}$ where $\Push_{\bar{\traj}}^*(v)$ is  defined. Note that  $\Push_{\bar{\traj}}^*$  is  defined only on those segments $s \subseteq A_i$ that can be successively pushed along $\mathcal{B}$ in a way that segments from  $\mathcal{ABC}$ are encountered in the order determined by the trajectory $\bar{\traj}$ (see Figure \ref{dibu2019-3}), i.e.\ as $s$ is pushed by $\Push_{\bar{\traj}}^*$, $s$ encounters the following segments from  $\mathcal{ABC}\cup \push^*(\mathcal{ABC})$ in the following order: $$Seg(Z_{\traj_1}), \push_{\traj_1}(Seg(Z_{\traj_1})), Seg(Z_{\traj_2})= Seg(Z_{\traj_3}), \push_{\traj_3}(Seg(Z_{\traj_3})), \dots$$ $$\dots, Seg(Z_{\traj_{2q}}) = Seg(Z_{\traj_{2q+1}}), \push_{\traj_{2q+1}}(Seg(Z_{\traj_{2q+1}})).$$
  A consequence of this observation is the following
  \begin{remark}\label{r: 4_sept_19}
  A segment $s\subseteq A_i$ is contained in $dom(\Push_{\ttraj})$ if and only if $s$ does not properly contain the preimage by $\Push_{\ttraj}$ of some vertex from $\{v_0, v_1, \dots, v_{|S|}\}$. Intuitively, if and only if $s$ never properly contains one of the vertices $v_0, v_1, \dots, v_{|S|}$  as $s$ is pushed by the maps of $\Push_{\ttraj}$. 
  \end{remark}

  \begin{figure}[ht]
  \begin{center}
  \begin{tikzpicture}[scale=1.2]
  	\tikzset{near start abs/.style={xshift=1cm}}
  	
  	\draw[thick] (0,0) -- (8,0);
  	\draw[thick] (1.5,3) -- (11,3);
  	\draw[thick] (4,0) -- (4,3);

  	\draw[thick,color=blue,middlearrow={>}] plot [smooth, tension=1] coordinates { (1,0) (1.5,1.5) (2.5,3)};
  	
  	\draw[thick,color=blue] plot [smooth, tension=1] coordinates { (4,0) (3.7,0.8) (4,1.5)};
  	
  	\draw[thick,dashed] plot [smooth, tension=1] coordinates { (2.5,0) (3,2) (4,3)};
  	
  	\draw[very thick, color=blue] (2.45,2.95) --(3.95,2.95)--(3.95,1.45);

  	\draw[thick,color=red,middlearrow={>}] plot [smooth, tension=1] coordinates { (7,0) (6.5,1.5) (5.5,3)};
  	
  	\draw[thick,color=red] plot [smooth, tension=1] coordinates { (4,0) (4.2,0.75) (4,1.5)};
  	
  	\draw[thick,dashed] plot [smooth, tension=1] coordinates { (5.5,0) (5,2) (4,3)};
  	
  	\draw[very thick, color=red] (5.55,2.95) --(4.05,2.95)--(4.05,1.45);

  	\draw[thick] plot [smooth, tension=1] coordinates {  (4,0)  (6.2,1.5)  (7,3) };

  	\draw[thick] plot [smooth, tension=1] coordinates { (7,0) (9,1.5) (10,3)};
  	
  	\filldraw (1,0) circle (1pt) node[below] {$a_i*$};
  	\filldraw (2.5,0) circle (1pt);
  	\filldraw (4,0) circle (1pt) node[below] {$*$};
  	\filldraw (5.5,0) circle (1pt);
  	\filldraw (7,0) circle (1pt) node[below] {$b_j*$};

  	\filldraw (2.5,3) circle (1pt);
  	\filldraw (4,3) circle (1pt);
  	\filldraw (5.5,3) circle (1pt);
  	
  	\filldraw (7,3) circle (1pt);
  	\filldraw (10,3) circle (1pt);
  	\filldraw (4,1.5) circle (1pt);
  	
  	\draw [color=blue,thick,decorate,decoration={brace,amplitude=7pt},xshift=-4pt,yshift=-1pt] (4.05,0)-- (1.2,0)  node [midway,yshift=-0.6cm] {$Seg(Z_{t_1})$};

  	\draw [color=red, thick,decorate,decoration={brace,amplitude=7pt},xshift=2.9cm,yshift=-1pt] (4.05,0) -- (1.2,0) node [midway,yshift=-0.6cm] {$Seg(Z_{t_2}) = Seg(Z_{t_3})$};
  	\draw [thick,decorate,decoration={brace,amplitude=7pt},yshift=1pt] (2.6,0) -- (3.9,0) node [midway,yshift=0.4cm, xshift=-0.1cm] {$dom(\Push_{\ttraj}^*)$};
  	
  	\draw [color=blue,thick,decorate,decoration={brace,amplitude=7pt},yshift=-1pt] (3.9,1.5) -- (2.5,2.9) node [midway,xshift=-0.2cm, yshift=-0.7cm] {\footnotesize$\push_{t_1}^*(Seg(Z_{t_1}))$};
  	\draw [color=red,thick,decorate,decoration={brace,amplitude=7pt},yshift=-1pt] (5.5,2.9)--(4.1,1.5) node [midway,xshift=0.7cm, yshift=-0.7cm] {\footnotesize$\push_{t_2}^*(Seg(Z_{t_2}))$};
  	\draw [thick,decorate,decoration={brace,amplitude=7pt},yshift=-1pt] (9.9,3) -- (7.1,3) node [midway,yshift=-0.4cm] {\footnotesize$\push_{t_3}^*(Seg(Z_{t_3}))$};
  	
  \end{tikzpicture}
  \end{center}
  \caption{A depiction of the domain of a hypothetical map $
  \Push_{\ttraj}^*=p_{t_3}^{*}\circ p_{t_2}^*{}^{-1}\circ p_{t_1}^*$. The blue, red, and black bands indicates how $p_{t_1}^*$, $p_{t_2}^*$, and. $p_{t_3}^*$ act, respectively, on $Seg(Z_{t_1})$, $Seg(Z_{t_2})$, and $Seg(Z_{t_3})$.} \label{dibu2019-3}
  \end{figure}
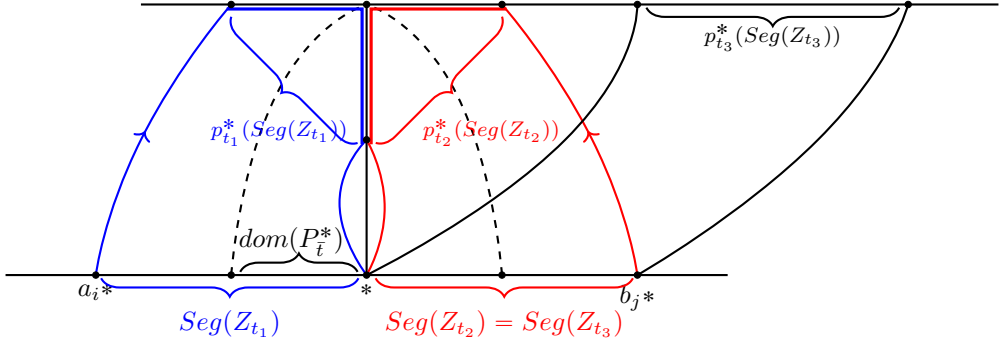

  Denote by $\bad$  the collection of all  points of $\mathcal Hull(\mc{V})\subset T$ that either belong to $\{v_0, \dots, v_{|S|}\}$ or that have degree in $Hull(\mc{V})$ strictly more than $2$ (note this may differ from the degree  in $T$).
  
  \begin{remark}
  \label{r: boundedly many branch points} The size of $\bad$ is uniformly bounded in terms of the equation $S$, i.e.\ there exists an integer $M>0$ such that $|\bad|\leq M$ and $M$ depends only on $|S|$.
  \end{remark}
  \newcommand{\A}[0]{A}
  \newcommand{\B}[0]{B}

  Given $\ttraj\in\Traj(x_i)$, let 
  $A_{i}^{Br}(\ttraj)$ be the collection of subsegments of $A_i$  disjoint from the pre-image of $\bad$ by $\Push_{\ttraj}^*$.

  A key step for proving Theorems \ref{l: small cancellation solutions} and \ref{l: small_cancellation_parameters} is provided by the following lemma, whose proof we postpone for now. This result uses heavily the property of weak $N$-small cancellation.
  
  \begin{lemma}\label{l: main_Merzlyakov_lemma}
  There exists $N_0>0$ such that if $N> N_0$, then for all $x_i \in x$ there exists an  open interval $\hat{A}_i \subseteq A_i$ with the properties $$\hat{A}_i\in \bigcap_{\ttraj\in\Traj(x_i)}A_{i}^{Br}(\ttraj), \quad \text{and} \quad |\hat{A}_i|>\frac{L}{N_0}$$
  where $L$ is the minimum length of the intervals $\{[*, a_i*]\}$ for all $a_i \in a^{\pm 1}$. In particular $\hat{A}_i$ is non-degenerate.
  \end{lemma}

  \begin{cor}
  \label{cl: main Merzlyakov claim} Let $N_0$ and $\hat{A}_i$ ($1\leq i \leq |x|$) be the integer and the intervals given  by  the previous Lemma \ref{l: main_Merzlyakov_lemma}. Assume $N\geq N_0$. Then for any $x_i\in x^{\pm 1}$ and for any $\ttraj \in\Traj(x_i)$ either $\hl_{i}$ is contained in $dom(\Push_{\ttraj}^*)$ or disjoint from it. Moreover, the following conditions hold:
  \enum{(a)}{
  \item \label{cond mux} Let $Z_k \in x^{\pm 1}$ and let $\ttraj\in\Traj(x_i)$. Then one of the following holds:
  \enum{(i)}{
  	\item $Z_k = x_i$ or $Z_k =x_i^{-1}$, $dom(\push_k^*{}^{-1}\circ \Push_{\ttraj}^*)$ contains $\hl_i$, and $\push_k^*{}^{-1}\circ \Push_{\ttraj}^*$ is the identity map when restricted on $\hat{A}_i$. Equivalently, $S_{[k]}^{-1} S_{[\Push_{\ttraj}]}$ acts as the identity element on  $\hat{A}_i$. Furthermore, if $a$ is small cancellation but  not weak small cancellation then  $S_{[k]}^{-1} S_{[\Push_{\ttraj}]}=1$.
  	\item The domain of $\push_k^*{}^{-1}\circ \Push_{\ttraj}^*$ is degenerate.
  }
  \item \label{cond muw} For any $Z_k\in z^{\pm 1}$ and any $\ttraj \in \Traj(x_i)$, the domain of $\push_k^*{}^{-1}\circ \Push_{\ttraj}^*$ is degenerate.
  }
  \end{cor}
  
  \begin{proof}

Let $x_i \in x^{\pm 1}$  and let $\hat{A}_i$ be the  interval given by  Lemma \ref{l: main_Merzlyakov_lemma}.
  
  Let $\ttraj \in \Traj(x_i)$. Then $\hat{A}_i \in A_i^{Br}(\ttraj)$, and so $\hat{A}_i$ is either contained in $dom(P_{\ttraj}^*)$ or disjoint from it, due to Remark \ref{r: 4_sept_19}. Now let $Z_k \in x^{\pm 1}$ with $Z_k = x_j$.
  Assume that
  \begin{equation}\label{e: main_coro_2}
  P_{\ttraj}^*(\hat{A}_i) \cap p_k^*(A_j) \neq \emptyset.
  \end{equation}
  Then since the endpoints of $p_k^*(A_j)$ belong to $Br$, by Lemma \ref{l: main_Merzlyakov_lemma} we have that  $P_{\ttraj}^*(\hat{A}_i) \subseteq p_k^*(A_j)$. From the  small cancellation condition it follows that $j=i$ and that $\push_k^*{}^{-1}\circ \Push_{\ttraj}^* $ acts as the identity on $\hat{A}_i$ (recall that, in general, the image of pushes $\Push_{\ttraj}^*$ and $\push_k^*$ coincides with the image of the actions given  by the  group elements \eqref{e: 4_oct_19}). This proves Item (a) of the corollary, and a similar argument yields Item (b).

  Now suppose $a$ is small cancellation (and not just weak small cancellation). Then since $p_k^*{}^{-1}\circ \Push_{\ttraj}^*$ acts as the identity map on $\hat{A}_i$, by construction we have that the element $S_{[k]}^{-1}S_{[\ttraj]} \in G$ acts as the identity on $\hat{A}_i$. Since $|\hat{A}_i|\geq L/N$ the small cancellation condition now implies that $S_{[k]}^{-1}S_{[\ttraj]}=1$. 
  \end{proof}
  
  We now  prove Theorems \ref{l: small cancellation solutions} and \ref{l: small_cancellation_parameters} assuming  Lemma \ref{l: main_Merzlyakov_lemma} and its Corollary  \ref{cl: main Merzlyakov claim}. Fix $N\geq N_0$, where $N_0$ is the integer given by Lemma \ref{l: main_Merzlyakov_lemma}.
  For any  $x_i \in x^{\pm 1}$ and $y_j \in y^{\pm 1}$, let $\Pushes(x_i, y_j)$ be the set 
  $$\Pushes(x_i, y_j) =_{def} \{\push_k^{-1}\circ \Push_{\ttraj} \mid \ttraj \in \Traj(x_i), \ Z_{k}=y_{j}  \}.$$
  We let $\Pushes(x_i, y_j)^*$ be the set consisting in the maps from $\Pushes(x_i,y_j)$ restricted on $A_i$.
  We claim that the images of $\hat{A}_i$ of two maps $\push_k^*{}^{-1}\circ\Push_{\ttraj}^*$ and  $\push_{k'}^*{}^{-1}\circ\Push_{\ttraj'}^*\in \Pushes(x_i, y_j)^*$  either are the same and have the same orientation, or have degenerate intersection (i.e.\ the intersection is a single point or empty). Indeed, if not then $\push_k^*{}^{-1}\circ \Push_{\ttraj}$ has nondegenerate domain, but the map $$(\push_{k'}^{*}{}^{-1}\circ\Push_{\ttraj'}^{*}{})^{-1} \circ (\push_k^{*}{}^{-1}\circ\Push_{\ttraj}^{*}{}),$$
  which is of the form $\push_{k''}^{*}{}^{-1}\circ \Push_{\ttraj''}^{*}{}$ for some $k''$ such that $Z_{k''}=x_i$ and some $\ttraj'' \in Traj(x_i)$ 
  does not act as the identity on $\hat{A}_i$,  contradicting Item a)  of Corollary \ref{cl: main Merzlyakov claim} (see  Figure \ref{dibu2019-1}).

  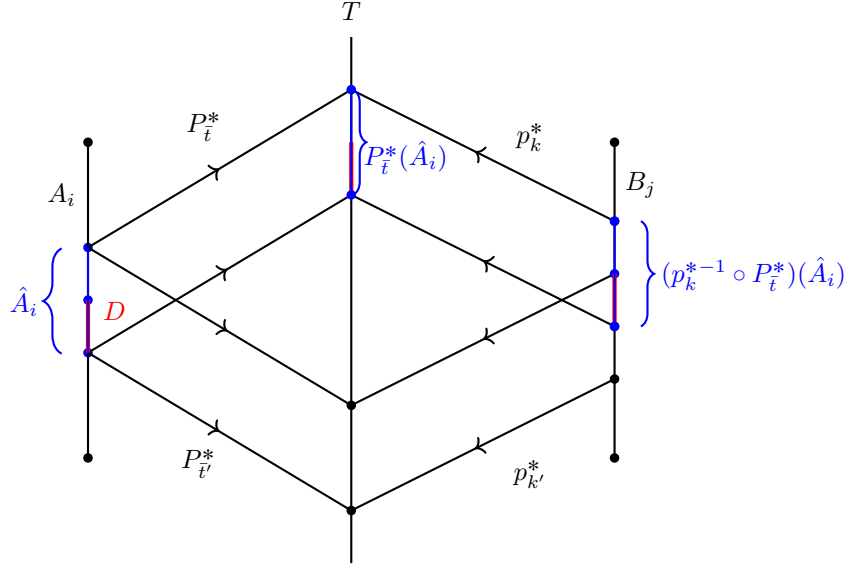
\begin{figure}[ht]
  \begin{center}
  \begin{tikzpicture}[scale =0.7]
  	\tikzset{near star abs/.style={xshift=1cm}}
  	  	\filldraw[thick, color=blue]  (0,4) circle (2pt); 
  	  	\filldraw[thick, color=blue]  (0, 5) circle (2pt); 	  	
  	  	\filldraw[thick]  (0, 8) circle (2pt);
  	  	 	\filldraw[thick]  (5, 3) circle (2pt); 
  	  	\filldraw[thick]  (5, 1) circle (2pt);
  	  	\filldraw[thick]  (10, 3.5) circle (2pt); 
  	  	\filldraw[thick]  (10, 8) circle (2pt); 
  	  	\filldraw[thick, color=blue]  (0, 6) circle (2pt); 
  	\draw[thick, middlearrow={<}] (5,9) -- (10,6.5);
  	  	  	\node at (-0.5,7) {$A_i$};

  	\node at (5,10.5) {$T$};
  	\draw[thick,middlearrow={>}] (0,6) -- (5,9);
  	\draw[thick,middlearrow={<}] (5,3) -- (0,6) ;

  	\draw[thick, middlearrow={>}] (10,5.5) -- (5,3);
  	\draw[thick, middlearrow={>}] (10,4.5) -- (5,7);
  	  	\draw[thick, middlearrow={<}] (5,7) -- (0,4);
  	\draw[thick, middlearrow={>}] (0,4) -- (5,1);
  	\draw[thick, middlearrow={<}] (5,1)
  	-- (10,3.5);
  	\draw[thick] (0,2) -- (0,8) ;
  	\filldraw[thick]  (0, 2) circle (2pt);
  	\draw[thick] (5,0) -- (5,10);
  	\draw[thick] (10,2) -- (10,8);
  	  	\filldraw[thick]  (10, 2) circle (2pt);

  	  	 	\filldraw[thick]  (10, 4.5) circle (2pt); 
  	  	\filldraw[thick, color=blue]  (10, 5.5) circle (2pt); 

  	\draw[ultra thick,color=red] (0,4) -- (0,5);
  	\draw[ultra thick,color=red] (5,7) -- (5,8);
  	\draw[thick, decorate, decoration={brace,amplitude=5pt}, xshift=1pt, color=blue]  (5,9) -- (5,7) node [midway, xshift=0.7cm, yshift=-0.2cm] {$\Push_{\ttraj}^*(\hat{A}_i)$};
  	 	\filldraw[thick, color=blue]  (5, 9) circle (2pt); 
  	  	\filldraw[thick, color=blue]  (5, 7) circle (2pt); 

  	\draw[ultra thick,color=red] (10,4.5) -- (10,5.5);
  	\draw[thick, color=blue] (5,7) -- (5,9);
  	\node at (2.2, 8.3) {$\Push_{\ttraj}^*$};
  	\node at (8.4, 8.1) {$p_{k}^*{}$};
  	\node at (8.4, 1.7) {$p_{k'}^*$};
  	\node at (2.1, 1.9) {$\Push_{\ttraj'}^*{}$};
  	  	\draw [thick,decorate,decoration={brace,amplitude=7pt},yshift=-0.5pt, color=blue] (-0.5,4) -- (-0.5,6) node [midway,xshift=-0.5cm] {$\hat{A}_i$};
  	  	\draw[color=blue, thick] (0,4) -- (0, 6);
  	\draw[thick, decorate, decoration={brace, amplitude = 5pt}, xshift=0.5cm,color=blue] (10, 6.5) -- (10, 4.5) node [midway,xshift=1.5cm] {$(p_k^{*}{}^{-1}\circ \Push_{\ttraj}^*)(\hat{A}_i)$};
  	\node at (10.5, 7.2) {$B_j$};
  	\draw[color=blue, thick] (10, 6.5) -- (10, 4.5);
  	   	\filldraw[thick]  (10, 6.5) circle (2pt);

  	\filldraw[thick, color=blue]  (10, 6.5) circle (2pt); 
  	  	\filldraw[thick, color=blue]  (10, 4.5) circle (2pt); 

  	\node[color=red] at (0.5, 4.8) {$D$};
  \end{tikzpicture}
  \end{center}
  \caption{It cannot happen that two maps of the form $\push_{k'}^{*}{}^{-1}\circ\Push_{\ttraj'}^{*}{}$ and  $\push_k^{*}{}^{-1}\circ\Push_{\ttraj}^{*}{}$ send $\hat{A}_i$ to different segments, as this would contradict Corollary \ref{cl: main Merzlyakov claim}. For the same reason they also cannot send $\hat{A}_i$ to the same segment with opposite orientations. Here $D$ denotes the domain of $p_{k'}^*{}^{-1}\circ\Push_{\ttraj'}^* \circ p_{k}^*{}^{-1}\circ\Push_{\ttraj}^*$, and the red segments denote the image of $D$ along the bands in the picture.} 
  \label{dibu2019-1}
  \end{figure}

  \newcommand{\ha}[0]{\hat{a}}
  \newcommand{\hx}[0]{\hat {x}}
  \newcommand{\hal}[0]{\hal{\alpha}}

  Thus the set \begin{equation}\label{e: 20march19}Im(Pushes^*)= \{im(\rho) \mid \rho \in Pushes(x_i,y_j)^*, x_i \in x^{\pm 1},\ y_j \in y^{\pm 1}\}\end{equation} determines a well-defined subdivision of each segment $B_j$ into subsegments. More precisely, $B_j$ is the following concatenation of subsegments: 
  
  \begin{equation}\label{e: 20march19-2}
  B_j = B_{j,0} \hat{B}_{j,1} B_{j,1} \dots  B_{j,r_j} \hat{B}_{jr_j} B_{j,r_j}
  \end{equation}
  where each $\hat{B}_{j,\ell}$ belongs to $Im(Pushes^*)$ and, conversely, any subsegment of $B_j$ which belongs to $Im(Pushes^*)$  coincides with one of the segments $\hat{B}_{j,\ell}$. The other segments $B_{j,\ell}$ are possibly degenerate, i.e.\ a single point $(\ell = 0,\dots, r_j)$.
  
 We claim that we can assume  that the action of $G$ on $T$ is cocompact. Indeed,  the small  cancellation property of a tuple $a$ is preserved after collapsing all edges of $T$ that do not lie in the union of the translates of the minimal tree of the group $\subg{a}$. The result of this collapse is a tree with finitely many orbits of edges.
 
 In views of the paragraph above, let $\delta$ be the diameter of $T/G$. For each $\hat{A}_i$ let $\hat{A}_i^e$ be a segment in $T$ containing $\hat{A}_i$, having its endpoints in the orbit of $*$, and of minimal length with these properties. Then the endpoints of $\hat{A}_i^e$ are at distance at most $\delta$ from the endpoints of $\hat{A}_i$. Let $a_i', a_i''$ and $\hat{a}_i$ be elements such that  $a_i'$ sends $*$ to the starting point of $\hat{A}_i^e$, $\hat{a}_i$ sends the starting point of $\hat{A}_i^e$ to the endpoint of  $\hat{A}_i^e$, and  $a_i''$ sends the endpoint of $\hat{A}_i^e$ to $a_i*$. Then $a_i^{-1} a_i'' \hat{a}_i a_i'$ stabilizes $*$, and multiplying $a_i'$ on the right by an appropriate element (i.e.\ $(a_i^{-1}a_i''\hat{a}_i a_i')^{-1}$) in the stabilizer of $*$  (and keeping the notation $a_i'$) we obtain $a_i = a_i'' \hat{a}_i a_i'$. For each segment $A_i$, let $A_{i}', A_{i}''$ be subintervals of $A_i$ such that $A_i$ is the concatenation of $A_{i}'$ and $\hl_i$ and $A_{i}''$. Furthermore, let $A_i^{e}{}'$, $A_i^{e}{}''$ be paths connecting $*$ (the starting point of $A_i$) with $\hat{A}_i^e$ and the ending point of $\hat{A}_i^e$ with $a_i*$ (the ending point of $A_i$). See Figure \ref{f: B_jle}.
  
  Let $\rho \in Im(\Pushes)$ and  $x_i \in x$ be such that $\rho(\hat{A}_i)=\hat{B}_{j, \ell}$ and $\rho= p_k^{-1}\circ P_{\ttraj}$ for some $\ttraj \in \Traj(x_i, y_j)$ and $Z_k \in y^{\pm 1}$, $1\leq k \leq |S|$. We define
  $$
  \hat{B}_{j,\ell}^e =_{def} (p_k^{-1}\circ P_{\ttraj})(\hat{A}_i^{e}).
  $$
  Notice that neither $\hat{A}_i^{e}$ nor  $\hat{B}_{j,\ell}^e$ is  necessarily contained in $A_i$ or in $B_j$, respectively.

  Observe also that $\hat{B}_{j, \ell}^e$ contains $\hat{B}_{j,\ell}$ and the endpoints of $\hat{B}_{j, \ell}^e$ are at distance at most $\delta$ from the endpoints of $\hat{B}_{j,\ell}$.

  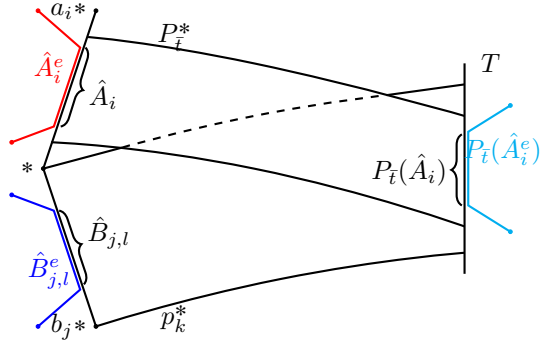
\begin{figure}[ht]
  \begin{center}
  \begin{tikzpicture}[scale=0.7]
  	\tikzset{near star abs/.style={xshift=1cm}}
  	
  	\draw[thick] (8,1) -- (8,5);
  	\node at (8.5,5) {$T$};
  	\draw[thick] (1,0) -- (0,3) -- (1,6);
  	
  	\draw[thick] plot [smooth, tension=1] coordinates {(1,0) (4.5,0.9) (8,1.4)};
  	
  	\draw[thick] (0,3) -- (1.5,3.4);
  	\draw[thick] (6.55,4.4) -- (8,4.6);
  	\draw[thick,dashed] plot [smooth, tension=1] coordinates {(1.5,3.4) (4,4) (6.55,4.4)};

  	\draw[thick] plot [smooth, tension=1] coordinates {(0.15,3.5) (4,3) (8,1.9)};
  	\draw[thick] plot [smooth, tension=1] coordinates {(0.85,5.5)  (4,5) (8,4)};
  	\node at (2.5,5.5) {$\Push_{\ttraj}^*$};
  	\node at (2.5,0.15) {$\push_{k}^*$};

  	\filldraw (1,0) circle (1pt) node [left] {$b_j*$};
  	\filldraw (0,3) circle (1pt) node [left] {$*$};
  	\filldraw (1,6) circle (1pt) node [left] {$a_i*$};
  	
  	\draw[thick, xshift=-0.1cm, color=blue] (0,0) -- (0.8,0.7) -- (0.3,2.2) -- (-0.5,2.5) node[xshift=0.5cm,yshift=-1cm] {$\hat{B}_{j,l}^e$};
  	\filldraw[color=blue, xshift=-0.1cm] (0,0) circle (1pt);
  	\filldraw[color=blue, xshift=-0.1cm] (-0.5,2.5) circle (1pt);
  	
  	\draw[thick, xshift=-0.1cm, color=red] (0,6) -- (0.8,5.3) -- (0.3,3.8) -- (-0.5,3.5) node[xshift=0.5cm,yshift=1cm] {$\hat{A}_{i}^e$};
  	\filldraw[color=red, xshift=-0.1cm] (0,6) circle (1pt);
  	\filldraw[color=red, xshift=-0.1cm] (-0.5,3.5) circle (1pt);
  	
  	\draw[thick, xshift=0.07cm,color=cyan] (8.8,4.2) -- (8,3.7) -- (8,2.3) -- (8.8,1.8) node[xshift=-0.1cm,yshift=1.1cm] {$\Push_{\ttraj}(\hat{A}_{i}^e)$};
  	\filldraw[color=cyan, xshift=0.07cm] (8.8,4.2) circle (1pt);
  	\filldraw[color=cyan, xshift=0.07cm] (8.8,1.8) circle (1pt);
  	
  	\draw [thick,decorate,decoration={brace,amplitude=4pt},xshift=0,yshift=0cm] (0.3,2.2) -- (0.77,0.8) node [black,midway,xshift=0.5cm, yshift=0.2cm] {$\hat{B}_{j,l}$};
  	
  	\draw [thick,decorate,decoration={brace,amplitude=4pt},xshift=0,yshift=0cm] ((0.8,5.3) -- (0.33,3.8) node [black,midway,xshift=0.4cm, yshift=-0.1cm] {$\hat{A}_{i}$};
  	
  	\draw [thick,decorate,decoration={brace,amplitude=4pt},xshift=-0.05cm,yshift=0cm] (8,2.3) -- (8,3.65) node [black,midway,xshift=-0.7cm, yshift=0cm] {$\Push_{\ttraj}(\hat{A}_{i})$};
  	
  \end{tikzpicture}
  \end{center}
  \caption{Depiction of  segments $\hat{A}_i^e$ and $\hat{B}_{j,\ell}^e$.} \label{f: B_jle}
  \end{figure}

  Let
  $$
  b_{j,0}, \hat{b}_{j,1}, b_{j,1}, \hat{b}_{j, 2}, \dots, \hat{b}_{j,r_j}, b_{j,r_j}
  $$
  be elements of $G$ such that $b_{j,0}$ takes the vertex $*$ to the starting point of $\hat{B}_{j,1}^e$, $\hat{b}_{j,1}$ takes the starting point of $\hat{B}_{j,1}^e$ to its endpoint, and so on until reaching the endpoint of $B_{j,r_j}$. We also let $$B_{j, 0}^{e}{}, \dots, B_{j, r_j}^e{}$$ be paths connecting $*$ with the starting point of $\hat{B}_{j,1}^e$, from the endpoint of $\hat{B}_{j,1}^e$ to the starting point of $\hat{B}_{j,2}^e$, and so on.  As we did with the elements $a_i', \hat{a}_i', a_i''$, we can assume without loss of generality  that
  \begin{equation}\label{e: 21march19}
  b_{j}=b_{j,r_j}\hat{b}_{j,r_{j}} b_{j,r_{j}}\dots b_{j,2}\hat{b}_{j,1}b_{j,1}.
  \end{equation}
  This is achieved by multiplying, if needed, the element  $b_{j,1}$  by an appropriate element in the stabilizer of $*$ in a similar way as we did before.
  
  Furthermore, we claim that we can assume without loss of generality that
  \begin{equation}\label{e: hat_b_=_hat_a}
  \hat{b}_{j,\ell} = \hat{a}_{i_\ell}^{ s_\ell}, \quad \ell = 1,\dots, r_j,
  \end{equation}
  for some indices $1\leq i_\ell\leq |x^{\pm 1}|$, and some conjugating elements $s_\ell\in G$.  Indeed, we have $\hat{B}_{j,\ell}^e = S_{\ttraj}\hat{A}_i^e$  for some $i$ such that $x_i \in x^{\pm 1}$, and some $\ttraj\in \Pushes(x_i, y_j)$. Then the element $S_{[\ttraj]}\hat{a}_{i}S_{[\ttraj]}^{-1}$ takes the starting point of  $\hat{B}_{j,\ell}^e$ to its endpoint, or vice-versa depending on whether $x_i$ is positive or negative, and so when choosing $\hat{b}_{j,\ell}$ in \eqref{e: 21march19}  we can let take this conjugate or its inverse . 
  
  \medskip

  To each $a_{i}'$, $\hat{a}_i$, and $a_{i}''$ with $1\leq i\leq |x|$, we associate new variables $x_{i}'$, $\hat{x}_i$, and $x_{i}''$. 
  Similarly, to each $b_{j, \ell}$   we associate a new variable $y_{j,\ell}$. 	Let $x'{}^{\pm 1}, \hat{x}{}^{\pm 1}, x''{}^{\pm 1}$ and $y'{}^{\pm 1}$ denote the tuples consisting of all the $x_i', \hat{x}_i, x_i''$ and all $y_{j,\ell}$, and their inverses, respectively. Denote  $v^{\pm 1} = (x'{}^{\pm 1}, x''{}^{\pm 1}, y'{}^{\pm 1})$. Let
  \begin{align*}
  \alpha_{j}(v, \hat{x})=y_{j,r_{j+1}}\hat{x}_{i_{r_{j}}}\cdots y_{j,2}\hat{x}_{i_1}y_{j,1}.
  \end{align*}
   This word matches the decomposition in \eqref{e: 21march19} and \eqref{e: hat_b_=_hat_a}.
  
  Let $\alpha$ be the tuple $(\alpha_1, \dots, \alpha_{|y|})$. Similarly, let $a'{}^{\pm 1}, \hat{a}{}^{\pm 1}, a''{}^{\pm 1}$ be tuples consisting of all the  $a_i', \hat{a}_i', a_i''$ and their inverses, respectively ($1\leq i\leq |x|$). For each $j$ let $b_j'$ be the tuple $(b_{j,0}s_1^{-1}, s_1b_{j,1}s_2^{-1}, \dots, s_{r_j}b_{j,r_{j+1}})$, and let $b'$ be the tuple consisting on all the tuples $b_j'$ ($1\leq j\leq |b|$).

  Consider the word
  \begin{align*}
  \ww(w,v, \hat{x})= \ww(w,&x',x'',y', \hat{x})=_{def}\\ &S(w,x_{1}''\hx_{1}x_{1}',x_{2}''\hx_{2}x_{2}'',\dots,x_{|x^{\pm 1}|}''\hx_{|x^{\pm 1}|}x_{|x^{\pm 1}|}',\alpha(v,\hat{x})),
  \end{align*}
  which is nothing but a ``refinement'' of the word $S(w, x,y)$ where each variable $y_j$ has been replaced by $\alpha_j$ and each variable $x_i$ has been replaced by $x_i'' \hat{x}_i x_i'$.  It follows from \eqref{e: 21march19}, from \eqref{e: hat_b_=_hat_a}, and from the equality $a_i=a_i''\ha_{i}a_i'$, that  $$\ww(c,a',a'',b', \hat{a})=S(c,a,b)=1.$$
  Write
  \begin{align*}
  \ww=	W_{1}W_{2}\cdots W_{|S^*|}
  \end{align*}
  where, for all $k=1, \dots, |S^*|$, $W_{k}$ is one of the letters appearing in the tuples
  $w^{\pm 1},\hat{x}^{\pm 1}, v^{\pm 1}$.
  As before, given $1\leq k\leq |S^*|$, we let  $\ww_{[k]}=W_{1}W_{2}\cdots W_{k}^{\tau_k}$, where $\tau_k=0$ if $W_k$ is positive, and $\tau_k=1$ otherwise. We set $\ww_{[0]}=1$. For each $1\leq k \leq |S^*|$ we denote by $Seg(W_k)$ and $Seg(W_k)^e$ one of the segments
  $$
  A_{i}', \hat{A}_i, A_{i}'', B_{j,\ell}, \hat{B}_{j,\ell},  C_t
  $$
  and
  $$
  A_{i}^e{}', \hat{A}_i^e, A_{i}^e{}'', B_{j,\ell}^e, \hat{B}_{j,\ell}^e,  C_t,
  $$
  respectively, depending on whether $W_k$ is  $x_i', \hat{x}_i, x_i'', y_{j,\ell}, \hat{y}_{j,\ell}$, or $w_t$, respectively ($i=1,\dots,|x^{\pm 1}|$; $j=1,\dots, |y^{\pm 1}|$; $t=1,\dots, |w^{\pm 1}|$; $\ell = 1, \dots, r_j$).   Note that for each $Seg(W_k)$ there exists a unique index $1\leq k' \leq |S|$ such that $Seg(W_k) \subseteq Seg(Z_{k'})$. We denote such index by $k_S$ (indicating that it is an index referring to a letter in $S$). 
  
  \newcommand{\hpi}[0]{\hat{\Pi}}
  
  \medskip
  
  Take a copy $S^{1}$ of the circle and divide it into $|S^*|$ edges $e_1, \dots, e_{|S^*|}$. Orient all edges counterclockwise and label them  so that the word $S^*$ can be read counterclockwise when starting at some fixed vertex $*_{0}$.
  Given an edge $Seg_{\Gamma}(W_k)$ of $\Gamma$ we define 
  \begin{align}\label{e: 4_oct_19_3}
  &\phi(Seg_\Gamma(W_k)) =_{def} \push_{k_S}(Seg(W_k)) = S_{[k_S]}(c,b,a) Seg(W_k)
  \end{align}

  Given  edges $e, e'$ in $\Gamma$ labeled with variables from the tuple $\hx^{\pm 1}$, let $e\sim e'$ if $\phi(e)=\phi(e')$. We have that $\sim$ is an equivalence relation, and by Item (a) in Corollary \ref{cl: main Merzlyakov claim} and by construction, if $e \sim e'$ then the labels of $e$ and $e'$ coincide and so the labeling is well-defined on the equivalence class.  Moreover, this same corollary yields that if $e\sim e'$   then the orientation of  $\phi(e)$ inherited  from $e$ is the same as the  orientation of  $\phi(e')$ inherited from $e'$.
  
  The following observation follows from the construction of $\Gamma$, the map $\phi$, and the segments $A_i^{e}{}', A_i^{e}{}'', \hat{A}_i^e, B_{j,\ell}^e, \hat{B}_{j,\ell}^e$.
  \begin{remark}\label{r: 4_oct_19}
  The image of $\Gamma$ by $\phi$ is a closed cycle in $T$.
  \end{remark}

  Let $\bar{\Gamma}$ be the 1-complex obtained as a quotient of $\Gamma$ by identifying edges $e$ and $e'$ such that $e\sim e'$. The identification is done  in such a way that 1) if the label of $e$ is positive and the label of $e'$ is negative (or vice-versa), then the initial point of $e$ is identified with the final point of $e$, and vice-versa; 2) if $e$ and $e'$ both have positive (or negative) label, then the initial point of $e$ is identified with the initial point of $e'$, and similarly for the final points of $e$ and $e'$.
  Since the label is well-defined in the equivalence class, $\bar \Gamma$ inherits from $\Gamma$ a labeling on its edges. Denote by $\pi:\Gamma\to \bar \Gamma$ the quotient map. Then there exists a natural map $\bar{\phi}  : \bar \Gamma \to T$  for which $\phi = \pi \circ \bar{\phi}$.
  For shortness we shall denote $\pi(e)$ simply by  by $\bar{e}$, for $e$ an edge of $\Gamma$.

  Let $\mathcal{E}_{x}$ be the collection of all the edges in $\bar{\Gamma}$ with labels in $\hx$. 
  \newcommand{\interior}[1]{%
  {\kern0pt#1}^{\mathrm{o}}%
  }
  Let $C$ be the complement in $\bar{\Gamma}$ of the interior of all the edges in $\mathcal{E}_{x}$. Let $C_1, \dots, C_{k}$ be all connected components of $C$. Observe that each $C_i$ has labels in $(x' \cup x''\cup y' \cup w)$ (without inverses).
  \begin{obs}
  \label{o: collapse}
 
  The fundamental group of $\bar{\Gamma}$ is in the normal closure of the fundamental groups of $C_1, \dots, C_{k}$.
  \end{obs}
  \begin{proof}
  This is equivalent to saying that $\bar{\Gamma}\setminus\interior{\bar e}$ is disconnected for any $\bar{e}\in\mathcal{E}_{x}$, where by $\interior{\bar e}$ we mean the interior of $\bar{e}$.
  Notice that the image of the quotient map $\pi:\Gamma\to\bar{\Gamma}$ is a closed surjective cycle in $\bar{\Gamma}$ and that, by Remark \ref{r: 4_oct_19}, the images of $\phi^e = \bar{\phi}^e\circ\pi$ and $ \phi=\bar{\phi} \circ \pi$ are in turn a cycle in the tree $T$.
  
  Suppose $\bar{e}$ does not disconnect $\bar{\Gamma}$.  Then there is some path $\bar{\gamma}$ in $\bar{\Gamma}$ such that $\bar{\gamma}$ starts and ends at the opposite ends of $\bar{e}$ but does not intersect $\bar{e}^\circ$.   Let $\hx_i$ be the label of $\bar{e}$ and suppose $e$ is the $k$-th edge of $\Gamma$, starting at $*_0$.  
  The fact that the image by $\bar{\phi}$ of $\bar{\Gamma}$ is a cycle in the tree $T$ implies that $\bar{\phi}(e) \subseteq \bar{\phi}(\bar{\gamma})$.  Hence there exists and edge $\bar{e}'$ of $\bar{\gamma}$ such that  $\bar{\phi}(\bar{e}) \cap \bar{\phi}(\bar{e}')$ is neither empty nor a single point. Suppose $e'$ has label $W_{r}$. Now $$\bar{\phi}(\bar e) = \push_{k_S}(Seg(W_k)) = \push_{k_S}(\hat{A}_i), \quad  \bar{\phi}(\bar{e}') = \push_{r_S}(Seg(W_r)).$$ Thus
  $$
  0 < |\push_{k_S}(\hat{A}_i) \cap \push_{r_S}(Seg(W_r))| < |\push_{k_S}(\hat{A}_i) \cap \push_{r_S}(Seg(Z_{r}))|.
  $$
  Hence the domain of $ p_{k_S}^*{}^{-1} \circ p_{r_S}^*$  is non-empty.
  By Corollary \ref{cl: main Merzlyakov claim} it follows that either $Z_{k_S} = Z_{r_S}^{\pm 1} \in x^{\pm 1}$ or $Z_{r_S}\in y^{\pm 1}$. By construction of the intervals $B_{j,\ell}, \hat{B}_{j,\ell}$ it follows that in all cases we have $W_{r} = W_k = \hat{x}_i$ or $W_r = W_k^{-1}=\hat{x_i}^{-1}$. 
  But this means that $\bar{e}$ and $\bar{e}'$ were the same segment of $\bar \Gamma$ to begin with, a contradiction with the construction of $\bar{\gamma}$.
  \end{proof}
  
  Since we have $a_i = a_i'' \hat{a}_i a_i'$, we can replace $\hat{x}_i$ in the words from $\alpha(v, \hat{x}_i)$ by $x_{i}''{}^{-1} x_i x_i'{}^{-1}$, for all $1\leq i\leq |x|$.  This yields  a new tuple of words,  denoted $\alpha(v, x)$, on variables $(v, x)$ (recall that $x_i', x_i''$ are part of $v$, for all $i$).
  Let  $\Theta(v, w)=1$ be the system of equations obtained by equating the labels of $C_1, \dots, C_{k}$ to $1$ (all these have labels in $v, w$).  It follows from the above observation that $\alpha(v, x)$ is a formal solution to $\psi$  relative to $\exists v \Theta(v,w)=1$ ($\psi$ is the positive $\forall\exists$-formula in the statement of Theorems \ref{l: small cancellation solutions} and \ref{l: small_cancellation_parameters}). This completes the proof of Theorem \ref{l: small cancellation solutions}  and of  Item (i) in Theorem \ref{l: small_cancellation_parameters}. 
  
  \medskip
  From now on we assume that $a$ is (non-weak) $N$-small cancellation over $c$. Next we find an integer $M$ depending only on the equation $S$ such that $|\Theta'| \leq M$ where $\Theta'$ is a system of equations equivalent to $\Theta$, and $|\Theta'|$ denotes the sum of the lengths of the words appearing in $\Theta'$. From this it follows that the set of diophantine conditions in the statement of Theorem \ref{l: small_cancellation_parameters} can be taken to be all systems of equations of synctactic length at most $M$.
  
  We let  $Seg_\Gamma(Z_r)$, $r=1,\dots, |S|$  be   paths in $\Gamma$ so that for each equality (with concatenation operation) $Seg(Z_r) =  Seg(W_{r_1}) \dots Seg(W_{r_1+r_2})$  ($1 \leq r_1 < r_2 \leq |S^*|$) we have an analogous equality $Seg_\Gamma(Z_r)= Seg_\Gamma(W_{r_1}) \dots Seg_\Gamma(W_{r_1+r_2})$.
  
  We say that $C_i$ is \emph{singular} if
  \begin{enumerate}
  \item \label{it:1 singular} it is adjacent to at least three  edges from $\mathcal{E}_x$, or
  \item \label{it:2 singular} it is adjacent to exactly one edge from $\mathcal{E}_x$, or
  \item \label{it:3 singular}it contains the initial or final vertex of  $\pi(Seg_\Gamma(Z_r))$ for some $1\leq r \leq |S|$.
  \end{enumerate}
  \begin{lemma}\label{l: bounded_num_singular_C_i}
  The number $s$ of singular components $C_i$ is  bounded in terms of $S$, i.e.\ there exists an integer $M$ depending only on $S$ such that $s\leq M$.
  
  Furthermore, each component $C_i$ (not necessarily singular) has length  bounded in terms of $|S|$, and if $C_i$ is not singular then it has length at most $2$.
  \end{lemma}
  \begin{proof}
  For all $i$, $\bar{\phi}(C_i)$ is a closed path in $Im(\bar{\phi}) = \bar{\phi}(\bar{\Gamma})$, which is the convex hull $Hull$ of the vertices $v_0=*, v_1=S_1(c,a,b)*, \dots, v_{|S|-1} = S_{|S|-1}(c,a,b)*, v_{|S|}=S(c,a,b)*=v_0$.
  
  Now if $C_i$ satisfies condition \ref{it:1 singular} or condition \ref{it:2 singular} above, then $\bar{\phi}(C_i)$ contains a vertex of degree at least $3$ in $Hull$, respectively (here we mean degree in the convex hull --- the degree in $T$ may be larger). On the other hand, if $C_i$ satisfies condition \ref{it:3 singular}, then $\bar{\phi}(C_i)$ contains one of the vertices $v_i$. Thus in all three cases $\bar{\phi}(C_i)$ contains a vertex from $Br$. By Remark \ref{r: boundedly many branch points}, $|Br|$ is bounded in terms of $S$.
  The number of times that $\bar{\phi}^{\mc{CS}}(\bar{\Gamma})$ traverses the same vertex is  bounded by $|S|$, and so the number of singular components $C_i$ is at most $|S||Br|$.%

  We now prove the second statement of the lemma. By construction of the cycles $C_i$, each vertex $v$ of $C_i$ is either  the endpoint or initial point of $\pi(Seg_\Gamma(Z_r))$  for some $1\leq r \leq |S|$, or it is an endpoint or the initial point of an edge from $\mc{E}_x$, in which case $v$ has degree at least $3$ and so does $\bar{\phi}(v)$ (because all vertices in $\Gamma$ are either endpoints or initial points of an edge with label in $\hat{x}^{\pm 1}$ or they are the endpoint/initial-point of a segment  $Seg(Z_k)$ for some $k$). Hence in both cases $\bar{\phi}(v)$ is a vertex of $Br$. Hence the length of $C_i$ is at most $|Br|$ and so it is bounded in terms of $S$ due to Remark \ref{r: boundedly many branch points}. If $C_i$ is  not singular then, by definition and the previous observation, all the vertices  of $C_i$ are necessarily endpoints/initial-points) of some vertex from $\mc{E}_x$, which implies that $C_i$ has at most two vertices. 
  \end{proof}

  Let $C_{i_1}, \dots, C_{i_\ell}$ be all non-singular components among $C_1, \dots, C_{k}$. By Lemma \ref{l: bounded_num_singular_C_i}, $C_{i_j}$ consists of either one or two edges and so the corresponding equation $label(C_{i_j})=1$ of $\Theta(v,w)=1$ has the form $X=1$ or $X = Y$ for some variables $X,Y \in v^{\pm 1} \cup w^{\pm 1}$. Thus one can replace each occurrence of $X$ in $\Theta$ by $1$ or by $Y$ and delete the equation $X=1$ or $X = Y$, respectively. After making this replacement for each $C_{i_j}$ we obtain a new system of equations $\Theta'(v,w)= 1$ equivalent to $\Theta(v,w)$, with a uniformly bounded (in terms of $S$) number of equations due to Lemma \ref{l: bounded_num_singular_C_i} ---there are as many equations as singular components $C_i$. Moreover, each such equation has length $|C_i|$, which is bounded in terms of $S$ by Lemma \ref{l: bounded_num_singular_C_i}. Hence $|\Theta'|$ is uniformly bounded in terms of $S$, as needed.
  
  \medskip
  
  Only Condition (ii) in the statement of Theorem \ref{l: small_cancellation_parameters} remains to be proved. For this, it suffices to show that each word from $\Theta(v, w)$ evaluates to $1$ on the tuple $d =_{def} (c,a',a'',b',\hat{a})$. 
  Recall that $S_{[k]}^*= W_1 \dots W_k^{\tau}$ where $\tau = 0$ if $W_k$ is positive and $\tau=1$ if $W_k$ is negative.

  \begin{lemma}
  \label{r: trivial} Suppose the tuple $a$ satisfies the (non-weak) small cancellation property with respect to $c$. Let $e_k, e_{r}$ be two distinct edges of $\Gamma$ such that $e_{k}\sim e_{r}$. Then $\ww_{[k]}{}^{-1}(d)\ww_{[r]}(d)=1$.
  \end{lemma}
  \begin{proof}
  
  Under the assumptions of the lemma  and due to the definition of $\phi$ (equation \eqref{e: 4_oct_19_3}) we have \begin{equation}\label{e: 4_oct_19_2}\phi(e_k) = S_{[k_S]}(c,b,a) Seg(W_k)^e =  S_{[r_S]}(c,b,a) Seg(W_r)^e = \phi(e_{r})\end{equation} as non-oriented segments of $T$. Let $x_i \in x$ be such that
  $Seg(W_k)^e =   \hat{A}_{i}^e.$ Then also $Seg(W_r)^e = \hat{A}_i^e$, and so from \eqref{e: 4_oct_19_2} we obtain that $S_{[k_S]}^{-1}S_{[r_S]}(c, a,b)$ fixes $\hat{A}_i^e$, and in particular it fixes  $\hat{A}_i$. The (non-weak) small cancellation condition now implies that  such element is the identity since $|\hat{A}_i|\geq \frac{L}{N}$.  Moreover, $\phi(e_k)$ and $\phi(e_r)$ have the same orientation, as argued previously. This implies that $Z_{k_S}$ and $Z_{r_S}$ are both positive, or they are both negative. In turn, the same statement holds for the variables $W_{k}$ and $W_{r}$. In the first case we have $S_{[k]}^{*}(d) = S_{[k_S]}(c,b,a) a_i' = S_{[r_S]}(c,b,a) a_i' = S_{[r]}^*(d)$, and in the second case $S_{[k]}^{*}(d) = S_{[k_S]}(c,b,a) a_i''{-1} = S_{[r_S]}(c,b,a) a_i''{}^{-1} = S_{[r]}^*(d)$, as required.
  \end{proof}
  
  Let $w$ be the label of some loop $C_{i}$ in $\bar{\Gamma}$. Let $d_0, \dots, d_{r-1}$ be all edges in $\bar{\Gamma}$ with label in $x^{\pm 1}$ and with an endpoint in $C_i$.     Then there are indices
  $1\leq j_0< j_1 < \dots < j_{2r-1} \leq |S^*|$ such that $e_{j_{2t}} \sim e_{j_{2t+1}}$ and $\phi(e_{j_{2t}}) = \phi(e_{j_{2t+1}})= d_t$ for all $t= 0, \dots,r-1$. Moreover, for each $t$ either $W_{j_{2t}}$ is positive and  $W_{j_{2t+1}}$ is negative or, vice-versa, $W_{j_{2t}}$ is negative and  $W_{j_{2t+1}}$ is positive.  Notice that $w$ is precisely $S^*$ after removing subwords of the form $(W_{j_i +\epsilon_i} \dots W_{j_{i+1} + \epsilon_i -1})$ from $S^*$, for some $\epsilon_i\in \{0,1\}$  (see Figure \ref{dibumaig2}). More precisely: 
  
  \begin{align}
  w=(W_{0} \dots W_{j_0 + \epsilon_0 -1}) &(W_{j_1 + \epsilon_0}\dots W_{j_2 + \epsilon_2 -1}) \dots \label{e: 13_oct_19}\\ &(W_{j_{2r-3} + \epsilon_{2r-4}} \dots W_{j_{2r-2} + \epsilon_{2r-2} -1}) (W_{j_{2r-1} + \epsilon_{2r-2}} \dots W_{|S^*|}).\nonumber
  \end{align}
  
  Each word $(W_{j_i +\epsilon_i} \dots W_{j_{i+1}+\epsilon_i-1})$   ($0\leq i\leq 2r-2$) is equal to $S_{[j_{i} +\epsilon_i]}^{*}{}^{-1}S_{[j_{i+1} +\epsilon_i +1]}^*$ if  $W_{j_{i}}$  is positive and $\epsilon_i = 0$ (see Figure \ref{dibumaig2}),  or if $W_{j_i}$ is negative and $\epsilon_i = 1$; or it is equal to  $W_{j_{i}}^{-1}S_{[j_{i}]}^{*}{}^{-1}S_{[j_{i+1}]}^* W_{j_{i+1}}$ if  $W_{i_j}$ is negative and $\epsilon_i=0$, or if $W_{i_j}$ is positive and $\epsilon_i=1$. By the previous Lemma \ref{r: trivial} and because $W_{j_{i}}(d) = W_{j_{i+1}}(d)$, each such factor evaluates on $d$ to the trivial element. 
  We conclude from \eqref{e: 13_oct_19} that  $1=S^{*} = w(d)$. This completes the proof of Theorem \ref{l: small_cancellation_parameters}.
  
  \begin{remark}\label{r: more than one eq}
   Suppose that $\phi(z) \equiv \forall x \exists y \Sigma$, where $\Sigma$ has more than one equation, namely, $\Sigma\equiv \bigwedge\limits_{i=1, \dots,s}\sigma_i=1$,  where $\sigma_i(x,y,z)=1$ is a single equation. Consider $\phi'(z)\equiv \forall x \exists y \Sigma'$, where $\Sigma'$ is the equation  $\prod\limits_{i=1, \dots, s} \sigma_i=1$. Let $N$ be provided by Theorem \ref{l: small_cancellation_parameters} for the formula $\phi'$. If $G$ acts on a tree and $\sigma_i(a,b,c)=1$, $i=1, \dots, s$, where $a\in G^{|x|}$ is $N$-small cancellation over $c$, then  clearly, $\Sigma'(a,b,c)=1$ in $G$. The formal solutions (relative to a set of diophantine conditions) obtained from running the procedure for the equation $\Sigma'=1$ using the tuple $(a,b,c)$ as above are, by construction, also formal solutions for the formula $\phi$ (relative to a set of diophantine conditions).
  \end{remark}

  \begin{figure}[ht]
  \begin{center}
  \begin{tikzpicture}[scale=1]
  	\tikzset{near start abs/.style={xshift=1cm}}
  	
  	\draw[thick] plot [smooth cycle] coordinates {(-1,2) (1,2) (2,0) (1.5,-1.5) (-1.5,-1.5) (-2,0)};
  	\filldraw (-1,2) circle (1pt) node[align=left, below, xshift=-0.4cm, yshift=0.3cm] {$e_{j_3}$} node[color=blue, align=right, above, xshift=-0.1cm, yshift=0.8] {$A_4$};
  	\filldraw (1,2) circle (1pt) node[align=right, below, xshift=0.5cm] {$e_{j_2}$} node[color=blue, align=right, above, xshift=-0.1cm, yshift=0.7] {$A_3$};
  	\filldraw (2,0) circle (1pt) node[align=left, above, xshift=0.3cm] {$e_{j_1}$} node[color=blue, align=right, above, xshift=0.2cm, yshift=0.5cm] {$A_2$};
  	\filldraw (1.5,-1.5) circle (1pt) node[align=left, below, xshift=-0.2cm] {$e_{j_0}$} node[color=blue, align=right, below, xshift=-0.75cm, yshift=-0.2cm] {$A_1$};
  	\filldraw (-1.5,-1.5) circle (1pt) node[align=left, below, xshift=0.3cm] {$e_{j_5}$} node[color=blue, align=right, below, xshift=-0.3cm, yshift=-0.7] {$A_6$};
  	\filldraw (-2,-0) circle (1pt) node[align=left, above, xshift=-0.1cm] {$e_{j_4}$} node[color=blue, align=right, below, xshift=-0.5cm, yshift=-0.5] {$A_5$};

  	\filldraw (-1.4,1.5) circle (1pt);
  	\filldraw (1.4,1.5) circle (1pt);
  	\filldraw (1.85,0.55) circle (1pt);
  	\filldraw (0.9,-1.62) circle (1pt);
  	\filldraw (-0.9,-1.62) circle (1pt);
  	\filldraw (-1.85,0.55) circle (1pt);
  	
  	\filldraw (-0,-1.65) circle (1pt) node[align=left, below] {$*$};
  	
  	\draw [thick, middlearrow = {<}]  plot [smooth, tension=1] coordinates { (-1,2) (0,1.5) (1,2)  };
  	\draw [thick,  middlearrow = {<}]  plot [smooth, tension=1] coordinates { (-1.4,1.5) (0,1)  (1.4,1.5)  };
  	\draw [thick, middlearrow = {<}]  plot [smooth, tension=1] coordinates { (2,0) (1.3,-0.6) (1.5,-1.5) };
  	\draw [thick, middlearrow = {<}]  plot [smooth, tension=1] coordinates { (1.85,0.55) (0.8,-0.2) (0.9,-1.62)};
  	\draw [thick, middlearrow = {>}]  plot [smooth, tension=1] coordinates { (-2,0) (-1.3,-0.6)  (-1.5,-1.5)};
  	\draw [thick, middlearrow = {>}]  plot [smooth, tension=1] coordinates { (-1.85,0.55) (-0.8,-0.2) (-0.9,-1.62)  };
  	\node at (3,2) {$\Gamma$};
  	
  \end{tikzpicture}
  \end{center}
  \caption{ A particular case setting of $\phi^{-1}(D)$ for $D$ a cycle of $\bar{\Gamma}$.  By definition, all edges have counterclockwise orientation. We have depicted with bands how different edges of $\Gamma$ are identified in $\bar{\Gamma}$. In this particular case, $\phi^{-1}(D)$ consists in the union of paths $[*, A_1], [A_2, A_3], [A_4,A_5], [A_6, *]$. If $W_{j_0}$ is positive, then the path from $A_1$ to $A_2$ has label $S_{[j_0]}^{-1} S_{[j_1]}$, otherwise it has label $W_{j_0}^{-1}S_{[j_0]}^{-1} S_{[j_1]} W_{j_1}$. Likewise, if $W_{j_2}$ is positive, then the path from $A_3$ to $A_4$ has label $W_{j_2}^{-1}S_{[j_2]}^{-1} S_{[j_3]}W_{j_3}$, and otherwise it has label $S_{[j_2]}^{-1} S_{[j_3]}$.  We stress that in general there can be more edges identifies in $\Gamma$ which have not been depicted, and which have nothing to do with $\phi^{-1}(D)$ projecting onto a cylcle.} \label{dibumaig2}
  \end{figure}
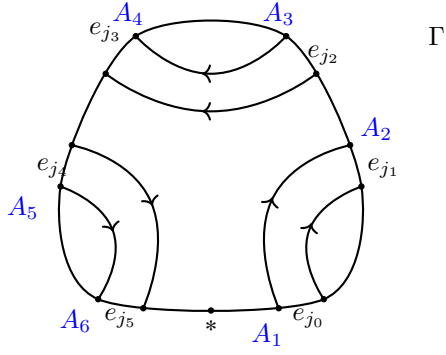
  
    \subsubsection{Proof of Lemma \ref{l: main_Merzlyakov_lemma}}

      The method used to construct the intervals $\hl_{i}$ follows a greedy sieve procedure. Roughly speaking, we consider all possible $\ttraj \in \Traj(x_i)$ one by one and every time we witness a failure of the conditions in the lemma for the interval at a particular stage, we replace it with a proper sub-interval whose length is larger than a fixed positive bound. A pigeon-hole principle argument, together with the weakly small cancellation property allows one to derive a uniform bound (i.e.\ independent from the particular solution $(c,a,b)$) on the number of times this reduction can take place. It follows that if the given $N$ is large enough, then the process eventually stabilizes on the non-degenerate interval $\hl_{i}$ as desired.
      
      Recall that $T$ is the $G$-tree and $*\in T$; the system of equations $\Sigma$ consists of a unique equation $S(w,x,y)=1$, where
      $S(w,x,y)=Z_{1}\cdots Z_{|S|}$, for all $1\leq r \leq |S|$,  $Z_{r}\in x^{\pm 1}\cup y^{\pm 1}\cup w^{\pm 1}$; 	$S_{[r]}=Z_{1}Z_{2}\cdots Z_{r}$ ($S_{0}=1$); and $v_{r}=S_{r}(c,a,b)\cdot*$, $r=1,\dots, |S|$. Furthermore $Hull(v_0, \dots, v_{|S|}) = \Hull(\mc{V})$ is the convex hull in $T$ of the vertices $v_i$, $i=0, \dots, |S|$, and $\bad$ is the collection of all points in $v_i$ and all branching points of $\mathcal Hull\subset T$, i.e. points of valency strictly more than 2 in $Hull$ (but note that such point could have higher valency as a vertex of $T$). Given $\ttraj\in\Traj(x_i)$, by  $A_{i}^{Br}(\ttraj)$ we denote the collection of open subintervals of $A_{i}$ disjoint from the pre-image of $\bad$ by $\Push_{\ttraj}$ and either contained or disjoint from the domain of $\Push_{\ttraj}$. Recall that $L$  is the minimum length of the intervals $\{[*, a_i*]\}$.
      
      \newcommand{\R}[0]{\mathbb{R}}

      \medskip
      
      We now enumerate $\Traj(x_i)$ as $\ttraj_{1},\ttraj_{2},\dots\ttraj_{n},\dots$ in such a way that for any two indices $\ell$ and $\ell'$ with $\ell' < \ell$, it is not the case that $\ttraj_{\ell}$ is a prefix of $\ttraj_{\ell'}$ (i.e.\ the first $|\ttraj_\ell|$ components of $\ttraj_{\ell'}$ are not equal to the  tuple $\ttraj_\ell$). Fix $x_i \in x^{\pm 1}$. We will prove that there exists an integer $N>0$ and a sequence of segments $(\tau_\ell)$ in $A_i$ such the following holds for all $\ell \geq 0$:
      
      \begin{enumerate}
      \item $\tau_{\ell+1}\subseteq \tau_\ell$,
      \item $|\tau_\ell|> L/N$,
      \item  for all $1\leq \ell'\leq \ell$ either $\tau_\ell$ is disjoint or contained in $dom(\Push_{\ttraj_{\ell'}}^*)$, and in the latter case $\tau_{\ell}\in A_{i}^{Br}(\ttraj_{\ell'})$.
      \end{enumerate}
      The proof of Lemma \ref{l: main_Merzlyakov_lemma} follows from the existence of such sequence: it suffices to take as $\hat{A}_i$ the largest segment contained in $\bigcap_{\ell} \tau_\ell$ and Item 2 assures that the segment is non-degenerate.

      We define the sequence $\tau_\ell$  inductively on $\ell$, starting with $\tau_{0}=A_{i}$. Assume that some intervals $\tau_0, \dots, \tau_\ell$ satisfying the above properties have been found, then we let
      \begin{enumerate}
      \item[(*)] $\tau_{\ell+1}=\tau_{\ell}$, if  $\tau_{\ell}\in A_i^{Br}(\ttraj_{\ell+1})$,
      \item[(**)]  otherwise $\tau_{\ell+1}$ is  the largest interval $\tau'$ such that $\tau' \subseteq \tau_\ell$ and $\tau' \in A_i^{Br}(\ttraj_{\ell+1})$.
      \end{enumerate}
      Observe that by definition $\tau_{\ell+1}\subseteq \tau_\ell$ and for each $\ell$ we have that $\tau_\ell$ is either contained or disjoint from $dom(\Push_{\ttraj_{\ell}})$. Hence the sequence $\tau_\ell$ satisfies Condition 3 above, because for each $1\leq \ell' \leq \ell$ we have $\tau_\ell\subseteq \tau_{\ell'}$, so $\tau_\ell$ is either contained or disjoint from $dom(\Push_{\ttraj_{\ell'}})$. Hence the proof of the lemma will be complete once we see that $\tau_\ell \geq L/N$ for all $\ell\geq 1$.
      
      We start with the following
      
      \begin{remark}\label{r: remark_merzlyakov_lemma}
      Suppose $\tau_{\ell+1}$ has been obtained from $\tau_\ell$ by applying operation (**). Then $|\tau_{\ell+1}|\geq \frac{|\tau_\ell|}{|Br|+1}$.
      \end{remark}
      This remark is true because when applying (**), $\tau_{\ell}$ is subdvided into at most $|Br|+1$ parts and then $\tau_{\ell+1}$ is set to be the largest of these parts.
      
      Let $\kappa_{\ell}$ be the number of times operation (**) has been applied while constructing $\tau_0, \dots, \tau_\ell$.
      We are going to prove that there exists some constant $K$ depending only on $S$ such that $\kappa_\ell \leq K$ for all $\ell \geq 0$. Suppose towards contradiction that such a constant $K$ does not exist.
      
      Assume  (**) is applied on step $\ell$. Then there exists a vertex $u_\ell\in Br$ such that one of the endpoints of $\tau_{\ell+1}$ is $\Push_{\ttraj_{\ell}}^{-1}(u_\ell)$. Since $|Br|$ is uniformly bounded in terms of $S$ and we are assuming that (**) is applied arbitrarily many times, there exists a vertex $u_*\in Br$ such that $u_\ell = u_*$  for arbitrarily many $\ell$'s. Furthermore, we can assume that there exists an edge from $Hull$ that is adjacent to $u_*$ such that, either for arbitrarily many $\ell$'s $\tau_{\ell+1}$ starts with $\Push_{\ttraj_{\ell}}^{-1}(e_*)$, or for arbitrarily many $\ell$'s $\tau_{\ell+1}$ ends with $\Push_{\ttraj_{\ell}}^{-1}(e_*)$.  
      
      Let $\ell_1$ and $\ell_2$ be the two smallest indices for which operation $(**)$ is applied on steps $\ell_1$ and $\ell_2$ and such that $\tau_{\ell_{i+1}}$ starts (or ends) with $P_{\ttraj_{\ell_{i+1}}}^{-1}(e_*)$, for both $i=1,2$. We have that $\kappa_{\ell_2}$ (the number of times (**) has been applied up to when step $\ell_2$ has been reached) is uniformly bounded in terms of $S$, say by an integer $M$.

      Taking $N > |Br|^{M}$ and using Remark \ref{r: remark_merzlyakov_lemma} repeatedly we obtain
      \begin{equation}\label{e: merzlyakov_lemma_final}
      |\tau_{\ell_2+1}|\geq \frac{|\tau_0|}{(|Br|+1)^{\kappa_{\ell_{2}}}}  \geq \frac{|\tau_0|}{(|Br|+1)^{M}}\geq \frac{L}{N}.
      \end{equation}
      Since $\tau_{\ell_2+1} \subseteq dom(\Push_{\ttraj_{\ell_1}}^*) \cap dom(\Push_{\ttraj_{\ell_2}}^*)$ we obtain (see Figure \ref{f: dib_6_sept_19})
      $$
      |S_{[\ttraj_{\ell_1}]} \cdot \tau_{\ell_2+1}\cap S_{[\ttraj_{\ell_2}]}\cdot \tau_{\ell_2+1}| = |P_{\ttraj_{\ell_1}}(\tau_{\ell_2 +1}) \cap P_{\ttraj_{\ell_2}}(\tau_{\ell_2+1})| \geq |\tau_{\ell_2+1}| \geq \frac{L}{N}.
      $$
      \begin{center}
      
        \begin{figure}[ht]
  \begin{center}
  \begin{tikzpicture}[scale=1]
  \filldraw (0,0) circle (2pt) node[below] {$*$};
   \filldraw (10,0) circle (2pt) node[below] {$a_i*$};
    \filldraw (1,0) circle (2pt);
     \filldraw (9,0) circle (2pt);
      \filldraw (1,3) circle (2pt) node[below] {$u_*$};
       \filldraw (9,3) circle (2pt);
       \filldraw[color=red] (4,0) circle (1pt);
        \filldraw[color=red] (9,0) circle (1pt);
        \filldraw[color=red] (1,3) circle (1pt);
        \filldraw[color=red] (6,3) circle (1pt);
    \draw[thick] (0,0) -- (1,0) -- (9,0) -- (10,0);
    \draw[thick] (-2,3) -- (12, 3);
    \draw[thick] (2,4) -- (1,3) -- (-1, 4);
    \draw[thick] (8,4) -- (9,3) -- (11, 4);
    \draw[thick, color=blue] (1,0) -- (9,0);
    \draw[thick, color=blue] (1,3) -- (9,3);
    \draw [color=red,thick,decorate,decoration={brace,amplitude=7pt},yshift=-1.5pt] (6,3)--(-2,3)  node [midway,yshift=-0.6cm] {$\Push_{\ttraj_{\ell_2}}(\tau_{\ell_1 +1})$};
\node[color=blue] at (3,-0.3) {$\tau_{\ell_1 +1}$};
\node[color=blue] at (7.3,2.7) {$\Push_{\ttraj_{\ell_1}}(\tau_{\ell_1 +1})$};
\draw [color=red,thick,decorate,decoration={brace,amplitude=7pt},yshift=1.5pt] (4,0)--(9,0)  node [midway,yshift=0.6cm] {$\tau_{\ell_2 +1}$};
\draw [color=red,thick,decorate,decoration={brace,amplitude=7pt},yshift=1.5pt] (1,3)--(6,3)  node [midway,yshift=0.6cm] {$\Push_{\ttraj_{\ell_2}}(\tau_{\ell_2 +1})$};
\draw[->, color=blue] (5,0.7) -- (5,2.2);
\draw[->, color=red] (4.75,0.7) -- (3.2,2.2);
\node[color=red] at (3.8, 1.2) {$\Push_{\ttraj_{\ell_2}}$};

\node[color=blue] at (5.5, 1.15) {$\Push_{\ttraj_{\ell_1}}$};

  \end{tikzpicture}
  \end{center}
  \caption{Hypothetical depiction of the image of $\tau_{\ell_2 
      1}$  by the map $\Push_{\ttraj_{\ell_2}}^*$, and of $\tau_{\ell_2 
      1}$  by  $\Push_{\ttraj_{\ell_1}}^*$. Using the small cancellation property we see that this scenario cannot hold and that one must have $\Push_{\ttraj_{\ell_1}}^*(\tau_{\ell_2+1})= \Push_{\ttraj_{\ell_2}}^*(\tau_{\ell_2+1})$.}
          \label{f: dib_6_sept_19}
  \end{figure}
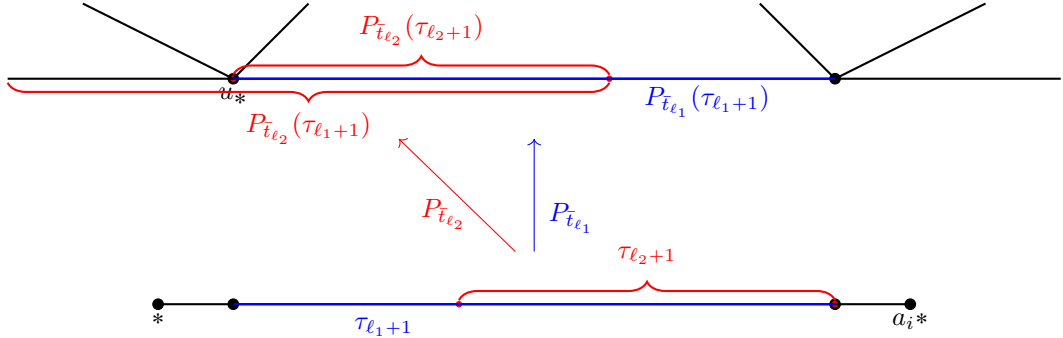
  \end{center}
      This implies that $|A_i \cap (S_{[\ttraj_{\ell_1}]}^{-1}S_{[\ttraj_{\ell_2}]} \cdot A_i)|\geq |\tau_{\ell_2+1} \cap (S_{[\ttraj_{\ell_1}]}^{-1}S_{[\ttraj_{\ell_2}]} \cdot \tau_{\ell_2+1})| \geq  L/N$, and so by the weak small cancellation property we have that  $S_{[\ttraj_{\ell_1}]}^{-1}S_{[\ttraj_{\ell_2}]}$ acts as the identity on $A_i \cap S_{[\ttraj_{\ell_1}]}^{-1}S_{[\ttraj_{\ell_2}]}\cdot A_i$. Since  $S_{[\ttraj_{\ell_2}]} = S_{[\ttraj_{\ell_1}]}S_{[\ttraj_{\ell_1}]}^{-1}S_{[\ttraj_{\ell_2}]}$, the image of $\tau_{\ell_2}$ by $S_{[\ttraj_{\ell_2}]}$ coincides with the image of $\tau_{\ell_2}$ by $S_{\ttraj_{\ell_1}}$. This means that $\Push_{\ttraj_{\ell_2}}(\tau_{\ell_2+1}) = \Push_{\ttraj_{\ell_1}}(\tau_{\ell_2+1})$, and so $\Push_{\ttraj_{\ell_1}}(\tau_{\ell_2+1})$  does not contain any vertex from $\bad$, because if it did then when constructing $\tau_{\ell_1+1}$ the segment $\tau_{\ell_1}$ would be subdivided along a point inside $\tau_{\ell_2+1}$, and then we would not have $\Push_{\ttraj_{\ell_1}}(\tau_{\ell_2+1})=\Push_{\ttraj_{\ell_2}}(\tau_{\ell_2+1})$.  
      This means that  transformation (*) is applied when constructing $\tau_{\ell_2+1}$ during step $\ell_2$, contradicting our assumptions. Thus the sequence $(\kappa_\ell\mid \ell)$ had to be bounded in terms of $S$ to begin with, say by an integer $M'$. Now, similarly as in \eqref{e: merzlyakov_lemma_final} and taking  $N > |Br|^{M'}$ we obtain  $$|\tau_\ell| \geq \frac{|\tau_0|}{(|Br|+1)^{\kappa_\ell}}\geq \frac{|\tau_0|}{(|Br|+1)^{M'}} \geq \frac{L}{N}$$ for all $\ell\geq 0$. This completes the proof of Lemma \ref{l: main_Merzlyakov_lemma}.
      
      The following corollary of Theorem \ref{l: small_cancellation_parameters} allows one to deal with arbitrary quantifier complexity.
      \begin{cor}
      \label{c: iterated formal solutions}
      Given a positive formula
      \begin{align*}
      \phi\equiv \forall x_{1}\exists y_{1}\cdots\forall x_{r}\exists y_{r}\Sigma(w,x_{1},y_{1},x_{2},y_{2}\cdots x_{r},y_{r})=1,
      \end{align*}
            where $w$, $x_{j}$ and $y_{j}$ are finite tuples of variables, there is some $N>0$ and a finite collection $\mathcal{D}$ of diophantine conditions on free variables $w$ with the following properties:
      \enum{(i)}{
      \item For each $\Delta\in \mathcal{D}$ there exists a formal solution $\alpha_{\Delta}$ to $\psi(w)$ relative to $\Delta$.
      \item \label{item witness_2} For any action of a group $G$ on a tree $T$ and tuples $c\in G^{|w|}$, $a_{j}\in G^{m_{j}}$ and $b_{j}\in G^{m_{j}}$ such that
      \enum{(a)}{
      \item $a_{j}$ is $N$-small cancellation over $(c,a_{i},b_{i})_{i<j}$ for any $1\leq j\leq r$ and
      
      \item $\Sigma(c,a_{1},b_{1},a_{2},b_{2}\cdots a_{r},b_{r})=1$,
      }
      there is a diophantine condition $\exists v\,\Theta(w,v)=1$ in $\mathcal{D}$ and $d\subset G$ such that $\Theta(c,d)=1$ holds in $G$ .
      }
      \end{cor}

      \begin{proof}
      We prove that if a tuple with the properties above exists then the given formula admits a formal solution by induction on $r$. We apply Theorem
      \ref{l: small_cancellation_parameters}  to the formula $\psi(w)\cong\forall x\exists y\Sigma(w',x,y)=1$ where we use the partition of variables $w'=(w,x_{1},y_{1},x_{2}\cdots y_{r-1})$, $x=x_{r}$ and $y=y_{r}$. The case $r=1$ is precisely Theorem \ref{l: small_cancellation_parameters}.
      
      Assume $r\geq 1$. Let $N_{0}$,  $\mathcal{D}=\{\exists u\,\Theta_{\lambda}(w',u)=1\}_{\lambda\in\Lambda}$, and $\{\alpha_\lambda\mid \lambda \in \Lambda\} $  be the positive integer, the diophantine conditions, and the formal solutions provided by  Theorem
      \ref{l: small_cancellation_parameters}.
      
      For each $\lambda\in\Lambda$ we can consider the formula:
      \begin{align*}
      \phi_{\lambda}(w)\equiv\forall x_{1}\exists y_{1}\forall 	x_{2}\cdots\exists y_{r-1}\forall x_{r-1}\exists y_{r-1}\exists u\,\,\Theta_{\lambda}( w', u)=1,
      \end{align*}
      where, as agreed above, $w'=(w,x_{1},y_{1},x_{2}\cdots y_{r-1})$.
      
      By induction hypothesis we can assume that the statement of the theorem holds for each $\phi_\lambda$, that is there is some $N_{\lambda}>0$, a set $\mathcal D_\lambda=\{ \exists v \ \Theta_{\lambda,i}(w,v) \}_{i\in I_\lambda}$ and formal solutions $\alpha_{\lambda,i}^{ (r-1)}$ such that if $G$ acts on a simplicial tree and some tuples $c',a_{1}',b_{1}',a_{2}',b_{2}',\dots b_{r-1}',d'$ in $G$ satisfy
      \begin{align*}
      \Theta_{\lambda}(c',a_{1}',b_{1}',a_{2}', \cdots a_{r-1}',b_{r-1}', d')&=1
      \end{align*}
      and $a_{j}'$ is $N$-small cancellation over $\{c',a_{l}',b_{l}'\}_{l<j}$ for any $1\leq j\leq r-1$, then $\Theta_{\lambda,i}(c',d'')=1$ holds in $G$ for some $d''\in G^{|u|}$.
      
      Take $N=\max\{N_{0},N_{\lambda}\}_{\lambda\in\Lambda}$ and let $\mathcal D$ be the set of diophantine conditions $$\{\exists v \  \Theta_{\lambda,i}(w,v)=1\}_{\lambda\in \Lambda, i \in I_\lambda}.$$
       First note that $\alpha_{\lambda,i}=(\alpha_{\lambda,i}^{(r-1)}, \alpha_\lambda)$  is a formal solution for $\phi$ relative to the diophantine condition $\Theta_{\lambda,i}$. Indeed, since $\alpha_\lambda$ is a formal solution of $\phi$ relative to the diophantine condition $\exists u \Theta_\lambda(w',u)$, by definition we have that $\Sigma(w,x_1,\dots, y_{r-1},x_r, \alpha_\lambda)$ is in the normal closure of $\Theta_\lambda(w',u)$. In turn, since $\alpha_{\lambda, i}^{(r-1)}$ is a formal solution for $\phi_\lambda$ relative to $\Theta_{\lambda,i}(w,v)$, it follows that $\Theta_{\lambda}(w, x_1, \alpha_{\lambda,i}^{(1)}, \dots, x_{r-1}, \alpha_{\lambda,i}^{(r-1)},u)$  belongs to the normal closure of $\Theta_{\lambda,i}(w,v)$ and so $\Sigma(w,x^1,\alpha_{\lambda,i}^{(1)}, \dots, \alpha_{\lambda, i}^{(r-1)}, x_r, \alpha_\lambda)$ also belongs to the normal closure of $\Theta_{\lambda,i}(w,v)$. This proves Item (i) in the statement of the corollary.
      
      Let us now check the second item of the statement of the corollary. Suppose that there is a group $G$ acting on a tree $T$ and tuples of elements $c,a_{1},b_{1},\dots a_{r},b_{r}$ in $G$ as in the assumption of the statement.
      
      Since $a_{r}$ is $N$-small cancellation over $\{c,a_{j},b_{j}\}_{j<r}$, it follows from Theorem \ref{l: small_cancellation_parameters} that $a_r$ and the tuples $\{c,a_{j},b_{j}\}_{j<r}$ satisfy at least one of the diophantine conditions $\Theta_\lambda$ for some $\lambda\in\Lambda$, so
      \begin{align*}
      \Theta_{\lambda}(c,a_{1},b_{1},a_{2}\dots, a_{r-1},b_{r-1}, d)=1
      \end{align*}
      for some $d\in G^{|u|}$. From the induction hypothesis described above, since $a_{l}$ is $N$-small cancellation over $c,a_{j},b_{j}$ for $j<l<r$  and
      $\phi_{\lambda}$ admits a formal solution relative to $\Theta_{\lambda,i}$, we have that there exists $e \in G^{|v|}$ such that $\Theta_{\lambda,i}(c,e)=1$. This proves the statement.
      \end{proof}
      
\section{Uniform weak small cancellation from weak stable elements}\label{s: section 4}

In the previous section, we determine a sufficient condition for a group acting on a tree to have trivial positive $\forall \exists$-theory. This condition, a priory, requires checking infinitely many conditions - the existence of weakly small cancellation $m$-tuples, for all $m\in \mathbb N$.  The goal of this section is to describe a uniform way to obtain $m$-tuples of (weak) small cancellation elements for all $m \in \mathbb N$ from just one (weak) stable hyperbolic element $h$ and an element $g$ which does not belong to the stabiliser of the axis of $h$, see Corollary \ref{cor:uniformWSC}. Therefore, if a group acts on a tree irreducibly, it suffices to prove the existence of a hyperbolic stable element in order to conclude that its positive theory is trivial.
  
  \medskip
  
  We start with some definitions and remarks. Throughout this section $G$ denotes a group acting on a real tree  $T$. Recall that, for $g\in G$, we denote by $E(g)$ the set of elements in $G$ that fix set-wise the set $A(g)$, where $A(g)$ is either the axis of $g$ if $g$ is hyperbolic, or the set of all elements pointwise fixed by $g$. The following definition was introduced already in the preliminary Section \ref{s: preliminaries}, and it is reproduced here for convenience.

  \begin{definition}\label{defn:weakly stable}
  	Let $h$ be a hyperbolic element of $G$ (with respect to the action of $G$ on $T$). Then $h$ is called \emph{weakly} $\lambda$-\emph{stable}  if for all $g\in G$ the inequality $|A(h)\cap gA(h)|> \lambda tl(h)$ implies that $g\in E(h)$. 
  	
  	We say that the hyperbolic element $h$ is $\lambda$-\emph{stable} if it is weakly $\lambda$-stable and $K(h)=\{1\}$.
  	
  	If $\lambda=1$ then we drop the prefix $\lambda$ in the above definitions.

  \end{definition}
  
  \begin{remark}\label{r: stability and powers}
  	The following observations will be used later on:
  	\item If $\lambda < \mu$ then any $\lambda$-stable element is $\mu$-stable.
  	
  	\item  For any $n>0$, if a hyperbolic element $h$ is (weakly) $\lambda$-stable then $h^{n}$ is (weakly) $\frac{\lambda}{n}$-stable.
  	
  	\item If $G$ acts on a real tree $T$ with trivial edge stabilizers, then every hyperbolic element of $G$ is
  	stable.
  \end{remark}

  \begin{definition}\label{d: acylindrical_pair}
  	Let $G$ be a group acting on a real tree. We will call a pair of elements $g,h$ \emph{acylindrical} if
  	\begin{itemize}
  		\item $\subg{g,h}$ acts irreducibly on $T$,
  		\item $|A(g)\cap A(h)|\le \frac{1}{2}\max\{tl(g),tl(h)\}$ and
  		\item if  $g$  is hyperbolic then $g$ is  weakly $\frac{1}{3}$-stable, and similarly for $h$.
  	\end{itemize}
  \end{definition}
  
  We also recall the notion of small cancellation:

  \begin{definition}
  	\label{d: small cancellation}
  	Fix an action of a group $G$ on a real tree $T$. Consider a tuple of elements $a=(a_{1},\dots, a_{m})\in G^{m}$.
  	Given $N>0$, we say that $a$ is \emph{weakly $N$-small cancellation} (in $T$) if the following holds for some base point $*\in VT$ 
  	
  	\enum{(a)}{
  		\item \label{SCA_prelim} $\dis{a_{i}}>N$ for all $i$,
  		\item \label{SCB_prelim} $\dis{a_{j}}\leq\frac{N+1}{N}\min_i tl(a_{i})$ for all $i, j$,  and
  		\item \label{SCC_prelim}
  		For all $g\in G,\epsilon\in\{1,-1\}$ and $1\leq i,j\leq m$  the condition $|[*,a_{i}*]\cap g[*,a_{j}*]^{\epsilon}]|\geq\frac{1}{N}\min_i tl(a_{i})$ cannot hold unless $i=j$, $\epsilon =1$ in which case $g$ acts like the identity on the collection of points of $[*,a_{i}*]$ that it sends to $[*,a_{i}]$, that is $g$ acts as the identity on the the intersection $[*,a_{i}*]\cap g[*,a_{i}*]$.
  	}
  	Given $c_{1},c_{2},\dots, c_{k}\in G$, we say that a tuple $a$ is weakly $N$-small cancellation \emph{over $c_{1},c_{2},\dots, c_{k}$} if $*$ can be chosen in such a way that $\min_{i}\dis{a_{i}}>N \cdot \dis{c_{j}}$, for all $1\leq j\leq k$. We say that $(a_{1},a_{2},\dots a_{m})$ is \emph{$N$-small cancellation} (over $c$)
  	if all the above applies and we additionally require that $g$ can only be equal to the identity in the last alternative of Item (\ref{SCC_prelim}).

  \end{definition}
  
  \newcommand{\N}[0]{\mathbb{N}}

  \begin{remark}\label{r: alb_remark_3}
  	The following are immediate consequences of Definition \ref{d: small cancellation}. Below $a_1, \dots, a_n$ is an $N$-small cancellation tuple.
  	\begin{enumerate}
  		\item $a_i$ is hyperbolic for all $i$.
  		\item If $a_{1},a_{2},\dots a_{k}$ is $N$-small cancellation over $c$ then so is $a_{1},a_{2},\dots a_{i}^{-1},a_{i+1},\dots a_{k}$.
  		\item $d(*,A(a_i)) \le \frac{L}{2N}$, where $L=min_i tl(a_i)$. Indeed, $d(*,a_i*) = 2d(*, A(a_i)) + tl(a_i)$. From this and Item \ref{SCB} we have that
  		$
  		d(*,A(a_i)) \leq  \frac{N+1}{2N}L - \frac{tl(a_i)}{2}
  		$
  		from where the inequality follows.
  		\item For all $i,j$ and all $g\notin \langle a_j\rangle$ we have $|A(a_i) \cap gA(a_j)| \leq \frac{L}{N}$.
  		
  	\end{enumerate}
  \end{remark}

  \begin{lemma}
  	\label{l: acylindrical pairs} Let $G$ act on a tree $T$ and take two elements $g,h\in G$ such that $\subg{g,h}$ acts irreducibly on $T$. Suppose that either $g$ is elliptic or is the $m$-th power of a weakly $\lambda$-stable element, where $m\geq 10\max\{\lambda,1\}$. Make the same assumption on $h$. Then the pair $\{g,h\}$ is acylindrical.
  \end{lemma}
  \begin{proof}
  	If $g$ is hyperbolic, then it is  weakly $\frac{1}{10}$-stable by assumption and by Remark \ref{r: stability and powers}. In particular, it is weakly $\frac{1}{3}$-stable again by this remark. The same is true for $h$.  Therefore Items 1 and 3 of Definition \ref{d: acylindrical_pair} hold. It remains to see that $|A(g)\cap A(h)|\le\frac{1}{2}\max\{tl(g),tl(h)\}$. 
  	
  	First of all, notice that if both $g$ and $h$ are  elliptic, then $A(g)\cap A(h)=\emptyset$, otherwise the action of $\langle g, h\rangle$ on $T$ would have a fixed point, contradicting its irreducibility. Now assume $g$ is hyperbolic and $h$ is elliptic.  Then by hypothesis $g= g_0^{m}$ for some weakly $\lambda$-stable hyperbolic element $g_0$. If we had $|A(g)\cap A(h)| > \frac{1}{2}\max\{tl(g),tl(h)\}$, then it would follow that $|A(g_0) \cap A(h)| = |A(g)\cap A(h)| > \frac{1}{2}\max\{tl(g),tl(h)\} = tl(g)/2 = m tl(g_0)/2 \geq \lambda tl(g_0)/2$, which implies $h \in E(g_0) = E(g)$; indeed, since $h$ is elliptic, we have that $h[A(g)\cap A(h)]=A(g) \cap A(h)$ and so $[A(g) \cap A(h)] \subset A(g) \cap h A(g)$, implying that $h\in E(g)$. This contradicts the fact that  $\langle g, h\rangle$ acts irreducibly on $T$. 
  	
  	The only case left is that in which both $g$ and $h$ are hyperbolic. Without loss of generality, $tl(h)\leq tl(g)$.
  	Then the $m$-th root $g_0$ of $g$ is a hyperbolic element with translation length at most $\frac{1}{10}tl(g)$. Suppose towards contradiction that $|A(g) \cap A(h) | > \frac{1}{2}\max\{tl(g), tl(h)\} = tl(g)/2$. Then
  	\begin{equation}\label{e: 7_jan}
  		|\left[A(g) \cap A(h)\right] \cap g_0\left[A(g) \cap A(h)\right]|\geq |A(g) \cap A(h)| - \frac{tl(g)}{10} \geq \frac{1}{10}\max\{tl(g),tl(h)\}.
  	\end{equation}
  	Since both $g$ and $h$ are weakly $\frac{1}{10}$-stable (as shown at the beginning of the proof), it follows from \eqref{e: 7_jan} that $g_0$ preserves both $A(g)$ and $A(h)$, a contradiction with the action of $\subg{g,h}$ being irreducible.
  \end{proof}
  
  \begin{lemma}
  	\label{l: multiple conjugation} Let $G$ be a group acting on a real tree $T$ and let $g,h\in G$ be an acylindrical pair. Consider the words $v(x,y)=x^{y^{x}}$ and $w^{cn}(x,y)=y^{v(x,y)}x$ (where `cn' stands for `conjugation').
  	Then the element $w^{cn}(g,h)$ is hyperbolic with $tl(w^{cn}(g,h))\geq\max\{tl(g),tl(h)\}$.  Moreover, $g \notin E(w^{cn}(g,h))$ and $A(g)$ and $A(w^{cn}(g,h))$ intersect coherently in case $g$ is hyperbolic.
  \end{lemma}
  
  \begin{proof}
  	Let $c = g^{h^g}$. Our first goal is to prove that $A(h^c)\cap A(g)=\emptyset$, since in
    that case applying Lemma \ref{l: chiswell2}  to  $h^c$ and $g$ we obtain
  	\begin{equation*}
  		tl(h^c g) \geq 2d(A(h^c), A(g)) + tl(h^c) + tl(g) = 2d(A(h^c), A(g)) + tl(h) + tl(g) 
  	\end{equation*}%
  	which implies that the element $w^{cn}(g,h)=c^{h}g$ is hyperbolic and has translation length bigger than $\max\{tl(g),tl(h)\}$. 

  	We distinguish several cases.
  	  In case $A(g)\cap A(h) = \emptyset$ an application of Remark \ref{r: alb_remark_1} yields $A(g)\cap (A(h^g)\cup A(h)) = \emptyset$ and $A(g)$ intersects the path between $A(h)$ and $A(h^{g})$. From a second application we get $A(g)\cap A(c)= \emptyset$. Item 2 of this same remark ensures that the path from $A(c)$ to $A(h)$ crosses $A(h^g)$. From this and the fact that the path from $A(h^g)$ to $A(h)$ crosses $A(g)$ we obtain that  the path from $A(c)$ to $A(h)$ crosses $A(g)$. Another application of Remark \ref{r: alb_remark_1}  yields $A(g)\cap A(h^c)=\emptyset$. 
  	
  	   Assume now  $A(g)\cap A(h)\neq \emptyset$. We have that  $|A(h) \cap A(g)| \leq \max\{tl(g), tl(h)\}/2$ by the acylindricity assumption on the pair $(g,h)$. Then at least one of the two elements $g$ and $h$ is hyperbolic, since otherwise $\langle g,h\rangle$ would fix a point,  contradicting the assumption that $\langle g, h\rangle$ acts irreducibly on $T$.
  	
  	Assume first that $tl(g) \geq tl(h)$, so that in particular  $g$ is hyperbolic. Notice that (see Figures \ref{f: section_4_1} and \ref{f: section_4_2})  
  	$$
  	d(A(h), A(h^g)) = tl(g) - |A(h)\cap A(g)|  \geq tl(g)/2 >0.
  	$$
  	Hence $A(h)$ and $A(h^g)$ are disjoint.

  		If $h$ is hyperbolic as well or if the intersection of $A(g)$ and $A(h)$ is non-degenerate, then since the path from $A(h)$ to $A(h^g)$ crosses $A(g)$ and $h^g$ acts on $T$ by translation along the axis $A(h^g)$, it follows that
  	$$
  	d(A(c), A(h)) =  d((h^g)^{-1}A(g), A(h)) \geq d(A(h^g), A(h)) > 0,
  	$$
  	$A(g)$ intersects the path between $A(c)$ and $A(h)$ and  $|A(c)\cap A(g)|<tl(c)=tl(g)$.   Remark \ref{r: alb_remark_11} can now be applied by taking. $c^{-1}$ in the role of $k$; indeed,  if $\A(g) \cap A(c) \neq \emptyset$ then Item 1 in the statement of the remark  is satisfied, and otherwise  the intersection is bounded by $tl(g) = tl(c)$, as shown above. Hence $A(g^{c^{-1}}) \cap A(h) =\emptyset$. Conjugating by $c$ be obtain  $A(g) \cap A(h^c) =\emptyset$.

  	\begin{figure}[ht]
  		\begin{center}
  			\begin{tikzpicture}[scale =0.8]
  				\tikzset{near start abs/.style={xshift=1cm}}
  				\draw[color=blue, thick] (0,0) -- (5,0);
  				\draw[color=blue, thick, middlearrow={>}] (5,0) -- (9,0);
  				\node (xx) at (8.6,-0.3) {\color{blue}{$A(h)$}};

  				\draw[thick, middlearrow={>}] (2,5) -- (2,0.05);
  				\draw[thick] (2,0.05) -- (3.5,0.05) -- (3.5,5);
  				\node (xx) at (1.5,4.8) { $A(g)$};

  				\draw[color=blue, thick] (9,2) -- (3.55,2) -- (3.55,3)--(5,3);
  				\draw[color=blue, thick, middlearrow={>}] (5,3) -- (9,3);
  				\node (xx) at (8.6,3.3) {\color{blue}{$A(h^g)$}};

  				\draw[thick, middlearrow={>}] (6,3.05) -- (6,5);
  				\draw[thick] (4.5,5) -- (4.5,3.05) -- (6,3.05);
  				\node (xx) at (6.8,4.8) { $A(g^{h^g})$};

  				\draw [thick,color=red,decorate,decoration={brace,amplitude=10pt},xshift=-4pt,yshift=0pt]
  				(3.6,0.1) -- (3.6,3) node [midway, left] {\color{red}{ $tl(g)$\ \ \;}};
  				\draw [thick,color=red,decorate,decoration={brace,amplitude=10pt},xshift=-4pt,yshift=0pt]
  				(3.7,1.9) -- (3.7,0.1) node [midway, right] {\color{red}{ \ \ $tl(g)-|A(h)\cap A(g)|\ge \frac12 tl(g)$}};
  				\draw [thick,color=red,decorate,decoration={brace,amplitude=5pt},xshift=-4pt,yshift=0pt]
  				(3.65,-0.05)-- (2.15,-0.05);
  				
  				\node (xx) at (2.75,-0.9) {\color{red}{ $\le \frac12 tl(g)$}};
  				
  			\end{tikzpicture}
  		\end{center}
  		\caption{The case in Lemma \ref{l: multiple conjugation} when $A(g)\cap A(h)\neq \emptyset$ and both elements $g$ and $h$ are hyperbolic. We stress that, in general, $A(g)$ and $A(g^{h^g})$ may intersect.} \label{f: section_4_1}
  	\end{figure}
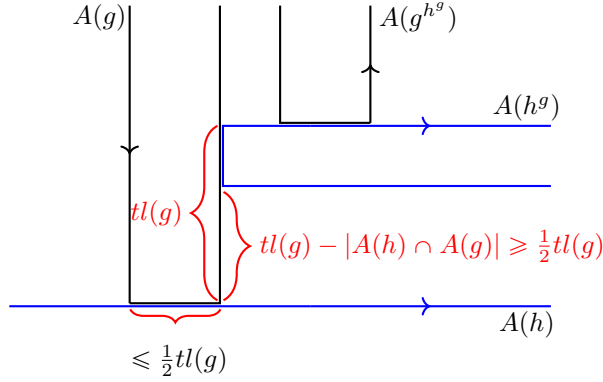
  	 Next we consider the case in which $g$ is hyperbolic, $h$ is elliptic and $A(g)\cap A(h)$ is a single point. By the definition of acylindrical pair $g$ is weakly $1/3$-stable and $\langle h^g, g \rangle = \langle h, g\rangle$ acts irreducibly on $T$. We thus have $|A(g) \cap A(c)| = |A(g) \cap (h^g)^{-1}A(g)|\leq tl(g)/3$.
  	   Moreover $$d(A(c), A(h)) \geq tl(g) - |A(c) \cap A(g)| \geq 2tl(g)/3 >0,$$ and the projection of $A(h)$ to $A(c)$ lie in $A(c)\cap A(g)$ (see Figure \ref{f: section_4_2}). An application of case \ref{option2} of Remark \ref{r: alb_remark_11} (taking $c^{-1}$ in the role of $k$) yields $A(h^c)\cap A(g)=\emptyset$ just as in the previous case. 

  	\begin{figure}[ht]
  		\begin{center}
  			\begin{tikzpicture}[scale =0.9]
  				\tikzset{near start abs/.style={xshift=1cm}}

  				\draw[ thick] (0,0) -- (4.5,0);
  				\draw[ thick, middlearrow={>}] (4.5,0) -- (6,0) node[below left] {\footnotesize $A(g)$};

  				\draw[ thick] (2,3) -- (3,0.05) --(4.5,0.05) -- (5.5,3) node[below right] {\footnotesize $A(g^{h^g})=(h^g)^{-1}A(g)$};

  				\draw [thick,color=red,decorate,decoration={brace,amplitude=7pt},xshift=-4pt,yshift=0pt] (1.2,0.1) -- (3.05,0.1) node [color=red,midway,xshift=0cm, yshift=0.45cm]{\footnotesize $\ge \frac{2tl(g)}{3}$};
  				\draw [thick,color=red,decorate,decoration={brace,amplitude=7pt},xshift=-4pt,yshift=0pt] (3.85,-0.1) -- (1.2,-0.1) node [color=red,midway,xshift=0cm, yshift=-0.45cm]{\footnotesize $tl(g)$};
  				\draw [thick,color=red,decorate,decoration={brace,amplitude=7pt},xshift=-4pt,yshift=0pt] (3.15,0.1)-- (4.6,0.1) node [color=red,midway,xshift=0cm, yshift=0.45cm]{\footnotesize $\le \frac{tl(g)}{3}$};

  				\filldraw[color=blue] (1,0) circle (1pt) node[color=blue, below left] {\footnotesize $A(h)$};
  				\filldraw[color=blue] (3.75,0) circle (1pt) node[color=blue,below right] {\footnotesize $A(h^g)$};
 
  			\end{tikzpicture}
  		\end{center}
  		\caption{The case in Lemma \ref{l: multiple conjugation} when $A(g)\cap A(h)\neq \emptyset$,  $g$ is hyperbolic and $h$ is elliptic. We stress that $A(h)$ is in general larger than a single point, as shown in the picture.} \label{f: section_4_2}
  	\end{figure}
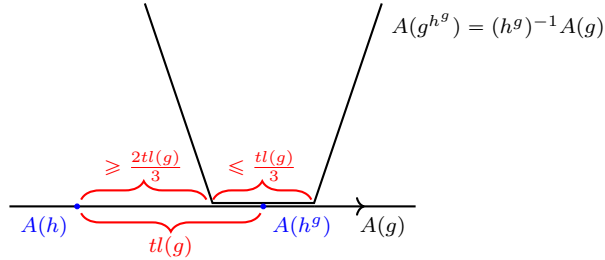
  
  		 Assume now $tl(g) < tl(h)$ and consider first the sub-case in which  $g$ is hyperbolic or $A(g)\cap A(h)$ is non-degenerate. Notice that the projection of $A(h)$ to $A(h^{g})$ ($pr_{A(h^g}(A(h))$) (in particular, $A(h)\cap A(h^{g})$ if $d(A(h),A(h^{g})=0$) is contained in $A(h^{g})\cap A(g)$. Since $| A(h^{g})\cap A(g) |\leq\frac{1}{2}tl(h)=\frac{1}{2}tl(h^{g})$, this implies by Remark \ref{r: alb_remark_11} that $A(c)=h^{g^{-1}}A(g)$ is disjoint from $A(g)\cup A(h)$. Moreover, both $A(g)$ and $A(h)$ project to the same point on $A(c)$. Another application of Remark \ref{r: alb_remark_11} yields $A(h^{c})\cap A(g)=\emptyset$.
  		The case in which $h$ is hyperbolic, $g$ elliptic and $A(g)\cap A(h)$ is a single vertex follows from a similar argument using the fact that $|A(h)\cap A(h^g)|\leq \frac{1}{3}tl(h)$ and is left to the reader. 
  		  This completes the proof that $tl(w^{cn}(g,h)) \geq \max\{tl(g), tl(h)\}$. 
  		
  		We next see that in all cases $g\notin E(h^c g) = E(w^{cn}(g,h))$. From  $A(g) \cap A(h^c) = \emptyset$ this is satisfied if $g$ is hyperbolic. Hence assume $g$ is elliptic, in which case we have $g \notin E^+(h^c g)$ (the set of elements that fix $A(h^c g)$ set-wise and preserve the orientation).  Assume towards contradiction that $g\in E(h^c g)\setminus E(h^c g)^{+}$. Observe that $\langle g,h^c g \rangle = \langle g, h^c \rangle = \langle h^c, h^c g\rangle$ preserves $A(h^c g)$. In particular, $h^c \in E(h^c g)$. We can assume that $h^c$ is elliptic, because otherwise we would have $A(h^c)= A(h^c g)$ a contradiction with Lemmas \ref{l: chiswell2} and \ref{l: chiswell3}. It follows that $g$ and $h^c$  both act as involutions on $A(h^c g)$. 
  		We now claim that $\langle g, h \rangle$  acts dihedrally on an infinite line (under the current assumptions that $g$ and $h^c$ are elliptic).  Once this is verified, the proof of the lemma is finished, since this claim contradicts the fact that $\langle g, h \rangle$ acts irreducibly on $T$ due to the definition of acylindricity.

  		We proceed to prove the claim. Let $J$ be the bridge between $A(h)$ and $A(g)$ (i.e.\ the path connecting these two subgraphs), and let $I$ be the bridge between $A(g)$ and $A(h^c)$ (see Figure \ref{fff_19}). Our goal is to apply Lemma \ref{l: lema dihedral action}, and so we need to prove that $g^2$ and $h^2$ fix $J$. By Remark \ref{l: chiswell2}, $I$ is contained in $A(h^c g)$. Assume we have proved that $g^{-1}J$ and $c^{-1}J$ are contained in $I$. Then since $g^2$ fixes $A(h^c g)$ (because $g\in E(h^c g) \setminus E^+(h^c g)$) we have $g^2(g^{-1}J)= gJ = g^{-1}J$ and so $g^2$ fixes $J$. Similarly, $h^2$ fixes $J$. Thus by Lemma \ref{l: lema dihedral action} we obtain that the action of $\langle g, h\rangle$ on $T$ is dihedral, a contradiction.

  		We now prove that $g^{-1}J$ and $c^{-1}J$ are contained in $I$. Indeed, since $A(h) \cap A(g) = \emptyset$ and both $h$ and $g$ are elliptic, we have that $A(h^{g})$ is disjoint from $A(g)$  and $A(h)$,  and the path between $A(h)$ and $A(h^g)$ crosses $A(g)$, by Remark \ref{r: alb_remark_1}. Similarly, $A(c)=h^{-1}{}^{g}A(g)$ is disjoint from $A(g)$ and the path from $A(g)$ to $A(c)$ crosses $A(h^g)$. Analogous arguments yield that the path between $A(g^c)$ to $A(h^c)$ is contained in $I$ (see Figure \ref{fff_19}). The former is precisely $c^{-1}J$, and so both $g^{-1}J$ and $c^{-1}J$ are contained in $A(h^c g)$.
  		  	\end{proof}
  		
  		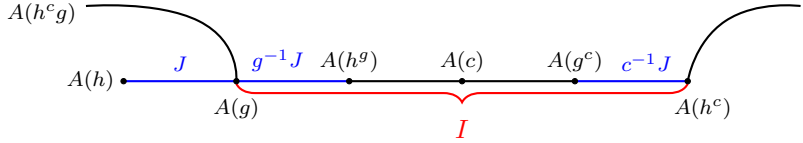
\begin{figure}[ht]
  			\begin{center}
  				\begin{tikzpicture}
  					\tikzset{near start abs/.style={xshift=1cm}}
  					
  					\draw[thick, color=blue] (1.5,0) -- (0,0) node[align=left,xshift=0cm, midway, above] {\footnotesize $J$};
  					\draw[color=blue, thick] (3,0) -- (1.5,0) node[align=left,xshift=-0.2cm, midway, above] {\footnotesize $g^{-1}J$};
  					\draw[thick] (4.5,0) -- (3,0);
  					
  					\draw[thick] (6,0) -- (4.5,0) node[align=left,xshift=0cm, midway, above] {\footnotesize {}};
  					\draw[thick, color=blue] (7.5,0) -- (6,0) node[align=left,xshift=0.2cm, midway, above] {\footnotesize $c^{-1}J$};
  					\draw[thick]  plot [smooth, tension = 1]  coordinates { (7.5,0) (8,0.8) (9,1)  };
  					
  					\draw [thick]  plot [smooth, tension=1] coordinates { (1.5,0) (1,0.8) (-0.5,1) } node[left, yshift=-0.1cm] {\footnotesize $A(h^c g)$};
  					 	\draw [color=red,thick,decorate,decoration={brace,amplitude=7pt}, yshift=-1pt] (7.5,0)-- (1.5,0)  node [midway,yshift=-0.6cm] {$I$};

  					\filldraw (0,0) circle (1pt) node[align=left,xshift=0cm, left] {\footnotesize $A(h)$};
  					\filldraw (1.5,0) circle (1pt) node[align=left,xshift=0cm, yshift=-0.1cm, below] {\footnotesize $A(g)$};
  					\filldraw (3,0) circle (1pt) node[align=left,xshift=0cm, above] {\footnotesize $A(h^g)$};
  					\filldraw (4.5,0) circle (1pt) node[align=left,xshift=0cm, above] {\footnotesize $A(c)$};
  					\filldraw (6,0) circle (1pt)
  					node[align=left,xshift=0cm, above]
  					{\footnotesize $A(g^c)$};
  					\filldraw (7.5,0) circle (1pt) node[align=left,xshift=0.2cm, yshift=-0.1cm, below] {\footnotesize $A(h^c)$};

  				\end{tikzpicture}
  			\end{center}
  			\caption{The axis of $A(h^cg)$ (in black), and the segments $J$ and $I$ in studied at the end of the proof of Lemma \ref{l: multiple conjugation}.} \label{fff_19}
  		\end{figure}

  	The next lemma shows  that there is a uniform way to combine a stable element with another one that does not belong to the normaliser of $g$ in order to obtain a small-cancellation tuple.

  	\medskip
  	
  	\begin{prop}
  		\label{l: from good pair to small cancellation}
  		
  		For all $N\geq 1$ and $m\geq 1$ there exists an $m$-tuple of words $w^{sc}\in\F(x,y)^{m}$ with the following property (`sc' stands for `small cancellation'). Suppose we are given an action of a group $G$ on a real tree $T$, a hyperbolic element $h\in G$, and another element $g\in G$ such that:
 
  		\enum{i)}{
  			\item $h$ is hyperbolic and stable (respectively, weakly stable),
  			\item $g\notin E(h)$,
  			\item \label{it: no cancellation}$A(g)$ and $A(h)$ intersect, $g$ does not invert a subsegment of $A(h)$ and if $g$ is hyperbolic, then $A(g)$ and $A(h)$ intersect coherently,
  			\item $tl(g)\leq tl(h)$.

  		}
  		Then $w(g,h)^{sc}$ is $N$-small cancellation (respectively, weakly $N$-small cancellation) as witnessed by any point in $A(g)\cap A(h)$, and for all $1\leq i\leq m$ we have \begin{equation}\label{e: from good pair eqn} tl(w_{i}^{sc}(g,h))\geq N \max\{tl(g),tl(h)\}=N tl(h).
  		\end{equation}

  	\end{prop}
  	
  	\begin{proof}
  		Take $K>20(m+2)N$. We denote by $w^{sc}_{i}(x,y)$ the word
  		$$
  		w^{sc}_i(x,y)=xy^{k_{i,1}}xy^{k_{i,2}}\cdots xy^{k_{i,K}},
  		$$
  		where $k_{i,j}=K^{2}+Ki+j$, for $1\leq i\leq m$ and $1\leq j\leq K$.
  		
  		Since by assumption, the axis $A(g)$ and $A(h)$ intersect, we let $*\in A(g)\cap A(h)$ be the  point in $A(g)\cap A(h)$ with minimal value with respect to the order induced on $A(g)$ by the translation action of $g$, i.e.\ $*$ is the `first' point where $A(g)$ and $A(h)$ meet.  The axis $A(w^{sc}_{i}(g,h))$ contains $A(g)\cap A(h)$  and so, in particular, $\ast \in A(w^{sc}_{i}(g,h))$. A picture of $A(w^{sc}_{i}(g,h))$ can be obtained by repeatedly drawing the figures from the preliminary lemmas \ref{l: chiswell2} and \ref{l: chiswell3}, and  Figure \ref{f: hyp_elliptic}.
  		
  		We  denote $w_{i,1} = g$, $w_{i,1}(h) = gh^{k_{i,1}}$, and
  		$$
  		w_{i,j} = w_{i,j-1}(h) g, \quad w_{i,j}(h) = w_{i,j} h^{k_{i,j}}, \quad j=2, \dots, K.
  		$$
  		We further let $[g]_{i,1} = [*, g*]$, $[h]_{i,1}= [g*, w_{i,1}(h)*]$, and
  		$$[g]_{i,j} = [w_{i,j-1}(h)*, w_{i,j}*], \quad  [h]_{i,j} = [w_{i,j}*, w_{i,j}(h)*] \quad j=2, \dots, K.
  		$$
  		Notice that $w_{i,j}*$ is the image of $w_{i,j-1}(h)*$ by $g^{w_{i,j-1}(h)^{-1}}$, and similarly, $w_{i,j}(h)*$ is the image of $w_{i,j}*$  by $w_{i,j}h^{k_{i,j}}{w_{i,j}}^{-1}$. Hence $|[g]_{i,j}| = tl(g)$ and $|[h]_{i,j}| = tl(h^{k_{i,j}}) = k_{i,j}tl(h)$.

  		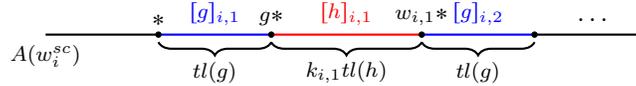
\begin{figure}[ht]
  			\begin{center}
  				\begin{tikzpicture}
  					\tikzset{near start abs/.style={xshift=1cm}}
  					\draw[thick] (1.5,0) -- (0,0) node[align=left,xshift=0cm, below] {\footnotesize $A(w_i^{sc})$};
  					\draw[thick, color=blue] (1.5,0) -- (3,0) node[align=left,xshift=0cm, midway, above] {\footnotesize $[g]_{i,1}$};
  					\draw[thick, color=red] (3,0) -- (5,0) node[align=left,xshift=0cm, midway, above] {\footnotesize $[h]_{i,1}$};
  					\draw[thick, color=blue] (5,0) -- (6.5,0) node[align=left,xshift=0cm, midway, above] {\footnotesize $[g]_{i,2}$};
  					\draw[thick] (6.5,0) -- (8,0);
  					\filldraw (1.5,0) circle (1pt) node[align=left,xshift=0cm, above] {\footnotesize $*$};
  					\filldraw (3,0) circle (1pt) node[align=left,xshift=0cm, above] {\footnotesize $g*$};
  					\filldraw (5,0) circle (1pt) node[align=left,xshift=0cm, above] {\footnotesize $w_{i,1}*$};
  					\filldraw (6.5,0) circle (1pt);
  					
  					\draw [thick,decorate,decoration={brace,amplitude=5pt},xshift=0pt,yshift=-0.1cm] (3,0) -- (1.5,0) node [midway,xshift=0cm, yshift=-0.4cm]{\footnotesize $tl(g)$};
  					\draw [thick,decorate,decoration={brace,amplitude=5pt},xshift=0pt,yshift=-0.1cm] (5,0) -- (3,0) node [midway,xshift=0cm, yshift=-0.4cm]{\footnotesize $k_{i,1} tl(h)$};
  					\node at (7.3,0.2) {$\dots$};
  					\draw [thick,decorate,decoration={brace,amplitude=5pt},xshift=0pt,yshift=-0.1cm] (6.5,0) -- (5,0) node [midway,xshift=0cm, yshift=-0.4cm]{\footnotesize $tl(g)$};
  					
  				\end{tikzpicture}
  			\end{center}
  			\caption{Part of the axis of $w_i^{sc}(g,h)$.} \label{ffff1}
  		\end{figure}
  		
  		We claim that $A(w_i^{sc}(g,h)*, *)$ contains the path obtained by concatenating $$[g]_{i,1}, [h]_{i,1}, \dots, [g]_{i,K}, [h]_{i,K},$$ and that there is no backtracking (i.e.\ 'cancellation') in the path $[g]_{i,1} \dots [h]_{i,K}$, that is
  		$$
  		d(w^{sc}_{i}(g,h)\cdot*,*)= tl(w^{sc}_{i}(g,h))
  		= K tl(g) + \sum\limits_{j=1}^{K} k_{ij} tl(h).
  		$$
  		Indeed, by assumption, if $g$ is hyperbolic then the axis of $g$ and $h^{k_{i, 1}}$ intersect coherently, and so by Lemma \ref{l: chiswell3}, $gh^{k_{i,1}}$ is hyperbolic and $tl(gh^{k_{i,1}})=tl(g) + tl(h^{k_{i,1}})$. Moreover given that $*\in A(g)\cap A(h^{k_{i,1}})$ we have that $[*, g*] \subseteq A(gh^{k_{i,1}})\cap A(h^{k_{i,1}})$ (see Figure \ref{f: fff} in Lemma \ref{l: chiswell3}). The claim then follows by induction on the length of the word $w_i^{sc}(g,h)$ using analogous arguments in each induction step.  If $g$ is elliptic, then since $g \notin E(h)$ and $g$ does not invert a subsegment of $A(h)=A(h^{k_{i,1}})$, we have $tl(gh^{k_{i,1}}) = tl(g) = tl(g) + tl(h^{k_{i,1}})$ due to Item 2 of Remark \ref{r: alb_remark_2}. Moreover  $g* = * \in A(gh^{k_{i,1}})\cap A(g)$ because $A(gh^{k_{i,1}})$ contains $A(g) \cap A(h)$. In this case the claim now follows by inductively combining this and the previous argument.

  		Since $tl(g)\le tl(h)$ (by assumption) and $1\le K\le K^2$, we have the follwing inequalities:
  		\begin{gather}\label{eq:<}
  			\begin{split}
  				d(w^{sc}_{i}(g,h)\cdot*,*)=& tl(w^{sc}_{i}(g,h))
  				= K tl(g) + \sum\limits_{j=1}^{K} k_{ij} tl(h)\\
  				< &(K+K(K^2+ mK +K))tl(h)<
  				(K^2(K+m+2))tl(h),
  			\end{split}
  		\end{gather}
  		and since $k_{i,j}\geq K^2+K$ for all $i,j$,
  		\begin{equation}\label{eq:>}
  			d(w^{sc}_{i}(g,h)\cdot*,*)= tl(w^{sc}_{i}(g,h))\ge K(K^2+K)tl(h)=K^2(K+1)tl(h).
  		\end{equation}
  		In particular, since $K\geq N$, \eqref{eq:>} implies the inequality \eqref{e: from good pair eqn} from the statement of the lemma.
  		
  		Using the above inequalities we now shows that
  		$(w^{sc}_{i}(g,h))_{i=1}^{K}$ satisfies the first two properties in the definition of small cancellation (Definition \ref{d: small cancellation}). Indeed, by the choice of $K$ and Equation (\ref{eq:>})  we have that $d(w^{sc}_{i}(g,h)\cdot*,*)>N$ for all $i$. We now establish inequality (b) from Definition \ref{d: small cancellation}, i.e.\ we prove that $\dis{w^{sc}_{i}(g,h)}\leq\frac{N+1}{N}\min_j tl(w^{sc}_{j})$ for all $i-$,. From Equations (\ref{eq:<}) and (\ref{eq:>}), we have, for all $i$,
  		$$
  		\begin{array}{l}
  			\frac{N+1}{N}\min_i(tl(w^{sc}_{i}(g,h)))\ge\frac{N+1}{N}K^2(K+1)tl(h);\\
  			d(w^{sc}_{i}(g,h)\cdot*,*)<(K^2(K+m+2))tl(h).
  		\end{array}
  		$$
  		Hence, it suffices to show that
  		$$
  		(K^2(K+m+2))tl(h)\le\frac{N+1}{N}K^2(K+1)tl(h).
  		$$
  		The latter is equivalent to $N(K+m+2)\leq (K+1)(N+1)$, which  holds by the choice of $K$, since $K> 20 (m+2)N > (m+1)N$.
  		
  		\medskip
  		
  		In order to check property (\ref{SCC}) of Definition \ref{d: small cancellation}, by \eqref{eq:>} it suffices to prove the following
  		
  		\begin{claim}\label{c: main_claim_section_4}
  			Let $1\leq i \leq m$, let $f\in G$, and let  $J$ be a subsegment of $A(w^{sc}_{i}(g,h))$ of length at least $\frac{1}{N}(K^{2}(K+1))tl(h)$ that is mapped by the action of  $f$ into a subsegment of $A(w^{sc}_{i'}(g,h))$, for some $1\leq i' \leq m$. Then
  			$i=i'$ and $f$  fixes  $A(w^{sc}_{i}(g,h))$ pointwise.  In case $h$ is stable then in fact we  have $f=1$.
  		\end{claim}
  		
  		If the claim is true, then whenever we have $|A(w_i^{sc}(g,h) \cap f A(w_{i'}^{sc}(g,h)|\geq \min_j\{tl(w_j^{sc}(g,h))\}$ we have $|A(w_i^{sc}(g,h) \cap f A(w_{i'}^{sc}(g,h)|\geq K^2(K+1)tl(h)$ due to \eqref{eq:>}, and then from the claim we obtain the required conclusions in Condition 3 of the definition of small cancellation. 
  		
  		\begin{proof}[Proof of Claim \ref{c: main_claim_section_4}]
  			Some preliminary observations are in order. Notice  that
  			\begin{align}\label{e: 19_sept_19}
  				&\max\limits_{i,j}(tl(h^{k_{ij}}))=(K^2+mK+K) tl(h) -2tl(h) <\\ &2 \min\limits_{i,j}(tl(h^{k_{ij}}))= 2(K^2+K+1)tl(h) -2tl(h),\nonumber
  			\end{align}
  			where the second inequality follows from $K> 20(m+2)N > 20m$.
  			Hence for all $1\leq i,i'\leq m$ and $1\leq j\leq K$ such that $f[h]_{i,j} \cap A(w_{i'}^{sc})$ is nonempty  there exists $1\leq j'\leq K$ such that
  			\begin{align}\label{e: 26_jan_19_0}
  				f[h]_{i,j} \cap A(w_{i'}^{sc}) &\subseteq [h]_{i',j'}[g]_{i',j'+1}[h]_{i',j'+1},\quad \text{or}\\
  				A(w_{i'}^{sc}) &\subseteq [g]_{i',j'}[h]_{i',j'}[g]_{i',j'+1}\nonumber
  			\end{align}
  			In this case,    since the translation length of $h^{k_{ij}}$ is at least $(K^2+K+1)tl(h)$ and $tl(g)\le tl(h)$, it follows that
  			\begin{gather}\label{e: 26_jan_19}
  				\begin{split}
  					&|f[h]_{i,j}\cap [h]_{i',j'}| \geq(K^2+K)tl(h)/2 \geq tl(h),  \quad \hbox{or} \\  &|f[h]_{i,j}\cap [h]_{i',j'+1}| \geq (K^2+K)tl(h)/2 \geq tl(h).
  				\end{split}
  			\end{gather}
  			We next prove that any segment of $A(w^{sc}_{i}(g,h))$ of length $\frac{1}{N}(K^2(K+1))tl(h)$ contains a translate of the path $[h]_{i,j-1}[g]_{i,j}[h]_{i,j}[g]_{i,j+1}[h]_{i,j+1}$, for some $j=1, \dots, K-1$ (by translate of a segment we mean the image of the action of an element on such segment). Indeed, in order to prove this, it suffices to show that $\frac{1}{N}(K^2(K+1))tl(h)$  is greater than $4 \max\limits_{i,j}(tl(h^{k_{i,j}}))+4tl(g)$. This follows from the inequality
  			\begin{gather}\notag
  				\begin{split}
  					4(K^2+mK+K)tl(h) + 4 tl(g) &\le  4(m+2)K^2tl(h) + 4tl(h) \le \\
  					\le & \left(\frac{1}{N} K^3 + 4\right) tl(h) \le  \frac{1}{N}(K^2(K+1))tl(h).
  				\end{split}
  			\end{gather}
  			Note that the inequality $4(m+2)K^2tl(h) \le \frac{1}{N} K^3 tl(h)$ holds by the choice of $K$, i.e. since $K> 20 (m+2)N > 4(m+2)N$.
  			
  			In views of the argument above, let $[h]_{i,j-1}[g]_{i,j}[h]_{i,j}[g]_{i,j+1}[h]_{i,j+1}$
  			be contained in the segment $J$. By \eqref{e: 26_jan_19}, there exists $1\leq j' \leq K$ such that  $$|f [h]_{i,j} \cap [h]_{i',j'}| \geq \frac{K^2+K}{2}tl(h)> tl(h).$$ We take $[h]_{i',j'}$ to be such that this intersection is as large as possible. Then it follows from \eqref{e: 19_sept_19} and from  \eqref{e: 26_jan_19} that, if $fA(w_i^{sc})$ and $A(w_{i'}^{sc})$ meet coherently,  then
  			$$|f[h]_{i,j+1} \cap [h]_{i',j'+1}| \geq  tl(h).$$ Similarly in this case we obtain $$|f[h]_{i,j-1} \cap [h]_{i',j'-1}| \geq  tl(h).$$
  			If $f(A(w_i^{sc}))$ and $A(w_{i'}^{sc})$ do not meet coherently then by analogous reasons we obtain $|f[h]_{i,j+\epsilon} \cap [h]_{i',j'-\epsilon}| \geq  tl(h)$ for $\epsilon\in\{-1,0,1\}$.
  			
  			Let $f_{0}=w_{i',j'}^{-1}fw_{i,j}$, $f_{-1} = w_{i',j'-1}(h)^{-1} f w_{i,j-1}(h)$, $f_1 = w_{i,j}(h)^{-1} f w_{i',j'}(h)$. Note that
  			\begin{equation}\label{e: 7june}
  				f_{-1} = f_0^{g}, \quad f_1 = h^{-k_{i',j'}}f_0h^{k_{i,j}}.
  			\end{equation}
  			The discussion above implies that each of the three elements $f_{0}$, $f_{-1}$, $f_1$ send a segment of the axis $A(h)$ of length at least $tl(h)$ to $A(h)$, as indicated in Figure \ref{f: 12_sept_19}.
  			
  			Assume first that the intersection $f[h]_{i,j}\cap[h]_{i',j'}$ is coherent. Then $f_{0}$, $f_{-1}$, $f_1$ send a segment of the axis $A(h)$ of length at least $tl(h)$ to $A(h)$ in an orientation preserving manner. Since $h$ is weakly stable, it follows that the three elements lie in $E^{+}(h)$ (i.e.\ they preserve $A(h)$ and its orientation). We now prove that  $f_{0}$ is elliptic. Indeed, if $f_0$ was hyperbolic we would have $A(f_0)=A(h)$, and thus $A(f_{0}^{g}) = g^{-1}A(f_0) =g^{-1}A(h)\neq A(h)$ because $g\notin E(h)$. But now the inequality $A(f_{0}^{g}) \neq A(h)$  contradicts the fact that $f_0^{g}$  preserves $A(h)$. This proves that $f_0$ is elliptic, which immediately implies that $f_{-1}=f_0^{g}$ is elliptic as well
  			
  			It follows that $f_0,f_{-1}\in K(A(h))$, i.e.\ $f_0$ and $f_{-1}$ act as the identity on $A(h)$ (since they are elliptic and preserve the orientation of $A(h)$). Notice that if $h$ was chosen to be stable, then this by itself implies that $f_{0}=f_{-1}=1$.
  			
  			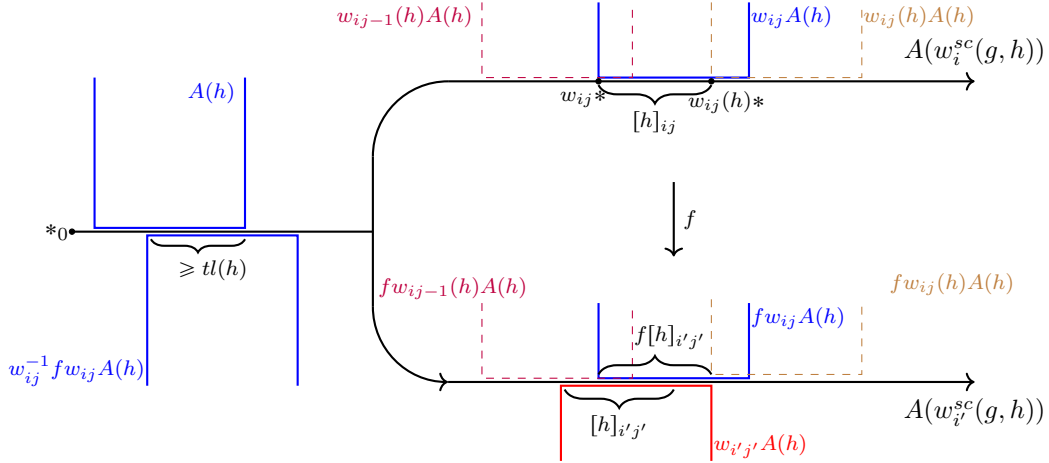
\begin{figure}[ht]
  				\begin{center}
  					\begin{tikzpicture}
  						\tikzset{near start abs/.style={xshift=1cm}}

  						\draw[->, thick] (5,2) -- (12,2);
  						\draw[thick] (5,2) arc[radius=1,start angle=90,end angle=180];
  						\draw[thick] (4,1) -- (4,-1);
  						\draw[->,thick] (4,-1) arc[radius=1,start angle=180,end angle=270];
  						\draw[->,thick] (5,-2) -- (12,-2);
  						
  						\draw[thick] (0,0) -- (4,0);
  							\filldraw (0,0) circle (1pt);
  						\node at (-0.2,0) {$*_0$};
  		
  						\node at (12,  -2.4) {$A(w_{i'}^{sc}(g,h))$};
  						\node at (12,  2.4) {$A(w_{i}^{sc}(g,h))$};
  										
						\draw[color=blue, thick, xshift=0.3cm, yshift=0.05cm] (0,2) -- (0,0) -- (2,0) -- (2,2) node[left, yshift=-0.2cm] {\footnotesize $A(h)$} ;
  						
  						\draw[color=blue, thick, xshift=1cm, yshift=-0.05cm] (0,-2) -- (0,0) -- (2,0) -- (2,-2) node[left, yshift=0.2cm, xshift=-1.9cm] {\footnotesize $w_{ij}^{-1}fw_{ij}A(h)$};
  						
  						\draw[color=blue, thick, xshift=7cm, yshift=2.05cm] (0,1) -- (0,0) -- (2,0) -- (2,1)  node[right, yshift=-0.2cm, xshift=-0.1cm] {\footnotesize $w_{ij}A(h)$};
  						
  						\draw[color=blue, thick, xshift=7cm, yshift=-1.95cm] (0,1) -- (0,0) -- (2,0) -- (2,1) node[right, yshift=-0.2cm, xshift=-0.1cm] {\footnotesize $fw_{ij}A(h)$};
  						
  						\draw[color=red, thick, xshift=6.5cm, yshift=-2.05cm] (0,-1) -- (0,0) -- (2,0) -- (2,-1) node[right, yshift=0.2cm, xshift=-0.1cm] {\footnotesize $w_{i'j'}A(h)$};
  						\draw[color=purple, dashed, xshift=5.45cm, yshift=2.05cm] (0,1) -- (0,0) -- (2,0) -- (2,1) node[left, yshift=-0.2cm, xshift=-2cm] {\footnotesize $w_{ij-1}(h)A(h)$};
  						\draw[color=brown, dashed, xshift=8.5cm, yshift=2.05cm] (0,1) -- (0,0) -- (2,0) -- (2,1) node[right, yshift=-0.2cm, xshift=-0.1cm] {\footnotesize $w_{ij}(h)A(h)$};
  						
  						\draw[color=purple, dashed, xshift=5.45cm, yshift=-1.95cm] (0,1) -- (0,0) -- (2,0) -- (2,1) node[right, yshift=0.2cm, xshift=-3.5cm] {\footnotesize $fw_{ij-1}(h)A(h)$};
  						\draw[color=brown, dashed, xshift=8.5cm, yshift=-1.9cm] (0,1) -- (0,0) -- (2,0) -- (2,1) node[right, yshift=0.2cm, xshift=0.25cm] {\footnotesize $fw_{ij}(h)A(h)$};
  						\draw [thick,decorate,decoration={brace,amplitude=5pt},xshift=0pt,yshift=0pt] (2.3,-0.1)-- (1.05,-0.1) node [midway,xshift=0.2cm, yshift=-0.4cm]{\footnotesize $\ge tl(h)$};
  						
  						\filldraw (7,2) circle (1pt) node[align=left,xshift=-0.2cm, below] {\footnotesize $w_{ij}*$};
  						\filldraw (8.5,2) circle (1pt) node[align=left,xshift=0.2cm,below] {\footnotesize $w_{ij}(h)*$};
  						
  						\draw [thick,decorate,decoration={brace,amplitude=7pt},xshift=0pt,yshift=-0.05cm] (8.5,2)-- (7,2) node [midway,xshift=0cm, yshift=-0.5cm]{\footnotesize $[h]_{ij}$};
  						\draw [thick,decorate,decoration={brace,amplitude=7pt},xshift=0cm,yshift=-0.1cm] (8,-2)-- (6.55,-2) node [midway,xshift=0cm, yshift=-0.5cm]{\footnotesize $[h]_{i'j'}$};
  						\draw [thick,decorate,decoration={brace,amplitude=7pt},xshift=0cm,yshift=0.1cm] (7,-2) --(8.5,-2) node [midway,xshift=0.2cm, yshift=0.5cm]{\footnotesize $f [h]_{i'j'}$};
  						
  						\draw[->, thick] (8,0.66) -- (8,-0.33) node[midway, right] {\footnotesize $f$};
  					\end{tikzpicture}
  				\end{center}
  				\caption{Auxiliary picture for the last arguments in the proof of Proposition \ref{l: from good pair to small cancellation}.  We show that $|[h]_{i',j'} \cap  f[h]_{i,j}| \geq tl(h)$, from where it follows that $|A(h) \cap w_{ij}^{-1}fw_{ij}A(h) | \geq tl(h)$ } \label{f: 12_sept_19}
  			\end{figure}
  			
  			Since $f_{0}\in K(A(h))$, the element $f_{1}=(h^{-k_{i',j'}}f_{0}h^{k_{i,j}})$  can only preserve $A(h)$ if
  			$k_{i',j'}=k_{i,j}$, because since $f_0$ fixes the axis of $h$ pointwise,  $f_1$ acts by translation on $A(h)$ with translation length $h^{k_{i,j}-k_{i',j'}}$, and since $h$ is weakly-stable the equality follows. By construction, $k_{i,j}=k_{i',j'}$ implies $i=i'$ and $j=j'$. Then the fact that $f_{0}$ fixes a subsegment of $A(h)$ implies its conjugate $f=f_{0}^{w^{sc}_{i}(g,h)^{-1}}$ fixes a subsegment of $A(w^{sc}_{i}(g,h))$. Hence the claim is proved in case the intersection $f[h]_{i,j}\cap[h]_{i',j'}$ is coherent.  We next prove that such intersection is always coherent.
  			
  			\renewcommand{\k}[0]{k_{i,j}}
  			\newcommand{\kk}[0]{k_{i',j'}}
  			Suppose towards contradiction that the intersection $f[h]_{i,j-1}\cap[h]_{i',j'+1}$ is not coherent.
  			In this case $|f[h]_{i,j+\epsilon} \cap [h]_{i',j'-\epsilon}|\geq tl(h)$ for $\epsilon\in\{0,1,2\}$. Just as in the first case, weak stability of $h$ allows us to conclude that each of the elements
  			\begin{align*}
  				f_{0},\,\,\,g^{-1}h^{-\kk}f_{0}g^{-1},\,\,gf_{0}h^{\k}g
  			\end{align*}
 
  			belong to $E^{-}(h)$. From this we conclude that
  			\begin{align*}
  				c=(h^{-\kk}f_{0}^{2}h^{\k})^{g}=(g^{-1}h^{-\kk}f_{0}g^{-1})(gf_{0}h^{\k}g)\in E(h).
  			\end{align*}
  			Notice that $f_{0}^{2}\in K(A(h))$. Just as above, $\k\neq\kk$ implies $c$ is hyperbolic, which together with $g\nin A(h)$ contradicts $f_{1}\in A(h)$. Thus $\k=\kk$ and hence $(i,j)=(i',j')$. Now, the exact same argument applied to $(i,j-1)$ yields
  			$k_{i,j-1}=k_{i',j'+1}$, contradicting $(i,j)=(i',j')$.  Hence, the intersection $f[h]_{i,j-1}\cap[h]_{i',j'+1}$ had to be coherent to begin with.
  		\end{proof}
  		
  		As explained previously, Claim \ref{c: main_claim_section_4} together with our previous arguments imply that  the tuple $w^{sc}$ is weakly $N$-small cancellation. If additionally $h$ is stable, then Claim \ref{c: main_claim_section_4} ensures that $f=1$  and thus the tuple $w^{sc}$ is $N$-small cancellation.
  		
  	\end{proof}
  	
  	\begin{cor}\label{cor:uniformWSC}
  		For all $N,m \geq 1$, there exist two $m$-tuple of words $W^{sc}, W^{sc}{}'\in\F(x,y)^{m} $ with the following property. Suppose we are given an action of a group $G$ on a tree $T$, two elements $g,h$ such that the subgroup $\subg{g,h}$ acts irreducibly on $T$, $h$ is hyperbolic and stable (resp. weakly stable) and $d(A(h),A(g))< tl(h)$. Then either $W(g,h)^{sc}$ or $W(g,h)^{sc}{}'$ is $N$-small cancellation (respectively, weakly $N$-small cancellation). 
  	\end{cor}
  	
  	\begin{proof}
  		Let $W^{sc}=w^{sc}(yy^x, y^4)$ and $W^{sc}{}'=w^{sc}(y (y^{-1})^x, y^4)$ the two $M$-tuples, where $w^{sc}(x,y)$ is the $m$-tuple given in Lemma \ref{l: from good pair to small cancellation}.
  		We show that given $g$ and $h$ as in the statement of the Corollary, either the pair $g'=hh^g$, $h'=h^4$ or $g'=h (h^{-1})^g$, $h'=h^4$ satisfy the requirements of Lemma \ref{l: from good pair to small cancellation}.
  		
  		Let $h$ be hyperbolic and (weakly)-stable. Since by assumption $g,h$ acts irreducibly, we have that $g \notin E(h)$.

  		Assume first that $A(h)$ and $A(g)$ intersect. In this case, $h^g$ is hyperbolic and
  		\begin{itemize}
  			\item (i) either $A(h)$ and $A(h^g)$ are disjoint; or
  			\item (ii) $A(h^g) \cap A(h) \neq \emptyset$ and $A( (h^{\epsilon})^g)$ and $A(h)$ intersect consistently, for $\epsilon \in \{\pm 1\}$.
  		\end{itemize}
  		In case (ii), by Lemma \ref{l: chiswell3} it suffices to take $h'=h^4$ and $g'=h(h^{\epsilon})^g$ since $g'$ is hyperbolic, $g'\notin E(h)$ and $tl(g')=2tl(h) \le 4tl(h)=tl(h')$.
  		
  		If $A(h)$ and $A(g)$ are disjoint, by Remark \ref{r: alb_remark_1}, so are $A(h)$ and $A(h^g)$. In this case and in the Alternative (i), by Lemma \ref{l: chiswell2}, it suffices to take $h'=h^4$ and $g'=h h^g$ since $tl(g')=2tl(h) +2d(A(h),A(g)) < 4tl(h)=tl(h')$.
  		
  		Hence, the pair $h'$ and $g'$ satisfies the requirements of Lemma \ref{l: from good pair to small cancellation}, we have that $W^{sc}(g,h)=w^{sc}(g',h')$ or $W^{sc}{}'(g,h)=w^{sc}(g',h')$  is (weakly) $N$-small cancellation. 
  	\end{proof}
  	
  	 \begin{remark}\label{r: stable_and_sc}
  	 
  	 It follows from the previous corollary that a group $G$ acting irreducibly on a tree has   stable elements if and only if it has an $m$-tuple of small cancellation elements for any $m\geq 1$.  
  	 
  	 Indeed,  let $h$ be a  $\lambda$-stable element of $G$. Let $g\in G$ be such that $\langle g, h \rangle$ acts irreducibly on $T$ and let $d=d(A(h),A(g))$. Then $h^d$ is a stable element and $h^d$ and $g$ satisfy the requirements of Corollary \ref{cor:uniformWSC}. Hence, $G$ has small cancellation elements. Assume now that $G$ has an $m$-tuple of $N$-small cancellation elements for all $m\geq 1$. Let $m=1$ and let $a$ be an $N$-small cancellation element with respect to a basepoint $*$. By Remark \ref{r: alb_remark_3}, $a^{-1}$ is also $N$-small cancellation. Then $a$ is $\frac{2}{N}$-stable. Indeed, if $|A(a) \cap g A(A)| \ge \frac{2}{N} tl(a)$, then either $|[*, a*] \cap g[*,a*]|$ or $|[*, a^{-1}*] \cap g[*,a^{-1}*]|$ is greater than or equal to $\frac{1}{N}$ and since both $a$ and $a^{-1}$ are $N$-small cancellation, we deduce that $g=1$ and so in particular $g\in E(a)$; hence $a$ is $\frac{2}{N}$-stable.
    \end{remark}

	\section{Uniform small cancellation over $c$ from stable elements} \label{sec:uniform small cancellation}

In Section \ref{sec:formal solutions}, we determined a sufficient condition for the existence of formal solutions relative to a set of diophantine conditions for $\forall \exists$-formulas when the group acts on a tree, see Theorem \ref{l: small_cancellation_parameters}. This condition requires the existence of small cancellation $m$-tuples that dominate a parameter $c$.

In the previous section, we proved that the existence of a (weakly) stable element in a group acting irreducibly on a tree is sufficient to obtain in a uniform way (weak) small cancellation $m$-tuples. Similarly, in this section, we describe in a uniform way how to obtain small cancellation tuples over $c$ given two elements $a$ and $b$ of a group $G$ acting on a tree such that the subgroup $\langle a, b\rangle$ has an irreducible action and with the assumption that hyperbolic elements of $G$ are also stable, see Theorem \ref{l: general small cancellation}.

This uniform description of small cancellation elements will be essential to prove the quantifier reduction in Section \ref{sec:quantifier reduction}, as well as in the preservation of the non-trivial positive theory under extensions in Section \ref{sec:preservation extensions}. 

\bigskip

  	We begin with two preliminary lemmas.
  	
  	\begin{lemma}
  		\label{l: product of conjugates}Suppose that $g$ is an $N$-small cancellation element, with $N\geq 1$. 
  		Then for any non-zero integer $k$ and any $c\in G$, the element $g^{k}(g^{k})^{c}$ is hyperbolic with
  		\begin{align}\label{e: product_of_conjugates}
  			(2k-1)tl(g)+d(A(g),c A(g))\leq tl(g^{k}(g^{k})^{c})\leq 2k\,tl(g)+d(A(g),c A(g))
  		\end{align}
  	\end{lemma}
  	\begin{proof}
  		Denote $g_{0}=g^{k}$. By Remark \ref{r: alb_remark_3}, $g$ is hyperbolic, and thus so is $g_0$, $g_0^c$ and $g_0 g_0^c$.
  		Now $tl(g_0)=k\cdot tl(g)$,  $A(g_0)= A(g)$, and $A(g_0^c) = A(g^c)= c^{-1}A(g)$. If the intersection  $A(g)\cap A(g^{c})$ is empty, then  by Remark \ref{r: alb_remark_2}, $$
  		tl(g_{0}g_{0}^{c})=2tl(g_{0})+d(A(g_0),A(g_0^{c})) =2k tl(g)+d(A(g),A(g^{c}))
  		$$
  		and so \eqref{e: product_of_conjugates} holds.
  		By Remark \ref{r: alb_remark_2}, if $A(g)$ and $A(g^{c})$ intersect coherently, then $tl(g_{0}g_{0}^{c})= 2k (tl(g))$ and  $d(A(g), A(g^c)) = 0$, hence \eqref{e: product_of_conjugates} holds as well.
  	Suppose $A(g)$ and $A(g^{c})$ have non-empty intersection and this intersection is non-coherent. Since $g$ is $N$-small cancellation Remark \ref{r: alb_remark_3}  ensures that $|A(g)\cap A(g^{c})| = |A(g) \cap c^{-1}A(g)| \leq \frac{tl(g)}{2N}\leq \frac{tl(g)}{2}$, unless $c \in \langle g \rangle$, in which case $A(g^c)=A(g)$ and the intersection $A(g)\cap A(g^c)$ is coherent. Now by Remark \ref{r: alb_remark_2} \begin{align*}tl(g_{0}g_{0}^{c})=&2tl(g_{0})-2|A(g_0)\cap A(g_0^{c})| =\\ &2k(tl(g))-2|A(g)\cap A(g^{c})| \geq 2k(tl(g)) - tl(g),\end{align*} and \eqref{e: product_of_conjugates} follows since in this case $d(A(g),cA(g)) = 0$.
  	\end{proof}

  	\begin{lemma}
  		\label{l: simple domination}
  		Let $x,y,z$ be variables and define $w_1^{dm}=xx^{z}$, $w_2^{dm} = yy^{z}$,  $w^{dm}=(w_1^{dm}, w_2^{dm}) = (xx^{z},yy^{z})\in\F(x,y,z)^{2}$ (`dm' stands for `domination'). Then there is some $N$ such that the following holds.
  		Let $G$ be a group acting on a real tree $T$, $g,h,c\in G$, and let $*\in T$ be such that the pair $(g,h)$ is $N$-small cancellation over $c$ with respect to $*$. Denote $a=(w^{dm}_{1}(g,h))^2$, $b=(w^{dm}_{2}(g,h))^2$. Then both $a$ and $b$ are hyperbolic, $A(a)\neq A(b)$,  and there exists $\gamma\in \{a, b\}$ such that $tl(\gamma)\geq\max\{d(c\cdot*,*),tl(g),tl(h)\}$.
  	\end{lemma}
  	
  	\begin{proof}
        Let $(g,h)$ be an $N$-small cancellation pair  over $c \in G$ with respect to a base point $*\in T$, and consider the elements $a=(gg^c)^2$ and $b=(hh^c)^2$. These elements are hyperbolic due to Lemma \ref{l: product of conjugates}. 
  		
  		Write $*'=c^{-1}*$ and let $J=[*,*']$. First we want to check that
  		\begin{align}\label{e: eqn_simple_domination}
  			|J|= d(*,*')\leq	2\max\{tl(g),tl(h)\},
  		\end{align}
  		since if such inequality holds then to prove the lemma it suffices to show that $A(a)\neq A(b)$ as the inequalities $\max\{tl(a),tl(b)\}\geq \max\{2tl(g), 2tl(h)\} \geq \max\{d(*, *'), tl(g), tl(h)\}$ are guaranteed by the previous Lemma \ref{l: product of conjugates} and by \eqref{e: eqn_simple_domination}, respectively.

  		Let $M = \max\{tl(g),tl(h)\}$ and let $L=\min\{tl(g), tl(h)\}$. Assume that $|J|\geq 2M$, and let $N'=N/4$.  By Remark \ref{r: alb_remark_3}, $(a,b)$ is $N'$-small cancellation over $c$, and by this same remark we have $d(*,A(g))\leq\frac{L}{2N'}$ and also $d(*,A(h))\leq\frac{L}{2N'}$. Hence $d(*',A(g^c))\leq\frac{L}{2N'}$ and $d(*',A(h^c))\leq\frac{L}{2N'}$, since $d(*',A(g^c)) = d(c^{-1} *, c^{-1}A(g))$, and similarly for the second inequality.
  		
  		It follows (see Figure \ref{f: 14_oct_19}) that there exists some subsegment $J'\subset J$ of length at least $|J|-\frac{L}{N'}$ which is entirely contained in the convex closure of
  		$A(g)\cup A(g^{c})$ as well as entirely contained in the convex closure of $A(h)\cup A(h^{c})$. This segment can be constructed explicitly as
  		$$
  		J'=J \cap (A(g)\cup A(g^c) \cup b(g,g^c))
  		\cap (A(h)\cup A(h^c) \cup b(h,h^c)),
  		$$
  		where $b(g,g^c)$ (resp. $b(h,h^c)$) is the bridge between $A(g)$ and $A(g^c)$ (resp. $A(h)$ and $A(h^c)$).
  		
  		Given $\alpha \in\{g,h\}$ we can subdivide
  		$J'$ in three consecutive subsegments $J^{\alpha}_{1},J^{\alpha}_{2},J^{\alpha}_{3}$, where $J^{\alpha}_{1}\subset A(\alpha),J^{\alpha}_{3}\subset A(\alpha^{c})$ and $J^{\alpha}_{2} \subset b(\alpha,\alpha^c)$ in case $d(A(\alpha),A(\alpha^{c}))>0$ and $J^{\alpha}_{2}$ is degenerate otherwise. The segment $J_1^\alpha$ is degenerate if  the intersection $A(\alpha) \cap J'$ is degenerate, and similarly for $J_3^\alpha$.
  		
  		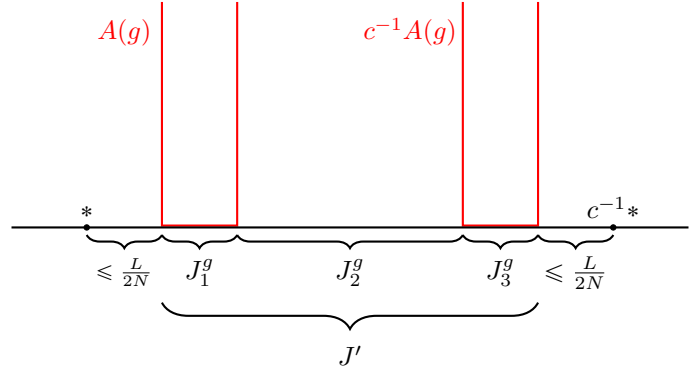
\begin{figure}[ht]
  			\begin{center}
  				\begin{tikzpicture}
  					\tikzset{near start abs/.style={xshift=1cm}}
                \filldraw (0,0) circle (1pt) node[above] {$*$};
                \filldraw (7,0) circle (1pt) node[above] {$c^{-1}*$};
                \draw[color=red, thick] (1,3) -- (1,0.025) -- (2,0.025) -- (2,3);
                \draw[color=red, thick] (5,3) -- (5,0.025) -- (6,0.025) -- (6,3);
                \node[color=red] at (0.5, 2.6) {$A(g)$};
                \node[color=red] at (4.3, 2.6) {$c^{-1}A(g)$};
                \draw[thick] (-1,0) -- (8,0);
  					\draw [thick,decorate,decoration={brace,amplitude=5pt}, ,yshift=-0.1cm]
  					(1,0) -- (0,0) node [black,midway, yshift=-0.5cm]
  					{\footnotesize $\le \frac{L}{2N}$};
  						\draw [ thick,decorate,decoration={brace,amplitude=5pt}, yshift=-0.1cm]
  					 (2,0) -- (1,0) node [black,midway, yshift=-0.5cm]
  					{$J_1^g$};
  					\draw [ thick,decorate,decoration={brace,amplitude=5pt}, ,yshift=-0.1cm]
  					(5,0) -- (2,0) node [black,midway, yshift=-0.5cm]
  					{$J_2^g$};
  					\draw [ thick,decorate,decoration={brace,amplitude=5pt}, ,yshift=-0.1cm]
  					(6,0) -- (5,0) node [black,midway, yshift=-0.5cm]
  					{$J_3^g$};
  					\draw [ thick,decorate,decoration={brace,amplitude=5pt}, ,yshift=-0.1cm]
  					(7,0) -- (6,0) node [black,midway, yshift=-0.5cm]
  					{$\leq \frac{L}{2N}$};
  					\draw [ thick,decorate,decoration={brace,amplitude=10pt}, yshift=-1cm]
  					(6,0) -- (1,0) node [black,midway, yshift=-0.7cm]
  					{$J'$};
  				\end{tikzpicture}
  			\end{center}
  			\caption{We take $J'\subseteq J$ to be the largest segment contained in the convex closures of $A(g) \cup A(g^c)$ and of $A(h) \cup A(h^c)$ (the latter is not shown in the figure). In this picture we have $J_1{}^g=\emptyset$, $J_2{}^g = J'$, and $J_3{}^g =\emptyset$. } \label{f: 14_oct_19}
  		\end{figure}
  		
  		Assume we have proved there exists $\alpha\in\{g,h\}$ such that $|J^{\alpha}_{2}|>\frac{|J|}{2}-L$. Then from our assumption $|J|\geq 2M$ we have that $|J^{\alpha}_{2}|> 0$ and so $J^{\alpha}_2$ is non-degenerate. Moreover from $d(A(\alpha), A(\alpha^c)) \geq |J_2^\alpha|$ we obtain  $A(\alpha)\cap A(\alpha^c)=\emptyset$. We now have  \begin{align*}
  			&tl(\alpha \alpha^{c})\geq  2d(A(\alpha),A(\alpha^{c}))+  2tl(\alpha)\\
  		\geq &|J|  -2L + 2tl(\alpha) \geq |J|\geq \max\{d(*, c*), tl(g), tl(h)\},\end{align*}
  		where  the first inequality is due to Lemma \ref{l: chiswell2};  the second last inequality follows from $tl(\alpha) \geq L$;  and the last inequality is due to our assumption that $|J|=d(*, c*) \geq 2M$. Thus in this case we will have proved the last part of the lemma.

  		In views of this discussion, our next goal is to show that there exists $\alpha\in\{g,h\}$ such that $|J^{\alpha}_{2}|\geq\frac{|J|}{2}-L$.
  		Assume for the sake of contradiction that
  		$|J^{g}_{2}|,|J^{h}_{2}|<\frac{1}{2}|J|-L$. This implies that the set $X=J'\setminus (J^{g}_{2}\cup J^{h}_{2})$
  		has size at least
  		$$
  		|X|\ge |J'|-|J^{g}_2|-|J^{h}_2| \ge \left(|J|-\frac{L}{N'}\right) - 2\left(\frac{|J|}{2}-L\right) = \frac{L(2N'-1)}{N'}.
  		$$
  		Denote $\mc{J}= \{J_1^{g} \cup J_3^{g} \cup J_1^h \cup J_3^h\}$. Since   $X =\cup_{I \in \mc{J}} I$, one of the segments from $\mc{J}$ must have size at least $\frac{2L(N'-1)}{4N'} = \frac{}.$
  		Furthermore, given that $J'\subset (A(g)\cup A(g^c) \cup b(g,g^c))
  		\cap (A(h)\cup A(h^c) \cup b(h,h^c))$, by the definition of $X$ we have that
  		$$
  		X\subset (A(g)\cup A(g^c))
  		\cap (A(h)\cup A(h^c))
  		$$
  		and so each segment from $\mc{J}$ is contained in a set of the form $A(g^{c^{\epsilon}})\cap A(h^{c^{\delta}})$, where $\epsilon,\delta\in\{0,1\}$, and one of them has length greater than or equal to  (recall that $N'=N/4$) $$\frac{L(2N'-1)}{4N'}\geq \frac{L}{2N'}\geq \frac{L}{N/2}.$$
  		Since $(a,b)$ is $N$-small cancellation this contradicts Item 4 of Remark \ref{r: alb_remark_3}.
 
  		Hence we have proved that $tl(d_{i})\geq\max\{d(c\cdot*,*),tl(g),tl(h)\}$ for some $i=1,2$.
  		
  		All is left is to show is that, for $N$ big enough, the two elements $gg^{c}$ and $hh^{c}$ have different axis. Assume this was not the case. Then the commutator element $[gg^{c},hh^{c}]$ would fix the axis $A(gg^c)= A(hh^c)$ point-wise.

  		Since $|A(g)\cap A(gg^c)| =tl(g) \geq \frac{L}{N}$, we have that $[gg^{c},hh^{c}]$ fixes a segment of $A(g)$ of length greater than $\frac{L}{N}$ and so the $N$-small cancellation property implies that $[gg^{c},hh^{c}]=1$. Now, consider the sentence:
  		\begin{align*}
  			\forall x\forall y\exists z\,\,[xx^{z},yy^{z}]=1.
  		\end{align*}
  		This sentence is false in a non-abelian free group $\mbb{F}$. Indeed take $x$ and $y$ to be any two elements $\alpha, \beta$ in a  basis $S$ of $\mbb{F}$. For  $\alpha\alpha^{z}$ and $\beta\beta^{z}$ to commute they must have some common non-trivial power. However, the epimorphism onto an infinite cyclic group given by trivializing all elements in $S$ except $\alpha$ sends $\beta\beta^z$ to $1$ and $\alpha\alpha^{z}$ to a non-trivial element, and so there is no $z$ such that $\alpha\alpha^{z}$ and $\beta\beta^{z}$ commute. By virtue of Theorem \ref{l: small cancellation solutions}, this means that there exists some $N>0$ such that if $(g,h)$ is a weakly $N$-small cancellation pair, then there exists no  $c\in G$ such that $[gg^c, hh^c]=1$. We conclude that the axes $A(gg^c)$ and $A(hh^c)$ are different.
  	\end{proof}
  	
  	\begin{thm}
  		\label{l: general small cancellation} Given positive integers $k,N,m$ and a real number $\lambda\geq 3$, there is a finite collection $\mathcal{W}\subseteq\F(x,y,z_{1},z_{2},\dots z_{k})^{m}$ of $m$-tuples of words in the normal closure of $x,y$ with the following property:
  		
  		Suppose we are given a group $G$ and a tree $T$ where $G$  acts in a way such that every hyperbolic element of $G$ is $\lambda$-stable. Let $(a,b)\in G^{2},c\in G^{k}$ be such that $\subg{a,b}$ acts irreducibly on $T$ and both $a$ and $b$ have roots of order $\geq 10\lambda$.
  		
  		Then there is $u\in\mathcal{W}$ such that the tuple $u(a,b,c)$ is $N$-small cancellation over $c$. 
  	\end{thm}
  	
  	\begin{proof}
  		We will not make this explicit in our notation, but all words constructed in the course of the proof depend on the given parameters $N$ and $\lambda$. 
  		
  		First we treat the case $k=0$. We know that $a$ and $b$ admit roots  $a_{0}$ and $b_{0}$ of order $s,t\geq 10\lambda$, respectively. Note that $a$ is hyperbolic if and only if $a_0$ is hyperbolic, hence by the hypothesis of the lemma, either $a_0$ is  elliptic or $\lambda$-stable.  The same considerations apply for $b_0$. Thus the hypotheses of Lemma \ref{l: acylindrical pairs} apply to $(a,b)$ and so $(a,b)$ is an acylindrical pair.
  		
  		We can now apply Lemma \ref{l: multiple conjugation} to the pair $(a,b)$. This ensures that the element  $w^{cn}(a,b)$ is hyperbolic with translation length larger than $tl(a)$ and $tl(b)$; that $a \notin E(w^{cn}(a,b))$, and that $A(a)$ intersects with $A(w^{cn}(a,b))$ coherently in case $a$ is hyperbolic, where $w^{cn}$ is the word given in Lemma \ref{l: multiple conjugation}. Denote $d_1=w^{cn}(a,b)$. By our assumptions, $d_1$ is $\lambda$-stable, and since the subgroup $\langle d_1^{\ceil{10\lambda}}, a \rangle$  acts irreducibly on $T$, we can apply again Lemma \ref{l: acylindrical pairs} to the tuple $(d_1^{\ceil{10\lambda}},a)$, obtaining thus that $(d_1^{\ceil{10\lambda}},a)$ is an acylindrical pair. We can now apply  Lemma \ref{l: multiple conjugation} to the pair $(d_1^{\ceil{10\lambda}},a)$. Thus by letting $d_2=w^{cn}(d_1^{\ceil{10\lambda}},a)$, we have that both $d_1$ and $d_2$ are hyperbolic with $tl(d_2)\geq tl(d_1)$
  		and $d_1\nin A(d_2)$ and $A(d_1)$ and $A(d_2)$ intersect coherently.
  		
  		Again, since $d_2$ is hyperbolic, by our assumptions it is $\lambda$-stable. Thus the tuple $(d_1,d_2)$ satisfies the hypothesis of Proposition  \ref{l: from good pair to small cancellation}. This ensures that $w^{sc}(d_1,d_2)$ is $N$-small cancellation  with respect to any base point $*$ in $A(d_1) \cap A(d_2)$, where similarly as before,  $w^{sc}$ is the $m$-tuple of words provided by Proposition  \ref{l: from good pair to small cancellation} (note that the condition in Proposition  \ref{l: from good pair to small cancellation} about inversion of segments is automatically met since a hyperbolic element does not invert any segment). 
  		
  		In sum, for the case $k=0$ one can take   the tuple of words $$w^{sc}(w^{cn}(x,y),w^{cn}(w^{cn}(x,y),y)).$$ We denote this tuple by $u^{N,0}(x,y)$. It is clear that in this case, $u^{N,0}(x,y)$ are in the normal closure of $x$ and $y$. This completes the proof for the case $k=0$.
  		
  		Let us now consider the general case $k\geq 0$. We need to find an $m$-tuple of words $w(x,y,z)$ in the normal closure of $x$ and $y$ such that $w(a,b,c)$ is $N$-small cancellation and
  		\begin{equation}\label{e: sc_condition_k>0}
  			d(*, w_i(a,b,c) *) > N d(*, c_j*)
  		\end{equation}
  		for all $1\leq i \leq m$ and $1\leq j\leq k$. We begin by finding tuples of  words that are weakly $N$-small cancellation and that satisfy a weaker form of  condition \eqref{e: sc_condition_k>0}. More precisely, we first show the existence of a tuple of words $u^{N,k}(x,y,z)$ in the normal closure of $x$ and $y$ such that:
  		\begin{enumerate}
  			\item[a)] $A(u^{N,k}(a,b,c)_i) \cap A(u^{N,k}(a,b,c)_j)\neq \emptyset$ for all $1\leq i,j\leq m$,  where the subindices $i,j$ denote the $i,j$-th component of the $m$-tuple $u^{N,k}$, respectively. Moreover, this intersection is coherent.
  			\item[b)] $u^{N,k}(a,b,c)$ is $N$-small cancellation as witnessed by any base point $*$ in the intersection of all the axis of the components of  $u^{N,k}(a,b,c)$,
  			\item[c)] $f_1 \notin E(f_2)$, and $tl(f_1)\leq tl(f_2)$.
  			\item[d)] the following inequality holds for all $1\leq i\leq m$ and $1\leq j\leq k$: \begin{equation}\label{e: sc_condition_2_k>0}
  				tl(u^{N,k}(a,b,c)_{i})\geq N tl(c_{j}).
  			\end{equation}
  		\end{enumerate}
  		
  		We proceed by induction on $k$, the case $k=0$ being done above. Assume that $N\geq N_{0}$, where $N_{0}$ is the constant of Lemma \ref{l: simple domination}.

  		\newcommand{\kot}[0]{\frac{k(k+1)}{2}}
  		
  		Suppose we have constructed a pair $u^{N,k-1}$ in the normal closure of $x$ and $y$ such that $u^{N, k-1}(a,b,c)$ satisfies conditions a), b) and c) above, for any $N\geq 1$ and any $(k-1)$-tuple $c=(c_1, \dots, c_{k-1})$.
  		Denote $(e_{1},e_{2})=u^{N,k-1}(a,b,c_{1},\dots, c_{k-1})$. By construction the pair $(e_1,e_2)$ is $N$-small cancellation with respect to any $*$ in the intersection of the axis of $e_1$ and $e_2$.
  		Denote $(d_1, d_2) = ((w_1^{dm}(e_1,e_2))^2, (w_2^{dm}(e_1,e_2))^2)$, where $w_i^{dm}$ are the words given in Lemma \ref{l: simple domination}. By this same  lemma  we have
  		\begin{equation} \label{e: eq_6_10}
  			tl(d_{i})\geq \max\{d(*, c_{k}*),tl(e_{1}),tl(e_{2})\} \geq \max\{tl(c_k),tl(e_{1}),tl(e_{2})\},
  		\end{equation}
  		for some $i=1,2$, say for $i=1$ without loss of generality.  Moreover,
  		\begin{equation}\label{e: december_29}
  			A(d_1) \neq A(d_2),
  		\end{equation}
  		and both $d_1$ and $d_2$ are hyperbolic. By our assumptions, $d_1$ and $d_2$ are $\lambda$-stable, and by Remark \ref{r: stability and powers}, $d_1$ and $d_2$ are $(\lambda+2)$-stable.
  		
  		If we let $p=\ceil{2(\lambda+2)}$, then the implication $WS(\lambda)\rightarrow AI(\lambda+2)$  of Lemma  \ref{l: forms of stability} and Remark \ref{r: stability and powers} yield
  		\begin{align}
  			&|A(d_{1}^{p}) \cap A(d_{2}^{p})|< \lambda \frac{\lambda+2}{p}\max\{tl(d_1^p), tl(d_2^p)\} = (\lambda+2) \max\{tl(d_1), tl(d_2)\} \leq \nonumber \\
  			\leq &\frac{p}{2} \max\{tl(d_1), tl(d_2)\} = \frac{1}{2} \max\{tl(d_1^{p}), tl(d_2^{p})\}. \nonumber
  		\end{align}
  		Hence $(d_1^{p}, d_2^{p})$ is an acylindrical pair (we have used that $d_1^p \notin E(d_2^p)$ because $A(d_1) \neq A(d_2)$, and vice-versa).
  		
  		We can now apply Lemma  \ref{l: multiple conjugation} to the pair $(d_1^{p}, d_2^{p})$. Denoting  $f_{1}=d_{1}^p$ and $f_{2}=w^{cn}(d_{1}^{p},d_{2}^{p})$, we have
  		\begin{enumerate}
  			\item $f_{2}$ is hyperbolic in $T$,
  			\item $tl(f_{2})\geq \max\{tl(d_1^p), tl(d_2^p)\} \geq  tl(d_1^p) = tl(f_{1})$,
  			\item $f_1 \not \in E(f_2)$, and
  			\item $f_1=d_1^p$ is hyperbolic and  $A(f_{1})$ and $A(f_{2})$ intersect coherently (in particular $f_1$ does not invert a subsegment of $A(f_2)$). 
  		\end{enumerate}
  		In fact, $f_{2}$ is $\lambda$-stable in $T$ as well (by our hypothesis), so  the pair
  		$f_{1},f_{2}$ satisfies the hypothesis of Proposition  \ref{l: from good pair to small cancellation}. It follows that the $m$-tuple $w^{sc}(f_{1},f_{2})$ is $N$-small cancellation with respect to a base point $*$ in the intersection of the axis of the components in $w^{sc}(f_{1},f_{2})$. Additionally, from Proposition  \ref{l: from good pair to small cancellation} we have that \begin{equation}\label{e: december_27} tl(w^{sc}_{i}(f_{1},f_{2}))\geq N\max\{tl(f_{1}),tl(f_{2})\}, \quad i=1,\dots, m.
  		\end{equation}
  		Now by Lemma \ref{l: multiple conjugation},
  		$$
  		tl(f_{1}) = tl( w^{cn}(d_1^{\lambda+1}, d_2^{\lambda +1})) \geq \max\{tl(d_1^{\lambda+1}), tl(d_2^{\lambda+1})\} \geq \max\{tl(d_1), tl(d_2)\}.
  		$$
  		Hence by \eqref{e: december_27},
  		$tl(w^{sc}_{i}(f_{1},f_{2})) \geq N\max\{tl(d_{1}),tl(d_{2})\}.$
  		By \eqref{e: eq_6_10}, $$\max\{tl(d_{1}),tl(d_{2})\} \geq tl(d_1) \geq \max\{tl(c_k), tl(e_1) tl(e_3)\} \geq tl(c_{k})$$ and so $$tl(w^{sc}_{i}(f_{1},f_{2}))\geq  N tl(c_k)$$ for all $i=1,\dots, m$.
  		Hence the $m$-tuple of words
  		\begin{align}\label{e: december_31}
  			u^{N,k}=w^{sc}(w^{cn}(w^{dm}_{1}(u^{N,k-1},c_{k})^{2p}),w^{dm}_{2}(u^{N,k-1},c_{k})^{2p}))
  		\end{align}
  		are in the normal closure of 
  		$x$ and $y$ since by the definition $w_i^{dm}$, $i=1,2$ and assumption on $u^{N,k-1}$, both are in the normal closure of $x$; furthermore, they  satisfies conditions a), b) c), and d) above, and so the word $u^{N,k}$  satisfies all required properties.
  		
  		We now claim that there exists a pair  of words $v^{N,k}(x,y,z)$ in the normal closure of $x$ and $y$ such that $v^{N,k}(a,b,c)$ is $N$-small cancellation and for $T$ there is some vertex $*\in T$ such that
  		\begin{equation}\label{e: sc_condition_3_k>0}
  			tl(v^{N,k}_{i}(a,b,c))\geq N d(c_{j} *,*)
  		\end{equation}
  		for all $1\leq i\leq m$ and $1\leq j\leq k$. Indeed, we shall prove that such pair can be obtained by taking
  		\begin{gather}\label{e: december_31_2}
  			\begin{split}
  				&v^{N,k}(x,y,z_{1},z_{2},\dots, z_{k})=\\
  				&u^{2N, k + \kot}(x,y,z_{1},z_{2},\dots, z_{k},z_{1}z_{2},z_{1}z_{3},\dots, z_1z_k, z_{2}z_{3},\dots, z_2z_k, \dots, z_{k-1}z_{k}).
  			\end{split}
  		\end{gather}

  		\newcommand{\dl}[1]{d(#1\cdot*,*)}
  		Let $(f_{1},f_{2})=v^{N,k}(a,b,c)$ (these $f_i$'s are unrelated to the previous ones, but we keep the notation $f_i$). Denote by $S$ the set of vertices $*$ of $T$ such that
  		\begin{equation}\label{e: 25_sept_19}
  			2\max_{1\leq j\leq k}\dis{c_{j}}\leq  \min\{tl(f_{1}),tl(f_{2})\}=:L.
  		\end{equation}
  		We know that $S$ is non-empty by Lemma \ref{l: one basepoint} together with \eqref{e: sc_condition_2_k>0} (the only reason we use the components $z_i z_j$ in the definition of $v^{N,k}$ is  to guarantee that Lemma \ref{l: one basepoint} can be applied here). Pick $*_{0}\in S$ at minimal distance from $A(f_1) \cap A(f_2)$ and let $*_1 \in A(f_1) \cap A(f_2)$ be such that $d(S, A(f_1) \cap A(f_2))=d(*_0,*_1)$. 
 
  		Assume first that $d(*_{0},*_{1})\leq L$. Let $(g_1,g_2):=w^{sc}(f_1,f_2)$, where $w^{sc}$ is the tuple of words from Proposition  \ref{l: from good pair to small cancellation} taking the small cancellation parameter to be $3N$. For convenience let us denote such tuple  by $w^{sc, 3N}$ until the end of the proof.  We now show that $(g_1,g_2)$ is $N$-small cancellation over $c_1, \dots, c_k$.

  		First note that since $f_1$ and $f_2$  satisfy the conditions of Proposition  \ref{l: from good pair to small cancellation} (due to the properties of the words $u^{N,k}$) and $(g_1,g_2)=w^{sc, 3N}(f_1,f_2)$, we have that $g_1$ and $g_2$ are $3N$-small cancellation, $A(g_1)\cap A(g_2)\neq \emptyset$, $\min\{tl(g_1), tl(g_2)\} \ge 3NL$,  and   $*_1$ belongs to $A(g_1) \cap A(g_2)$. 
  		For all $j=1, \dots, k$,
  		$$
  		d(c_{j}*_{1},*_{1})\leq d(*_1, *_0) + d(*_0, c_j*_0) +  d(c_j*_0, c_j*_1) \leq 2d(*_1, *_0) + \frac{L}{20} \leq 3L 
  		$$
  	
  		Hence  $3L \leq \frac{\min\{tl(g_{1}),tl(g_{2})\}}{N}$. This together with condition b) in the formulation of the words $u^{N',k'}$, and since $*_1 \in A(g_1) \cap A(g_2)$, implies that $(g_{1},g_{2})$ is $N$-small cancellation over $c$ as witnessed by $*_{1}$. Hence in the case that $d(*_0, *_1)\leq L$  we are done.
  		
  		Denote $J=A(g_1) \cap A(g_2)$ and
  		assume now that $d(S,J)>L$, in which case we have $d(*_0, c_j *_0) \leq \frac{L}{20}\leq \frac{d(*_0, J)}{20}$ and so $*_0 \notin J$. We claim there exists $1\leq j_0\leq k$ such that the path $Q_j$ from $A(c_{j})$ to  $J$ contains the path $K$ from $*_{0}$ to $J$. Indeed, for each $1\leq j\leq k$ write $K =K_{j1}K_{j2}$ where $K_{j2} = Q_j \cap K$, and $K_{j1}$ is the remaining part of $K$. Note that $K_{j2}$ may be empty. Assume towards contradiction that our claim is false. Then $K_{j1}$ is nonempty for all $j$. It follows that there exists a vertex $*_0' \in \bigcap_j K_{j1}$ such that $*_0' \neq *_0$ (see Figure \ref{f: f6}). Consequently $d(*_0', J) < d(*_0, J)$, and  $*_0'$ is closer than $*_0$ to $A(c_j)$, for all $j$. This last property implies that $d(*_0', c_j *_0') \leq d(*_0, c_j *_0)\leq L/20$ for all $j$, which implies $*_0' \in S$, contradicting the fact that $*_0$ is the vertex from $S$ at minimal distance from $J$. The claim is proved.

  		\begin{figure}[ht]
  			\begin{center}
  				\begin{tikzpicture}
  					\tikzset{near start abs/.style={xshift=1cm}}

  					\filldraw (0,0) circle (1pt) node[left] {$*_0$};
  					\filldraw (1,0) circle (1pt) node[below, yshift=0.05cm] {$*_0'$};
  					\filldraw (5.5,0) circle (2pt) node[below, yshift=0.05cm, xshift = 0.4cm] {$*_1$};
  					
  					\draw[thick, color=teal] (0.5,2) -- (0.5, 1) -- (1.5,1) -- (1.5, 2);
  					
  					\draw[thick, color=teal] (2,-2) -- (2, -0.025) -- (3,-0.025) -- (3, -2);
  					
  					\draw[thick, color=teal] (1,1) -- (1,0);
  					
  					\draw[thick, color=blue] (5,2) -- (5,0.025) -- (7,0.025) -- (7,2);
  					\draw[thick, color=red] (5.5,-2) -- (5.5,-0.025) -- (6.5,-0.025) -- (6.5,-2);
  					\draw[thick] (0,0) -- (7,0);

  			    	\draw [thick,decorate,decoration={brace,amplitude=5pt},xshift=0pt,yshift=-0.3cm](2,0) -- (0,0) node [black,midway,yshift=-0.45cm] {\footnotesize$K_{21}$};
  					
  					\node[color=teal] at (2, 1.8) {$A(c_1)$};
  					\node[color=teal] at (3.6, -1.8) {$A(c_2)$};
  					\node[color=blue] at (5.5, 1.8) {$A(g_1)$};
  					\node[color=red] at (6, -1.8) {$A(g_2)$};
  					\draw [thick,decorate,decoration={brace,amplitude=5pt},xshift=0pt,yshift=0.1cm](0,0) -- (1,0) node [black,midway,yshift=0.45cm] {\footnotesize$K_{11}$};
  					\draw [thick,decorate,decoration={brace,amplitude=5pt},xshift=0pt,yshift=0.1cm](1,0) -- (5.5,0) node [black,midway,yshift=0.45cm] {\footnotesize$K_{12}$};
  					\draw [thick,decorate,decoration={brace,amplitude=5pt},xshift=0pt,yshift=-0.3cm] (5.5,0) -- (2,0)  node [black,midway,yshift=-0.45cm] {\footnotesize $K_{22}$};
  					\draw [thick,decorate,decoration={brace,amplitude=5pt},xshift=0pt,yshift=-0.3cm] (6.5,0) -- (5.5,0)  node [black,midway,yshift=-0.45cm] {\footnotesize $J$};
  					
  				\end{tikzpicture}
  			\end{center}
  			\caption{} \label{f: f6}
  		\end{figure}
  		
  		Consider the elements $g=f_{i_0}f_{i_0}^{c_{j_0}}$ and $g'=f_{i_0}^{2}(f_{i_0}^{2})^{c_{j_0}}$, where $j_0$ is the index given by the previous claim and $i_0$ is either $1$ or $2$ depending on whether $K$ has non-degenerate intersection  with $A(f_1)$ or with $A(f_2)$ (if none of the conditions are met then $i_0$ is chosen arbitrarily).
  		By Lemma \ref{l: product of conjugates}, both $g$ and $g'$ are hyperbolic in $T$ and $tl(g')\geq tl(g)\geq tl(f_{i_0})$. Moreover,  the path from $A({c_{j_0}})$ to $A(f_{i_0})$  contains $K$, hence $A(c_{j_0})$ and $A(f_{i_0})$ are disjoint. By Item 1 of Remark \ref{r: alb_remark_1} so are  $A(f_{i_0})$ and $A(f_{i_0}^{c_{j_0}})$.

  		Since $K$ is contained in the path connecting $A(f_{i_0})$ and $A(f_{i_0}^{c_{j_0}})$, it follows from Remark \ref{r: axis_of_product} that $*_{0}\in A(g)\cap A(g')$.   
  		Next we claim that $g \notin E(g')$. Indeed, it suffices to check that $A(g) \neq A(g')$. This is clear from Remark \ref{r: axis_of_product}: $A(g)$ intersects 
  		$A(f_{i_0})$
  		on a segment of length $tl(f_{i_0})$, while $A(g')$ intersects
  		this axis
  		on a segment of length $2tl(f_{i_0})$, so $A(g)\neq A(g')$. Note also that $g'$ is hyperbolic and so $\lambda$-stable, and $tl(g)\leq tl(g')$.  Hence the two elements $g$ and $g'$ satisfy the assumptions of Proposition  \ref{l: from good pair to small cancellation}.
  		
  		Therefore the pair $(h_{1},h_{2})=w^{sc}(g,g')$ is $N$-small cancellation as witnessed by $*_{0}$. Here we take  the small cancellation parameter in $w^{sc}$ to be $N$, and for convenience we denote this word by $w^{sc, N}$ during the rest of the proof.
  		In fact, it is small cancellation over $c_{1},c_{2}\dots c_{k}$, since $$d(c_{j}\cdot *_{0},*_{0})\leq \frac{1}{2} \min\{tl(f_{1}), tl(f_2)\} \leq  \frac{1}{N}\min\{tl(h_{1}),tl(h_{2})\},$$
  		where the first inequality is due to \eqref{e: 25_sept_19} and the second is a consequence of Proposition  \ref{l: from good pair to small cancellation}.
  		This completes the proof of the main part of the proposition: indeed, if we let $(v_{1},v_{2})=v^{N,k}$ and
  		$\alpha(x,y)=xx^{y}$, then the following family of $m$-tuples of words satisfies the required property:
  		\begin{align*}
  			\mathcal{W}=\left\{w^{sc,3N}(v^{N,k})\right\} \cup \left\{w^{sc, N}\left(\alpha(v^{N,k}_{i},z_{j}),\alpha((v^{N,k}_{i})^{2},z_{j})\right)\mid (i,j)\in\{1,2\}\times\{1,2,\dots k\} \right\}
  		\end{align*}
  		The first set $\left\{v^{2N,0}(v_{1},v_{1}v_{2})\right\}$ corresponds to the case $d(S,J)\leq L$, and the second set to the case $d(S,J) > L$. Note that the words belong to the normal closure of $x$ and $y$ since $v^{N,k}$ also belongs.
  	\end{proof}

  \begin{cor}
  	\label{c: small cancellation for free actions} Given any integer $N$, tuples of variables $x,x',z$ and an integer $m$, there is a finite collection $\mc{S}$ of $m$-tuples of words $w\in\F(x,x',z)^{m}$ in the normal closure of $\subg{x}$ satisfying the following property.
  	
  	Suppose we are given an action of a group $G$ on a real tree $T$ with trivial edge stabilizers and tuples $a\in G^{|x|}, a' \in G^{|x'|}$, $c\in G^{|z|}$ such that
  	$\subg{a,a'}$ acts irreducibly on $T$. Assume further that $a_{i_{0}}\neq 1$ for some $1\leq i_{0}\leq|x|$. Then  $w(a,a',c)$ is $N$-small cancellation over $c$ for some $w\in \mc{S}$.
  \end{cor}
  \begin{proof}
  	Let us first prove the following claim, which shows that from two tuples that generate a subgroup that has an irreducible action on a tree, one can uniformly describe a finite set of pair of words such that at least one of the pairs defines a 2 generated subgroup that also has an irreducible action on the tree.
  	
  	\begin{claim}\label{claim1}
  		There is a finite collection  $\mathcal{U}_{irr}\subset\F(x,x')^{2}$ of pairs of words in the normal closure of $x$ such that for any action of a group $G$ on a tree $T$ and $a, a'$ as above, the pair $(u(a,a'),v(a,a'))$ generates a group that acts irreducibly on $T$ for at least one $(u,v)\in \mathcal U_{irr}$.
  	\end{claim}
  	\begin{proof}[Proof of Claim]
  		The collection $\mathcal{U}_{irr}$ of all the pairs $(x_{i_{0}}x_{i_{0}}^{t_{1}},x_{i_{0}}^{t_{2}}x_{i_{0}}^{t_{3}})$, where $1\leq i_0 \leq |x|$ and $t_{1},t_{2},t_{3}$ range among all triples of terms in $\{1\}\cup x\cup x'$ satisfies the required conditions. Indeed, first note that for all $1\leq i_0\leq |x|$ there is $g\in\{a_{j},a'_{k}\}_{1\leq j\leq |x|,1\leq k\leq |x'|}$ such that $A(a_{i_0})$ and $A(a_{i_0}^{g})$ are different; otherwise the group generated by $a,a'$ would preserve the axis $A(a_{i_0})$, which is either a finite subtree (if $a_{i_0}$ is elliptic) or a bi-infinite line (if $a_{i_0}$ is hyperbolic), and thus the action of the group generated by $a$ and $a'$ on $T$ would not be irreducible.
  		
  		If $a_{i_{0}}^{2}\neq 1$, since $A(a_{i_0}^2)$ and $A((a_{i_0}^g)^2)$ are different, then  $\subg{a_{i_{0}}^{2},(a_{i_{0}}^{g})^{2}}$ acts irreducibly on $T$ and $(a_{i_{0}}^{2},(a_{i_{0}}^{g})^{2}) \in \mathcal{U}_{irr}$.
  		
  		If $a_{i_{0}}^{2}=1$ (and so $a_{i_0}$ is elliptic), since $A(a_{i_0})
  		$ and $A((a_{i_0}^g))$ are different, we have that $a_{i_{0}}a_{i_{0}}^{g}$ is hyperbolic. Since the action of $\subg{a,a'}$ on $T$ is irreducible, there exists $h\in \{a_{j},a'_{k}\}_{1\leq j\leq |x|,1\leq k\leq |x'|}$ such that $a_{i_0}a_{i_0}^h$ is hyperbolic and the axis $A(a_{i_0}a_{i_0}^g)$ and $A(a_{i_0}a_{i_0}^h)$ are different. Hence
  		$\subg{a_{i_{0}}a_{i_{0}}^{g},a_{i_{0}}a_{i_{0}}^{h}}$ acts irreducibly on $T$ and $(a_{i_{0}}a_{i_{0}}^{g},a_{i_{0}}a_{i_{0}}^{h}) \in \mathcal{U}_{irr}$.
  	\end{proof}
  	Let us now address the proof of Corollary \ref{c: small cancellation for free actions}.
  	Let $\mc{W}$ be the collection of $m$-tuples of words provided by Theorem \ref{l: general small cancellation}  for $\lambda=1$ and the same value of $N$ as in the statement of the corollary. Note that since $G$ acts on $T$ with trivial edge stabilizers, it follows that each hyperbolic element is $1$-stable.
  	
  	We claim that the following collection :
  	\begin{align}
  		\mathcal{S}:=\{w(u(x,x')^{s},v(x,x')^{t},z)\,|\,(s,t)\in\{10,11\}^{2},\ (u,v)\in\mathcal{U}_{irr}, w\in \mc{W}\}
  	\end{align}
  	satisfies the required properties. Notice that from Theorem \ref{l: general small cancellation}, we have that, for any $w\in \mc{W}$, the word $w(x,x',z)$ can be assumed to be in the normal closure of $x$ and $y$. 
  	  Since by Claim \ref{claim1}, the word $v(x,x')$ is in the normal closure of $x$, we conclude that any $S\in\mathcal{S}$ can be assumed to be in the normal closure of $x$.
  	
  	Suppose we are given $G$, $T$, $a,a',c$ as in the hypothesis.
  	From Claim \ref{claim1}, it follows that there is $(u,v)\in\mathcal{U}_{irr}$ such that $\subg{u(x,x'),v(x,x')}$ acts irreducibly on $T$.
  	By Lemma \ref{c: irreducibility of powers}, there is $(s,t)\in\{10,11\}^{2}$ such that $(b,b')=(u(x,x')^{s},v(x,x')^{t})$ acts irreductibly on $T$ and both $b$ and $b'$ have roots of order greater than or equal to $10\lambda=10$. Hence the triple $(b,b',c)$ satisfies the assumptions of Theorem \ref{l: general small cancellation} and so we conclude that $w(u(a,a')^s,v(a,a')^t,c)$ is $N$-small cancellation over $c$ for some $w\in\mathcal{W}$, $u,v\in \mc{U}_{irr}$, and $s,t\in \{10,11\}$.
  \end{proof}

\section{Applications of uniform small cancellation}
   In this section we will use the results of Section \ref{sec:uniform small cancellation} to deduce two corollaries, of interest on their own. On the one hand, we prove a quantifier reduction for positive sentences. More precisely, we show that if an arbitrary group satisfies a non-trivial positive sentence, then it also satisfies a non-trivial positive $\forall \exists$-sentence, see Theorem \ref{l: quantifier reduction}.
  
  As a second application we show that the class of groups with non-trivial positive theory is closed under extensions, see Theorem \ref{lem:extensionpreservation}.

  \subsection{Quantifier reduction for positive sentences}\label{sec:quantifier reduction}
  
  The goal of this section is to show that if a group satisfies a non-trivial positive sentence, then it also satisfies a non-trivial positive $\forall \exists$-sentence.
  
  In fact, we will prove the aforementioned result for a wider set of sentences which we call \emph{generic almost positive}. This fragment of the theory is not as classical and well-known as the positive theory, but it is relevant in the forthcoming work on interpretable sets in graph products of groups, see \cite{graphproductspreprint} and since the proofs are very similar for the two types of sentences, we treat the most general case, for generic almost positive sentences, here. For readers only interested in the positive theory, they can go directly to Theorem \ref{l: quantifier reduction} and assume that the systems of inequalites is empty.
    
    \medskip
    
    \begin{definition}[Almost positive sentence]
    	We say that a sentence is \emph{almost positive} if it is of the form:
    	\begin{align*}
    		\phi\equiv\exists z\,(\Theta(z)\neq 1\wedge\psi(z))
    	\end{align*}
    	where $\psi(z)$ is a positive formula whose free variables are in the tuple $z$ and $\Theta(z)\neq 1$ a non-trivial system of inequalities.
    \end{definition}
    
    These more general sentences allow to express properties of the group such as the group has a non-trivial center or that it is boundedly generated.
    
    \begin{example}\
    	\begin{itemize}
    		\item The property of having non-trivial center can be axiomatized using an almost positive sentence:
    		\begin{align*}
    			\exists z\,\,z\neq 1 \wedge \forall x\,\,[x,z]=1
    		\end{align*}
    		
    		\item The property of a group to be bounded centraliser generated, that is $G=C(z_1) \cdots C(z_n)$ for some $z_1, \dots, z_n \in G$ is also axiomatized by an almost positive sentence:
    		\begin{align*}
    			\exists z_{1}\exists z_{2}\cdots\exists z_{k}\, (\bigwedge_{i=1}^{k}z_{i}\neq 1)\wedge(\forall x\exists y_{1}\exists y_{2}\cdots\exists y_{k}\,(\bigwedge_{i=1}^{k}[y_{i},z_{i}]=1)\wedge x=y_{1}y_{2}\cdots y_{k})
    		\end{align*}
    	\end{itemize}
    \end{example}
    
    In our approach what determines whether or not a positive sentence is trivial is the existence of a formal solutions. Our goal is to define a set of sentences that on the one hand allow some (controlled) inequalites and on the other one, the triviality of the sentence (the fact that it is true in all groups) is determine by the existence of a formal solution that witness both the inequalites as well as the positive part of the sentence. More generally, we want to restrict to those cases in which there is no $H$ containing a tuple $a$ such that $\Theta(a)\neq 1$ and the sentence with parameters $\psi(a)$ is trivially satisfied in any group $G$ containing $H$.
    
    \newcommand{\easy}[0]{easy } 
    
    This brings us to the definition of generic almost positive sentence.
    
    \begin{definition}[Generic almost positive sentence]
    	We say that an almost positive sentence $\phi\equiv\exists z\,(\Theta(z)\neq 1\wedge\psi(z))$ is \emph{generic} if given any formal solution for the formula $\psi(z)\equiv \forall x^1, \exists y^1, \dots,\forall x_m, \exists y^m \, \Sigma(z,x^1, \dots, y^m)=1$ relative to a diophantine condition
    	$\exists w\,\,\Pi(w,z)=1$, see Definition \ref{defn:formal solution}, there is some word in $\Theta(z)$ which is in the normal closure of $\Pi(w,z)$ in $\F(w,z)$ or, equivalently, one of the words in $\Theta(z)$ is the trivial element in the group $\F(w,z)/\Pi(w,z)=1$.
    	
    \end{definition}
    
  \bigskip
    
    We begin recalling an easy consequence of the well-known fact
    \footnote{    \begin{align*}
    	x=1 \vee y=1   \leftrightarrow \bigwedge_{\epsilon,\delta\in\{1,-1\}}([xa^{\epsilon}xa^{-\epsilon},yb^{\epsilon}yb^{-\epsilon}]=1), 
    \end{align*} where $a$ and $b$ are any two elements that do not commute.}, due to Gurevich, that any finite disjunction of equations in the free group over a non-abelian subgroup
    of parameters is equivalent to a single system of equations.

    \begin{lemma}
    	\label{l: no disjunctions} Given a positive formula 
    	$$\phi(z)\cong \forall x_{1}\exists y_{1}\forall x_{2}\cdots \exists y_{r}
    		\,\,\chi(z,x_{1},\dots x_{r},y_{1},\dots, y_{r})$$ 
    	in the language of groups, where $\chi$ is a positive quantifier free formula,
    	there is a system of equations $\Sigma=1$ and
    	a positive formula of the form:
    	\begin{equation}
    		\theta(z,u,v)\cong \forall x_{1}\exists y_{1}\forall x_{2}\cdots \exists y_{r}
    		\,\,\Sigma(z,u,v,x_{1},\dots, x_{r},y_{1},\dots, y_{r})=1
    	\end{equation}
    	such that $\forall z\forall u\forall v(\phi(z)\rightarrow\theta(z,u,v))$  is valid in any group. Moreover
    	\begin{align*}
    		\F\models\forall z\forall u \forall z\,\,([u,v]= 1\vee(\phi(z)\leftrightarrow \theta(z,u,v))).
    	\end{align*}
    	In particular, if the given $\phi$ is a non-trivial positive sentence, then the sentence $\forall u\forall v\,\, \theta(u,v)$ is a non-trivial positive sentence as well. 
    \end{lemma}
    
    \begin{proof}
    The only statement left to prove is that if $\phi$ is non-trivial, then so is the sentence $\theta'= \forall u \forall v \ \theta(u,v)$. Since $\F \models \forall u \forall v\,\,([u,v]= 1\vee(\phi\leftrightarrow \theta(u,v)))$, for a non-commuting pair of elements of the free group, we have that the sentences $\phi$ and $\theta'$ are equivalent, that is $\F$ satisfies $\phi$ if and only if it satisfies $\theta'$. Since $\phi$ is non-trivial, $\F$ does not satisfy $\phi$ and so it neither satisfies $\theta'$. Hence $\theta'$ is non-trivial.
    
    \end{proof}

        \begin{thm}[Quantifier Reduction]
    	\label{l: quantifier reduction} Given a generic almost positive sentence $\psi\equiv\exists z\,\Pi(z)\neq 1\wedge\phi(z)$, where $\phi(z)$ is a positive formula, one can effectively describe a generic almost positive sentence
    	$\psi'\equiv\exists z\,\Pi(z)\neq 1\wedge\phi'(z)$ where $\phi'(z)$ is a positive $\forall\exists$-formula. In particular, if a group has non-trivial positive theory, it must satisfy some non-trivial positive $\forall\exists$-sentence.
    \end{thm}
    \begin{proof}
    Using Lemma \ref{l: no disjunctions} we can assume that the positive formula $\phi(z)$ is of the form:
    	\begin{align*}
    		\phi(z)\equiv \forall x_{1}\exists y_{1}\cdots\forall x_{r}\exists y_{r}\;\Sigma(z,x_{1},\dots,x_r, y_1,\dots,y_{r})=1,
    	\end{align*}
    	where each $x_i, y_i$ are tuples of variables and $\Sigma(z,x_{1},\dots,x_{r},y_{1},\dots,y_{r})=1$ is a system of equations.
    	From the definition of generic almost positive sentence and Theorem \ref{c: iterated formal solutions}, it follows that there is some $N>0$ such that given any $G$ acting on a tree $T$ and given tuples $c\in G^{|z|}$, $a_{i}\in G^{|x_{i}|}$ and $b_{i}\in G^{|y_{i}|}$ such that
    	$a_{i}$ is $N$-small cancellation in $T$ over $(c,a_{j},b_{j})_{j<i}$, then both $\Sigma(c,a,b)=1$ and $\Pi(c)\ne 1$ cannot hold simultaneously,  where $a$, $b$ denote the collections of all tuples $a_i$ and $b_i$, $1\leq i \leq r$, respectively.
    	
    	We define a set of words by recursively using Corollary \ref{c: small cancellation for free actions} as follows. Let $\tau_{1}(t,t',z)$ be the set of words provided in Corollary \ref{c: small cancellation for free actions} for the tuples of variables $t,t',z$ and $m=1$; assume that $\tau_{j-1}(t,t',z, y_{1},y_{2},\dots, y_{j-2})$ has been defined; then $\tau_{j}(t,t',z, y_{1},y_{2},\dots, y_{j-1})$ is the set of words provided in Corollary \ref{c: small cancellation for free actions} for the tuples of variables $t,t',z_j$, where $z_j$ is the tuple $\left(z,\tau_{1}(t,t',z), \dots, \tau_{j-1}(t,t',z, y_{1},y_{2},\dots, y_{j-2}), y_1, \dots, y_{j-1}\right)$, for $j=2, \dots, r$.
    	
    	From Corollary \ref{c: small cancellation for free actions}, these tuples have the property that given any action of a group $G$ on a tree $T$ with trivial edge stabilizers, given $a,\in G^{|t|}$, $a'\in G^{|t'|}$ such that $\subg{a,a'}$ acts irreducibly on $T$ and given $c\in G^{|z|}$, $b_{i}\in G^{|y_{i}|}$, some tuple $d_{j}=\tau_{j}(c,b_{1},b_{2},\dots, b_{j-1})$ is $N$-small cancellation over $c, d_{1},d_{2},\dots, d_{j-1},b_{1},\dots, b_{j-1}$.
    	
    	We now take the positive $\forall \exists$-formula:
    	\begin{align*}
    		\phi'(z)\equiv&\forall t\forall t'\exists y_{1}\exists y_{2}\dots\exists y_{r}\,\,\\
    		\Sigma(z,\tau_{1}(z,t,t'),&\tau_{2}(z,t,t',y_{1}),\dots,\tau_{r}(z,t,t',\dots,y_{r-1}),y_{1},y_{2},\dots, y_{r})=1
    	\end{align*}
    	Clearly the sentence $\psi'\equiv \exists z \, \Pi(z) \ne 1 \wedge \phi'(z)$ is valid in every group in which the sentence $\psi$ is valid as, in fact, $\forall z \, \phi(z) \to \phi'(z)$ holds in any group.
    	
    	We are left to show that the almost positive sentence $\phi'$ is generic. 
    	Consider a formal solution
    	\begin{align*}
    		\alpha=(\alpha_{1}(w,z,t,t'),\alpha_{2}(w,z,t,t'),\dots,\alpha_{r}(w,z,t,t'))
    	\end{align*}
    	for the positive formula $\phi'$ relative to a diophantine condition $\exists w\,\Theta(w,z)=1$.
    	
    	Since $\alpha$ is a formal solution, by definition, see Definition \ref{defn:formal solution}, each of the words in the expression
    	\begin{align*}
    		\Sigma(z,\tau_{1}(z,t,t'),\tau_{2}(z,t,t',\alpha_{1}),\dots,\tau_{r}(z,t,t',\dots,\alpha_{r-1}),\alpha_{1},\alpha_{2},\dots, \alpha_{r})
    	\end{align*}
    	is trivial as an element of $G=G_{\theta}\frp\F(t,t')$, where $G_{\Theta}$ is the quotient of the group $\F(w,z)$ by the relations given by the system $\Theta(w,z)=1$. Let $T$ be the Bass-Serre tree dual to the decomposition $G=G_{\Theta}\frp\F(t,t')$, regarded as an iterated HNN-extension of the group $G_{\Theta}$ with trivial edge stabilisers. Clearly $G$ acts on a tree with trivial edge stabilisers and $t,t'$ act irreducibly on $T$. By the properties of the words $\tau_i$, we deduce that $\tau_{j}(t,t',z, \alpha_{1},\dots,\alpha_{j-1})$ is $N$-small cancellation over
    	$(z, \tau_i,\alpha_{i})_{i<j}$. As we note above, under the aforementioned conditions $\Sigma(z, \tau_1, \dots, \tau_r,\alpha_1, \dots, \alpha_r)=1$ and $\Pi(z)\ne 1$ cannot hold simultaneously in $G$. Since $\Sigma(z, \tau_1, \dots, \tau_r,\alpha_1, \dots, \alpha_r)=1$ in $G$ we conclude that $\Pi(z)\ne 1$ cannot hold in $G_\Theta$ (note that in this context, $z$ is view as an element of $G$ and not as a variable). It follows by definition that $\phi'$ is a generic almost positive sentence.
    \end{proof}
    
    Combining Theorem \ref{l: quantifier reduction} and Theorem \ref{l: small cancellation solutions}, we obtained the following Theorem.
    
    \begin{thm}\label{thm: wscitp}
    	If $G$ acts irreducibly and minimally on a tree and contains weak small cancellation elements, then its positive theory is trivial.
    \end{thm}
    
  \subsection{Preservation of the non-trivial positive theory under extensions} \label{sec:preservation extensions}
    
    The goal of this section is to show that the class of groups with non-trivial positive theory is closed under group extensions, that is, if $N$ and $Q$ are group with non-trivial positive theory and
    
    $$
    1 \to N \to G \to Q \to 1,
    $$
    
    then $G$ has non-trivial positive theory.
    
    The most common property used to show that a group has non-trivial positive theory is the finite width of some of its verbal subgroup (see also the discussion in Section \ref{sec:openquestions}). However, we want to remark that this property is \emph{not} closed under finite extensions. Indeed, in \cite{george1976verbal} (see also \cite{segal2009words}) they describe a group $H$ with commutator width 1 and a group $G$ which contains $H$ as index 3 subgroup such that $G$ has infinite commutator width. This helps us to stress a key point of our result: if $N$ and $Q$ satisfy some non-trivial positive sentence $\phi$, then $G$ satisfies a non-trivial positive sentence, not necessarily $\phi$. In that example, $H$ is a nilpotent group while $G$ is solvable. 
    
\medskip
    
    We now address the main result of this section.
    
    \begin{thm} \label{lem:extensionpreservation}
    	Let $\mathcal C$ be either the class of all groups satisfying generic almost positive sentences, the class of groups with non-trivial positive theory or the class of groups satisfying a non-trivial law. Then $\mathcal C$ is closed under extensions. That is: if $K,Q\in \mc C$,  $K \triangleleft E$ and $E/K \simeq Q$, then $E\in \mc C$.
    \end{thm}
    \begin{proof}
    	In virtue of Lemma \ref{l: no disjunctions} and Theorem \ref{l: quantifier reduction}, we can assume that the positive formula is of $\forall\exists$-type and there involves no disjuntions, i.e. there are sentences
    	$\phi_{K}\in Th^{+}(K)$ and $\phi_{Q}\in Th^{+}(Q)$ of the form
    	\begin{align*}
    		\phi_{Q}\equiv\,\,&\exists s\,\,\Pi_{Q}(s)\neq 1\wedge\psi_{Q}(s) \\
    		\phi_{K}\equiv\,\,&\exists t\,\,\Pi_{K}(t)\neq 1\wedge\psi_{K}(t)
    	\end{align*}
    	where $\psi_{Q}$ and $\psi_{K}$ are of the form:
    	\begin{align*}
    		\psi_{Q}\equiv\forall x\exists y\,\,\Sigma_{Q}(s,x,y)=1\\
    		\psi_{K}\equiv\forall z\exists w\,\,\Sigma_{K}(t,z,w)=1
    	\end{align*}
    	and $\phi_{K},\phi_{Q}$ do not admit a formal solution relative to a diophantine condition compatible with $\Pi_{Q}(s)\neq 1$
    	and $\Pi_K(t)\neq 1$ respectively.
    	
    	From the definition of generic almost positive sentence and Theorem \ref{l: small_cancellation_parameters}, it follows that there is some $N>0$ such that given any $G$ acting on a tree with trivial edge stablisers
    	and tuples $(c,a,b)\in G^{|s|}\times G^{|x|}\times G^{|y|}$ such that $a$ is $N$-small cancellation over $c$, then both $\Pi_{Q}(c)\neq 1$ and 
    	$\Sigma_{Q}(c,a,b)=1$ cannot hold simultaneously in $G$.
    	
    	Let $q=(q_{1},q_{2},\dots, q_{m})$ be a multivariable, where $m$ is the number of words in $\Sigma_{Q}$. Let also $n=|s|$ and $o=|t|$.

    	Applying Corollary \ref{c: small cancellation for free actions} with $q$ in place of $x$, $u,u'$ in place of $x'$ and $t$ in place of $z$, we get a finite collection $\mathcal{T}$, where each $\tau\in\mathcal{T}$ is a $|z|$-tuple of words $\tau(q,u,u',t)$ in the normal closure of $\subg{q}$ such that given any $G$-tree $T$ with trivial edge stabilizers and
    	$b_{1},b_{2},c\in G$ and $a\in G^{m}$ where $\subg{a,b_{1},b_{2}}$ acts irreducibly and $a_{i}\neq 1$ for some $1\leq i\leq m$,
    	the tuple $\tau(a,b_{1},b_{2},c)$ must be $N$-small cancellation over $c$ in $T$. Consider now the formula:

    	\begin{align*}
    		\psi(s,t)\equiv\forall u\forall u'\forall x\exists y
    		\,\bwedge{\tau\in\mathcal{T}}{\exists w^{\tau}\,\,\Sigma_{K}(t,\tau(\Sigma_{Q}(s,x,y),u,u',t),w^{\tau})=1}{}.
    	\end{align*}
    	
    	Let us first check that $E\models\phi$ where $\phi\equiv\exists s\exists t\;\Pi_{K}(s)\neq 1\wedge \Pi_{Q}(t)\neq 1\wedge\psi(s,t)$. It will be convenient for us to phrase our arguments in terms of the validity of certain sentences in the expansion of $G$ by a predicate $P_{K}$ for $K$. For simplicity, we will
    	write them using syntactic abbreviations, for instance if $x$ is a 2-tuple, we write $\exists x\in K\,\phi(x)$ instead of $\exists x_{1}\exists x_{2}\,\,P_K(x_{1})\wedge P_{K}(x_{2})\wedge\phi(x)$ etc.
    	Notice first that the validity of $\phi_{Q}$ in $Q$ implies that $E$ satisfies:
    	\begin{equation}
    		\label{phiQ}\exists s\,\,\Pi_{Q}(s)\neq 1\wedge\forall x\exists y\,\,\Sigma_{Q}(s,x,y)\in K.
    	\end{equation}

    	Now, for $\tau\in\mathcal{T}$ the fact that all components of $\tau(q,u,u',t)$ belong to the normal closure of $q$ implies the validity of the following sentence in $E$:
    	\begin{align*}
    		\forall t,\forall u\forall u' (\forall q\in K)\,\,\,\tau(t,u,u',q)\in K.
	\end{align*}
    	From this, together with the validity of $\phi_K$ on $K$, it follows that:
    	\begin{equation}
    		\label{phiK}\exists t\in K\;\Pi_{K}(t)\neq 1\wedge\forall u\forall u'(\forall q\in K)
    		\,\,\exists w\in K\,\, \Sigma_{K}(w,\tau(t,u,u',q),t)=1.
    	\end{equation}
    	The validity of $\phi$ in $G$ follows from (\ref{phiQ}) and (\ref{phiK}).
    	
    	We now show that the almost positive sentence $\phi$ is generic. Consider a formal solution
    	$(\alpha,\beta)$ ($\alpha$ being an expression for $y$ and $\beta$ for $w=(w^{\tau})_{\tau\in\mathcal{T}}$) for $\phi$
    	relative to some diophantine condition $\exists r\,\Theta(s,t,r)=1$.
    	Our goal is to show that $\exists r\,\,\Theta(s,t,r)=1\wedge\,\,\Pi_{Q}(s)\neq 1\wedge\Pi_{K}(t)\neq 1$ is inconsistent. Assume towards contradiction that this is not the case, i.e. the pair $(s,t)$ in the group $G_{\psi}=\subg{s,t,r\,|\,\Theta(s,t,r)=1}$ verifies the condition $\Pi_{Q}(s)\neq 1\wedge\Pi_{K}(t)\neq 1$.
    	
    	Recall from the definition of formal solution, see Definition \ref{defn:formal solution}, that we can see $(\alpha,\beta)$ as the image in $G=G_{\psi}\frp\F(u,u',x)$ of the tuples $y,w$ by a homomorphism from $G_{\Sigma'}$ to $G$, where $\Sigma'=\Sigma_{K}(t,\tau(\Sigma_{Q}(s,x,y),u,u',t),w^{\tau})=1$ is the system of equations inside the quantifiers of $\psi(s,t)$.
    	
    	Denote by $T$ be the Bass-Serre tree with trivial edge stabilizers associated with the decomposition of $G$ above, where vertex stabilizers are the conjugates of $G_{\psi}$.
    	Since $\psi_{Q}(s)$ does not admit a formal solution relative to a diophantine condition compatible with $\Pi_{Q}(s)\neq 1$, at least one of the words in the tuple:
    	\begin{align*}
    		\sigma=\Sigma_{Q}(s,x,\alpha(r,s,t,u,u',x))
    	\end{align*}
    	has to be non-trivial. Otherwise, the result of postcomposing $\alpha$ with the retraction killing $u$ and $u'$ in $G_{\Sigma_{Q}}$ results in a formal solution for $\phi_{Q}$ relative to the diophantine condition $\exists r\,\exists t\;\Theta(s,t,r)=1$.
    	
    	Clearly, $G$ acts on its Bass-Serre tree with trivial edge stabiliser, $\langle \sigma,u,u'\rangle$ acts irreducibly on it and as notice above, at least one of the words in $\sigma$ is non-trivial; the choice of $\tau$ and Corollary \ref{c: small cancellation for free actions} implies the existence of some $\tau\in\mathcal{T}$  such that $\tau(\sigma,u,u',t)$ is $N$-small cancellation over $t$. The equality $\Sigma_{K}(t,\tau(\Sigma_{Q}(s,x,y),u,u',t),w^{\tau})=1$ in $G$, the fact that $\tau(\sigma,u,u',t)$ is $N$-small cancellation over $t$ and that $\Pi_{K}(t)\neq 1$ holds in $G_{\psi}$ implies by Theorem \ref{l: small_cancellation_parameters} that $\psi_{K}(t)$ admits a formal solution relative to a diophantine condition compatible with $\Pi(t)\neq 1$, contradicting the assumption that $\psi_K(t)$ is generic almost positive. Therefore $\exists r\,\,\Theta(s,t,r)=1\wedge\,\,\Pi_{Q}(s)\neq 1\wedge\Pi_{K}(t)\neq 1$ is inconsistent and so $\phi$ is generic.
     	
    	If we take $\Pi_{K}$ and $\Pi_{Q}$ to be empty, the argument shows the preservation of non-triviality of the positive theory under extensions. If we take $y$ and $w$ to be empty, we recover the fact that the extension of a group satisfying a non-trivial law by another one which does also satisfies a non-trivial law.
    \end{proof}

\section{Characterization of groups acting on trees with trivial positive theory} \label{sec: characterisation}
  
  The goal of this section is to show that if a group $G$ acts minimally on a real tree and the action on the boundary neither fixes a point nor is 2-transitive, then it has weak small cancellation elements and so the positive theory is trivial. In order to prove it, we will make use of a characterisation of the type of action in consideration given by Iozzi, Pagliantini and Sisto in \cite{iozzi2014characterising}.  In their work, the authors show that if a group admits this type of action on a tree, then the action admits a good labeling and from it, they deduce the existence of very long random segments with small cancellation features. The strategy of the proof is then to prove that having a good labeling (and so these long random segments) is equivalent to having weak small cancellation elements as in Definition \ref{d: small cancellation}.  
  
  We remark that while in the rest of the article the trees we consider may be real, in this section we require them to be simplicial.
  
  \bigskip
  
  For completeness, we recall the definitions from \cite{iozzi2014characterising} just necessary to prove this equivalence.
  
  Fix an action of a group $G$ on a simplicial tree $T$ and a vertex $v\in T$. Denote by $\mathcal{E}^{(n)}$ the collection of all oriented geodesic segments of length $n$.
  \begin{definition}
  	An $o$-geodesic is an \emph{oriented} segment in $T$ both whose endpoints are in the orbit $G\cdot v$. The $o$-length of $J$, denoted by $|J|_{0}$, is $\#(J\cap G\cdot v)-1$.
  \end{definition}
  
  Given any $o$-geodesic $J$, let $J^{op}$ stand for the inverse orientation of $J$.
  
  \begin{definition}[Good labelling, Sisto \& co.\cite{iozzi2014characterising}] \label{defn:goodlabelling}
  	Fix an alphabet with three letters $\{a,b,c\}$. A labelling is given by some pair $(k,f)$, where $k$ is a positive integer and $f$ a $G$-invariant map (the image is determined by the orbit) from $\mathcal{E}^{(k)}$ to $\{a,b,c\}$. We say that the labelling is \emph{good} if both $f^{-1}(a)$ and $f^{-1}(b)$ are finite and at least one of the following two conditions is satisfied:
  	\enum{(i)}{
  		\item \label{gl 1}$|f^{-1}(a)|=|f^{-1}(b)|=1$, i.e. all segments labelled by $a$ (resp. by $b$) are in a unique $G$- orbit;
  		\item \label{gl 2} $f(J)=a$ implies $f(J^{op})\neq b$, i.e. $f(J^{op})$ can only have labels $a$ or $c$.
  	}
  	We will write $C_{f}:=\max\,\{|J|\,\}_{J\in f^{-1}(\{a,b\})}$.
  \end{definition}
  
  	For each $N\geq k$, any labeling $(k,f)$ determines  a $G$-invariant function
  	$f^{*}:\mathcal{E}^{(N)}\to\{a,b,c\}^{N-k+1}$ as follows. Given $J\in\mathcal{E}^{(N)}$,
  	if $J_{1},J_{2},\dots J_{N-k+1}\in\mathcal{E}^{(k)}$ are the $o$-geodesics of $J$ of
  	$o$-length $k$ (from left to right) with the same orientation as $J$, ordered from left to right, then
  	we let $f^{*}(J)=f(J_{1})f(J_{2})\dots f(J_{N-k+1})$.
  
  Observe that the image by $f^{*}$ of an initial (resp. final) oriented subsegment of an oriented segment $J$ is an initial (resp. final) subword of $f^{*}(J)$ of the appropriate length.
  
  If $J,J'$ are oriented segments whose orientation intersect coherently (i.e. the two orientations agree on the intersection), then $f^{*}(J\cap J')$ is a common subword of $f^*(J)$ and $f^*(J')$.
  
  Given a sequence of positive integers $\bar{e}=(e_{1},e_{2},\dots e_{q})$
  and $X,Y$ two letters from $\{a,b,c\}$, let $u_{\bar{e}}(X,Y)$ be the word
  $XY^{e_{1}}XY^{e_{2}}\dots XY^{e_{q}}X$.
  
  Note that in general there is no relation between the words $f^*(J)$ and $f^*(J^{op})$; however if the labelling is good as defined in Definition \ref{defn:goodlabelling}, then the next lemma determines some relations between $f^*(J)$ and $f^*(J^{op})$.
  
  \begin{lemma}
  	\label{reverse} Let $(k,f)$ be a good labeling and $J$ an $o$-geodesic. Assume that $f^{*}(J)=u_{\bar{e}}(a,b)$ for some tuple $\bar{e}$ of positive integers $>2$ and let $v:=f^{*}(J^ {op})$.
  	Assume that
  	$v_{0}=ab^{m_{1}}ab^{m_{2}}ab^{m_{3}}a$ with $m_{0},m_{1},m_{2}>0$ is a sub-word of $v$. Then either:
  	\enum{(A)}{
  		\item \label{reverse 1} there exists $2\leq i\leq|\bar{e}|-1$ such that  $m_{2}=e_{i}$, $m_{1}\leq e_{i+1}$ and $m_{3}\leq e_{i-1}$ or
  		\item \label{reverse 2}$2m_{l}\leq e_{i}$ for some $1\leq i\leq |\bar{e}|$ and $l\in\{1,2,3\}$.
  	}
  \end{lemma}
  \begin{proof}
  	Assume first that the good labelling satisfies condition (\ref{gl 1}) in Definition \ref{defn:goodlabelling}. Since $v_0$ is a subword of $v$ and by assumption all segments labelled by $a$ (resp. $b$) belong to the same orbit, we have that for all segment $I$ in $J$, $f^*(I)= a$ (resp. $f^*(I)=b$) if and only if $f^*(I^{op})=a$ (resp. $f^*(I^{op})=b$) and so $v$ is the word $f^*(J)$ read from right to left. It follows that alternative (A) of the statement holds.
  	
  	Assume now we are in case (\ref{gl 2}) of Definition \ref{defn:goodlabelling}, that is if $f^*(I)=a$, then $f^*(I^{op})$ is either $a$ or $c$. Hence if $f^*(I^{op})=b$, then we have that $f^*(I)=b$ and so any maximal subword of $v$ of the form $b^{m}$ is the image of a segment whose opposite is contained in a segment with image
  	$b^{e_{j}}$ in $u_{\bar{e}}(a,b)$. In this case, we say that $b^m$ is a segment of $b^{e_j}$.
  	
  	If $b^{m_2}$ is a fragment of $b^{e_i}$, then we have two alternatives: either $m_2=e_i$ or $m_2<e_i$. If $m_2=e_i$, then $b^{m_{1}}$ and $b^{m_{3}}$ are fragments of $b^{e_{i+1}}$ and $b^{e_{i-1}}$ respectively and thus the alternative (\ref{reverse 1}) of the statement holds. If $m_2< e_i$, then either $b^{m_1}$ or $b^{m_3}$ (or both) is a fragment of the same $b^{e_i}$ and so at least one of the fragments $b^{m_1l}$, $l\in \{1,2,3\}$ has length at most $\frac{e_{j}}{2}$ and so Alternative (B) from the statement holds.
  \end{proof}
  
  \begin{definition}\label{defn:independenttuple}
  	We say that a finite tuple $(\bar{e}^{i})_{i=1}^{r}$ of tuples of positive integers is
  	\emph{independent} if the following holds:
  	\enum{(a)}{
  		\item $e^{i}_{j}>1$ for any $i,j$
  		\item \label{independence 1}$ e^{i}_{j}=e^{i'}_{j}$ implies that $(i,j)=(i,j')$
  		\item \label{independence 2} $2e^{i}_{j}>e^{i'}_{j'}$ for any $i,i'$ and $j,j'$.
  	}
  	The following Lemma describes the key property of a good labeling.
  \end{definition}

  \begin{lemma}
  	\label{l: gl intersection}Let $(k,f)$ be a good labeling and $J_{1},J_{2},\dots J_{k}$ $o$-geodesics such that
  	for all $1\leq j\leq k$ we have $f^{*}(J_{j})=u_{\bar{e}^{j}}$ for some tuple $\bar{e}^{j}$ of positive integers
  	and let $e_{max}:=\max\{e^{j}_{i}\}_{i,j}$.
  	Assume moreover that $(\bar{e}^{j})_{j=1}^{m}$ is independent. Then for any $g\in G$ and $1\leq l,l'\leq k$ at least one of the following holds:
  	\begin{itemize}
  		\item $l=l'$ and $g$ fixes some non-degenerate subsegment of $J_{l}$
  		\item $|gJ_{l'}\cap gJ_{l}|< (4e_{max}+k+5)C_{f}$,
  	\end{itemize}
  	where $C_{f}:=\max\,\{|J|\,\}_{J\in f^{-1}(\{a,b\})}$.
  \end{lemma}
  \begin{proof}
  	We may assume (unoriented intersection) $gJ_{j}\cap J_{j'}$ is not degenerate.
  	If $gJ_{l'}$ and $J_{l}$ intersect coherently, let
  	$J'=J_{l'}$. Otherwise, let $J'=J_{l'}^{op}$.
  	Then $w:=f^{*}(gJ'\cap J_{l})$ is a common subword of
  	$f^{*}(J')$ and $f^{*}(J_{l})$.
  	
  	If $|gJ'\cap J_{l}|\geq (4e_{max}+k+5)C_{f}$ and since the distance between two points in the same orbit is bounded above by $C_f$, we have that
  	there exists some $o$-geodesic $\tilde{J}$ contained in $gJ'\cap J_{l}$  with
  	$|\tilde{J}|\geq (4e_{max}+k+3)C_{f}$ and so $|\tilde{J}|_{o}\geq 4e_{max}+k+3$. The aforementioned inequality and the definition of the $(k,f)$ implies that
  	$|f^{*}(\tilde{J})|\geq 4e_{max}+4$. and so there is some $o$-geodesic $\hat{J}$ contained in $\tilde{J}$ for which
  	$f^{*}(\hat{J})=ab^{e^{l}_{i}}ab^{e^{l}_{i+1}}ab^{e^{l}_{i+2}}a$.
  	
  	In case $J'=J_{l'}$ it follows immediately from the independence of $(\bar{e}^{j})_{j=1}^{m}$ that
  	$l=l'$ and the embedding of $\hat{J}$ into $J_{l}$ and $gJ_{l}$ corresponds to identical instances of $w$ in $f^{*}(J_{l})=f^{*}(gJ_{l})$. This implies that $g$ fixes  $J_{l}\cap gJ_{l}=J_{l}\cap gJ_{l'}$ point-wise.
  	
  	Assume now that $J'=J_{l}^{op}$ and apply the Lemma \ref{reverse} to $w$ and $J_{l'}$. If the alternative (\ref{reverse 2}) in Lemma \ref{reverse} holds, then it implies the existence of $1\leq i'\leq|\bar{e}^{l'}|$ and $\epsilon\in\{0,1,2\}$
  	such that $2e^{l}_{i+\epsilon}\leq e^{l'}_{i'}$, contradicting condition (\ref{independence 2}) in Definition \ref{defn:independenttuple}. If the alternative Case (\ref{reverse 1}) of Lemma \ref{reverse} holds, then it implies that
  	$l=l'$ and $e^{l}_{i+\delta}\leq e^{l}_{i-\delta}$ for $\delta\in\{-1,1\}$, which is also impossible.
  \end{proof}

  \begin{cor}
  	\label{c: labellings and wsc}Assume that an action of a group $G$ on a simplicial tree $T$ admits a good labeling $(k,f)$ and for any
  	sequence $\bar{e}$ there exists an $o$-geodesic $J$ such that $f^{*}(J)=u_{\bar{e}}(a,b)$. Then for any $N,m>0$ there exists a weak $N$-small cancellation $m$-tuple of elements from $G$ (with respect to the action of $G$ on $T$).
  \end{cor}
  \begin{proof}
  	Let choose $C$ such that $|J|\leq C$ for any $o$-geodesic $J$ such that $f^{*}(J)\in\{a,b\}$.
  	One can always choose an independent tuple $(\bar{e}^{j})_{j=1}^{m}$, where  $\bar{e}^{j}$ has length $q>0$ for all $j$ and the following holds:
  	\begin{itemize}
  		\item $e^{j}_{i}>\frac{N}{q}$ for any $1\leq j\leq m$ and $1\leq i\leq q$
  		\item $2e^{j}_{i}>e^{j'}_{i'}$
  		for any $1\leq j,j'\leq m$, $1\leq i\leq |\bar{e}^{j}|$
  		and $1\leq i\leq q$
  		\item $N(4e_{max}+k+5)C_{f}\leq 2q\min_{i,j}e^{j}_{i}$

  	\end{itemize}
  	Any tuple $(a_{j})_{j=1}^{m}\in G^{m}$ such that $f^{*}(J_{j})=u_{\bar{e}^{j}}(a,b)$ where $J_{j}=[v,a_{j}v]$  (notice such tuple must exist, by hypothesis) is weakly $N$-small cancellation.
  	Let $L=\max_{1\leq l\leq m}|J_{l}|$.
  	
  	Indeed, given $g\in G$ and $1\leq l, l'\leq m$ such that intersection $|gJ_{l'}\cap J_{l}|\geq\frac{L}{N}$, the third bound above together with Lemma \ref{l: gl intersection} implies that $l=l'$ and $g$ fixes $J_{l}\cap J_{l'}$ point-wise.
  \end{proof}

  The conclusion of the first part of the proof of Theorem $1$ in  \cite{iozzi2014characterising} can be formulated as follows:
  \begin{thm}\label{thm:1SistoCo}
  	Suppose that $G$ acts minimally on a tree $T$ (i.e.\ there is no proper subtree that is invariant under this action) and that the valence of every vertex of $T$ is at least $3$. Then exactly one of the following alternatives holds:
  	\begin{enumerate}[(i)]
  		\item \label{alternative transitivity}There exists $l\in\{1,2\}$ such that for all $n>0$ and any vertex $v$ of $T$ the group  $G$ acts transitively on
  		\begin{align*}
  			\{[x,y]\in\mathcal{E}^{(n)}\,|\,d(x,v)\equiv 0\,\,(mod\,\,l)\}.
  		\end{align*}
  		\item \label{alternative fixed end}$G$ fixes one end of $T$.
  		\item \label{alternative wsc} The action of $G$ on $T$ admits a good labeling $(k,f)$ and for any sequence $\bar{e}$ there exists an $o$-geodesic $J$ such that $f^{*}(J)=u_{\bar{e}}(a,b)$.
  	\end{enumerate}
  \end{thm}
  
  We will say that a minimal action of a group $G$ on a simplicial tree $T$ (without inversions) is \emph{weakly acylindrical} if $T$ is not a line and satisfies condition (\ref{alternative wsc}) from Theorem \ref{thm:1SistoCo}.
  
  \begin{prop}\label{p: from_weak_acyl_to_sc}
  	Let $G$ be a group acting on a simplicial tree. Then the action of
  	$G$ is weakly acylindrical if and only if there exists a weak $N$-small cancellation $m$-tuple of elements from $G$, for all $m,N>0$.
  \end{prop}
  \begin{proof}
  	Clearly the existence of $N$-small cancellation $m$-tuples is incompatible with either (\ref{alternative transitivity}) or (\ref{alternative fixed end}) in  Theorem \ref{thm:1SistoCo}, giving us the 'if' part, while the 'only if' part follows from Corollary \ref{c: labellings and wsc}.
  \end{proof}
  
  One could see the second part of Theorem $1$ in \cite{iozzi2014characterising} as the statement that any group admitting a weakly acylindrical action on a simplicial tree has infinite-dimensional second bounded cohomology group.
  
  We combine Proposition \ref{p: from_weak_acyl_to_sc} and Corollary \ref{c: iterated formal solutions} to obtain the main result of the paper.
  
  \begin{thm}\label{thm:characterisation}
  	If a group $G$ admits a weakly acylindrical action on a simplicial tree, then $G$ has trivial positive theory.
  \end{thm}
  
\section{Examples of groups with (non-)trivial positive theories}\label{s: applications}

 In this section we study concrete families of groups, namely acylindrically hyperbolic groups acting on trees, generalised Baumslag-Solitar, one-relator and graph products of groups and use our results to characterised them in terms of their positive theory. We summarize the results of this section in the following Corollary.
  
  \begin{cor}\label{cor:summaryex}
  Non-virtually  solvable  fundamental  groups  of closed,  orientable,  irreducible 3-manifolds; groups acting acylindrically on a tree;  non-solvable generalised Baumslag-Solitar groups;  (almost all)  non-solvable  one-relator groups;  and  graph  products  of  groups  whose underlying  graph is not complete have trivial positive theory. In particular, these groups have verbal subgroups of infinite width and are not boundedly simple.
  \end{cor}

      \paragraph{Acylindrically hyperbolic groups acting on trees}

        \begin{cor} \label{cor:acylindricallyhyperbolic}
        	Every group $G$ acting acylindrically hyperbolically and irreducibly on a tree has trivial positive theory. In particular, all verbal subgroups have infinite width.
        \end{cor}
        
        \begin{proof}
        	Indeed, Osin showed in \cite{osin2016acylindrically} that if the group $G$ has a non-elementary acylindrical action on a hyperbolic space, then it has a hyperbolic WPD element and so small cancellation elements.
        \end{proof}

        In \cite{bestvina2015constructing}, the authors show that in acylindrically hyperbolic groups, all verbal subgroups have infinite verbal width. Our result is a generalisation in the case that the group acts non-elementary acylindrically on a tree.
   
      \paragraph{3-manifold groups}
              
        Let $G$ be the fundamental group of a $3$-manifold of a closed, orientable, irreducible $3$-manifold $M$. By \cite[Lemma 2.4]{wilton2010profinite}, $M$ has a finite-sheeted covering space that is a torus bundle over a circle, or the JSJ decomposition of $G$ is 4-acylindrical. Hence, from Theorem \ref{thm:characterisation}, we have the following result:
        
        \begin{cor}
        	Let $G$ be the fundamental group of a closed, orientable, irreducible $3$-manifold. Then either the positive theory of $G$ is trivial or $G$ is virtually polycyclic (and so its positive theory is not trivial).
        \end{cor}
        
        \begin{proof}
        	Indeed, if $M$ has a finite-sheeted covering space that is a torus bundle over a circle, then by Thurston's classification, $M$ is virtually polycyclic. If the JSJ decomposition of $G$ is 4-acylindrical, then it contains weakly stable elements.
        \end{proof}
        
        We recover one of the results from  \cite{bestvina2015constructing}.
        
        \begin{cor}[cf. Theorem 1.1 in \cite{bestvina2015constructing}]
        	Let $G$ be the fundamental group of a 3-manifold and suppose that $G$ is non-polycyclic. Then all verbal subgroups of $G$ have infinite width.
        \end{cor}
        
      \paragraph{Generalised Baumslag-Solitar groups}
        The next class of groups we consider is the family of generalised Baumslag-Solitar groups, that is fundamental groups of graphs of groups with infinite cyclic vertex and edge groups.
        
        \begin{cor}
        	A Baumslag-Solitar group $BS(m,m)=\langle a, y\mid a^n = y^{-1}a^m y\rangle$ either has trivial positive theory or it is solvable and so $n=1$ or $m =1$.
        \end{cor}
        
        \begin{proof}
        	The element $g=tat^{-1}a$ is weakly 2-small cancellation and the action is irreducible.
        \end{proof}
        
        \begin{cor}\label{cor:GBS}
        	A generalised Baumslag-Solitar group $G$ has trivial positive theory or it is isomorphic to $BS(1,n)$, for some $n\in \mathbb Z$ or it is infinite cyclic.
        \end{cor}
        
        \begin{proof}
        
        	The proof is almost verbatim to that of \cite{button2016nonhyperbolic}[Theorem 3.2] replacing the SQ-universal property by having trivial positive theory. One should also note that if a group $G$ has a quotient with trivial positive theory, then so does $G$, and that by \cite{sela2009diophantineVII} torsion-free hyperbolic groups have trivial positive theory.
        \end{proof}
        
      \paragraph{One-relator groups}
        
        There is a natural conjecture formulated in \cite{neumann1973sq} in 1973 characterising when 1-relator groups are SQ-universal, namely,  a non cyclic
        1-relator group is SQ-universal unless it is isomorphic to $BS(1, n)$, for some $n\in \mathbb Z$. A similar conjecture should hold if one replaces SQ-universal by having trivial positive theory.
        
        \begin{conj}
        	A non cyclic 1-relator group has trivial positive theory unless it is isomorphic to $BS(1, n)$, for some $n\in \mathbb Z$.
        \end{conj}
        
        In order to address this conjecture, we will follow the strategy used in \cite{button2016nonhyperbolic} to approach the question of SQ-universality.
        
        We first use the trichotomy given by Minsasyian-Osin in \cite{minasyan2015acylindrical} on the structure of one-relator groups.

        \begin{prop}(Proposition 4.20 \cite{minasyan2015acylindrical})
        	Let $G$ be a one-relator group. Then at least one of the following holds:
        	
        	\begin{itemize}
        		\item [(i)] $G$ acts acylindrically hyperbolically on a tree;
        		\item [(ii)]$G$ is 2-generated and contains an infinite cyclic s-normal subgroup. More precisely, either
        		$G$ is infinite cyclic or it is an HNN-extension of the form
        		$$
        		G=\langle a,b,t \mid a^t=b, w=1\rangle
        		$$
        		
        		of a 2-generator 1-relator group $H=\langle a, b \mid w(a,b) \rangle$ with non-trivial center, so that a
        		$r^b=b^s$ for some $r,s \in \mathbb Z\setminus \{0\}$. In the latter case $H$ is (finitely generated free)-by-cyclic and contains a finite index normal
        		subgroup splitting as a direct product of a finitely generated free group with an infinite cyclic group.
        		
        		\item [(iii)] $G$ is 2-generated and isomorphic to an ascending HNN extension of a finite rank free group.
        	\end{itemize}
        	
        \end{prop}
        
        Theorem \ref{thm:characterisation} establishes that 1-relator groups in the class $(i)$ have trivial positive theory.
        
        It follows from \cite{button2016nonhyperbolic}[Theorem 3.2] that groups from class $(ii)$ are generalised Baumslag-Solitar groups and so by Corollary \ref{cor:GBS} it follows that they also have trivial positive theory unless they are solvable.
        
        In the third alternative, it follows from \cite{button2008large} that an ascending HHN extension is either large (and so has trivial positive theory), or hyperbolic (and again has trivial positive theory) or it contains a solvable Baumslag-Solitar group, see discussion previous to \cite{button2016nonhyperbolic}[Corollary 3.4] for details. Summarizing, we have the following corollary.

        \begin{cor}\label{cor:1-relator}
        	All 1-relator groups with more than 3 generators have trivial positive theory.
        	
        	If $G$ is a group given by a $2$-generator 1-relator presentation
        	that is not $\mathbb Z$, $\mathbb Z \times \mathbb Z$ or the Klein bottle group then either $G$ has trivial positive theory or $G$ is a strictly ascending HNN extension $F_k \ast_\theta$ of a free group $F_k$
        	which is not word hyperbolic and such that:
        	\begin{itemize}
        		\item either $G$ contains no Baumslag-Solitar subgroup (conjecturally this does not occur) or
        		\item $G$ contains a Baumslag-Solitar group $BS(1, m)$, for $|m| \ne 1$ but does not contain $\mathbb Z \times \mathbb Z$ and the virtual first Betti number of $G$ is 1 (conjecturally this
        		only occurs if $G=BS(1, n)$ for $|n| \ne1$).
        	\end{itemize}
        \end{cor}

      \paragraph{Graph products of groups}
        \newcommand{\grpr}{\mathcal G(\Gamma,G_{v})_{v\in V}}
        \begin{thm}
        	Let $G =\grpr$ be the graph product of non-trivial groups with respect
        	to a non-complete finite graph $\Gamma$. Suppose that $H < G$ is a
        	subgroup such that no conjugate of $H$ is contained in a subgroup of $G$ defined by a proper subgraph. Then either $H$ is virtually cylic or it has trivial positive theory.
        \end{thm}
        
        \begin{proof}
        	
        	Any graph product $\grpr$ can be seen as a free product with amalgamation
        	$$
        	\grpr = \mathcal G(\Gamma \setminus \{v\}) \ast_{\mathcal G(link(v))} (\mathcal G(star(v)),
        	$$ where $link(v)$ is the set vertices of $\Gamma$ joined to $v$ by an edge and $star(v)=link(v) \cup \{v\}$, and so it admits an action on a tree.
        	
        	If the graph is not bipartite, then any hyperbolic element $g$ such that no conjugate of $g$ belongs to a subgroup defined by a proper subgraph, is 2-stable. Therefore, if $H$ is not virtually cyclic and no conjugate of $H$ is contained in a subgroup defined by a proper subgraph, then $H$ acts irreducible on the tree and contains a $2$-stable elements and so small cancellation tuples. It follows that $H$ has trivial positive theory.
        	
        	If the graph is bipartite, since it is not complete, the graph product has a retraction to a subgroup defined by a proper subgraph which is not bipartite and so its positive theory is trivial.
        \end{proof}
        
        The corollary below covers the particular case when $H = G$.
        \begin{cor}
        	Let $G =\grpr$ be a graph product of non-trivial groups with respect
        	to a finite graph
        	$\Gamma$. If the graph is non-complete, then $G$ has trivial positive theory. If the graph is complete, then $G$ has trivial positive theory if and only if one of its vertex groups has trivial positive theory.
        \end{cor}
        
        \begin{proof}
        	In order to prove the corollary we are only left to consider the case when the graph is complete and so the graph product is a direct product. If one of the vertex groups has trivial positive theory, then so does the groups since it admits a retraction to the vertex group. If all vertex groups have non-trivial positive theory, then since by Theorem \ref{lem:extensionpreservation} it is a property closed under extensions, we have that $G$ also has non-trivial positive theory.
        \end{proof}
        
        Note that in \cite{diekert2004existential}, Diekert and Lohrey studied the positive theory of graph products in the language of groups with constants. In that case, the authors show that the positive theory of equations with recognizable constraints in graph products of finite and free groups is decidable.
        
      \paragraph{Simple groups}
        
        Using the main result of this paper, the authors show in \cite{simplegrouppreprint} that there are uncountably many finitely generated simple groups with trivial positive theory and so with verbal subgroups of infinite width. These groups are described as free products with amalgamation of finitely generated free groups over (non finitely generated) infinite index subgroups. On the one, this result contrast with the uniform bound on the width of verbal subgroups for the class of finite simple groups proven by Liebeck, O'Brien, Shalev and Tiep in \cite{liebeck2010ore}. On the other one, it also contrasts with the result by Schupp on the SQ-universality of free products with amalgamation of finitely generated free groups over finitely generated infinite index subgroups.

\section{Questions}
  \label{sec:openquestions}
  
  \subsection{Hyperbolic spaces}
    
    In this paper we study the relation between groups acting on trees and their positive theory. As we remarked, the notion of weakly stable element in the context of groups acting on trees coincides with that of WWPD introduced by Bestvina, Bromberg and Fujiwara in \cite{bestvina2013stable}. In that paper, the authors characterise elements with positive stable commutator length and use it to show that the mapping class group has infinite dimensional second bounded cohomology. It is natural to ask if our results can be generalised to groups acting on quasi-trees or hyperbolic spaces, namely,
    \begin{conj}
    	Let $G$ be a group acting on a hyperbolic space with independent WWPD\footnote{for trees, the equivalent condition in our language is weakly stable} (as defined in \cite{bestvina2015constructing}). Then either $G$ is virtually cyclic or the positive theory of $G$ is trivial. In particular, all acylindrically hyperbolic groups have trivial positive theory and so infinite verbal width.
    \end{conj}
    
  \subsection{Non weakly acylindrical actions on trees}
    
    Our main result can be stated as if a group acting minimally on a tree of valency at least 3 in each vertex and with the action on the boundary neither fixing a point nor 2-transitive, then the group has trivial positive theory. It is natural to ask how close this is to be a characterisation. More precisely, we ask whether there are groups with a faithful, minimal action on a tree with valency at least 3 in each vertex and 2-transitive on the boundary (alternative (i) of Theorem \ref{thm:1SistoCo}), and have trivial positive theory. Of special interest is the Burger-Mozes group, since it admits this type of action (see also next section).
    
    \begin{question}
    	Do Burger-Mozes groups have non-trivial positive theory? Are they boundedly simple of have finite commutator width?
    \end{question}
    
    Proving that the positive theory of some of these groups is trivial would involve developing new tools which would be of interest on their own. Furthermore, this would provide examples of finitely presented simple groups with trivial positive theory. On the other hand, if the Burger-Mozes groups have non-trivial positive theory, then we would conclude that this property is not a quasi-isometric invariant, since Burger-Mozes groups are quasi-isometric to the direct product of two free groups, which does have trivial positive theory. This would also indicate that groups which act faithfully and 2-transitively on the  boundary do not have trivial positive theory, bringing us closer to a full criterion for group acting on trees. In this direction, and in views of the remaining alternative in Theorem \ref{thm:1SistoCo}, we also ask:
    
    \begin{question}
    	Can one characterise when finitely generated groups acting on a tree and fixing a boundary point have trivial positive theory?
    \end{question}
    
    Note that there are strict HNN-extensions which are hyperbolic and so with trivial positive theory.

  \subsection{Trivial positive theory and second bounded cohomology}

    All results known to the authors that show that either the group has infinite dimensional second bounded cohomology or trivial positive theory rely in the existence of weak small cancellation tuples, see for instance \cite{iozzi2014characterising}, \cite{bestvina2015constructing}.
    
    It is natural to ask about the relation between these two concepts. More precisely, we ask the following questions.
    
    \begin{question}
    	Are there (finitely generated) groups with trivial positive theory whose second bounded cohomology group is finite dimensional?
    \end{question}
    
    \begin{question}
    	Are there (finitely generated) groups with infinite dimensional second bounded cohomology but with non-trivial positive theory?
    \end{question}
    
    These questions already can be considered for groups acting on trees. Indeed, if a group admits a minimal irreducible action on a tree and contains weak small cancellation elements (which we prove to be equivalent to have a good labeling, see Section \ref{sec: characterisation}),  then the second bounded cohomology is infinite-dimensional, see \cite{iozzi2014characterising} and the positive theory is trivial, see Theorem \ref{thm: wscitp}. 
    
    On the other hand, a cocompact lattice in a product of locally finite trees, only has trivial quasimorphisms if and only if both closures of the projections on the two factors act 2-transitive on the boundary, see \cite{iozzi2014characterising} and \cite{burger2000groups}. Some of this lattices, for instance the Burger-Mozes groups, act on a tree with a 2-transitive action on the boundary.

    \begin{question}
    	Let $\Gamma$ be a cocompact lattice in a product of locally finite trees such that both closures of the projections on the two factors are 2-transitive in the boundary (and so $\Gamma$ has trivial second bounded cohomology). Does $\Gamma$ have non-trivial positive theory? In particular, do Burger-Mozes groups have non-trivial positive theory?
    \end{question}

  \subsection{Weakly stable and weakly small cancellation}
    
    It follows from our results that a group acting irreducibly on a tree has hyperbolic stable elements if and only if it has small cancellation elements (see Remark \ref{r: stable_and_sc}).  
    
    However,  it is not clear if such a characterisation still holds after weakening (as in Definitions \ref{defn:weakly stable} and \ref{d: small cancellation}) these two properties: the existence of weakly stable elements implies the existence of weakly small cancellation elements, but does the converse hold? More concretely, we ask the following questions.
    
    \begin{question}
    	Does the existence of weak-small cancellation elements imply the existence of weak stable ones?
    \end{question}
    One could also formulate a weaker question:
    \begin{question}
    	Does there exist $N_0>0$ such that the existence of a weakly $N_0$-small cancellation tuple implies the existence of a weakly $M$-small cancellation tuple for arbitrary $M$?
    \end{question}
    
  \subsection{Closure under finite extensions}
    
    In Section \ref{sec:preservation extensions} we have shown that the class of groups with non-trivial positive theory is closed under extensions. However, do not know whether or not it is a commensurability invariant: if $H$ is a finite index subgroup of $G$ and $G$ has non-trivial positive theory, does $H$ have also non-trivial positive theory? Note that if the answer is affirmative together with our result on preservation under extension would imply that having a non-trivial positive theory is a commensurability invariant.
    
    Equivalently, one can ask if having trivial positive theory is preserved under finite extensions.
    
    \begin{question}
    	Do finite extensions of groups with trivial positive theory have trivial positive theory?
    \end{question}
    
  \subsection{Trivial positive theory and SQ-universality}
    
    It is an easy observation that if a group $G$ admits an epimorphism to a group $H$, then $Th^+(G) \subset Th^+(H)$. This may naively lead to think that a group rich in quotients would have a trivial positive theory. This intuition is confirmed for large groups so one may wonder about SQ-universal groups. Notice that, as proven in Section \ref{s: applications}, many notorious classes of SQ-universal groups do have trivial positive theory, namely groups acting on trees acylindrically hyperbolically, generalised Baumslag-Solitar groups, non-solvable 1-relator groups, etc. We propose the following question:
    
    \begin{question}
    	Do all SQ-universal groups have trivial positive theory?
    \end{question}
    
    The converse is by far not true. Indeed, in the preprint \cite{simplegrouppreprint} we show that there exist fg (torsion-free) simple groups with trivial positive theory. In particular, these finite simple groups have all verbal subgroups of infinite width, generalising the result of Muranov \cite{muranov2007finitely}, are not boundedly simple (i.e. for all $n$ there exist $g$ and $h$ non-trivial such that $g$ is not a product of $n$ conjugates of $h$), etc.
    
  \subsection{Trivial positive theory and amenability}
    
    All the tools known to the authors  to prove that a group has trivial positive theory rely on having (weak) small cancellation elements. One of its implications is that all of these groups contain non-abelian free groups and so they are non-amenable. It is reasonable to ask for examples of non-amenable groups which do not contain free groups and have trivial positive theory.
    
    \begin{question}
    	Is there a finitely presented (finitely generated) non-amenable group without non-cyclic free subgroups that has trivial positive theory? Do Golod-Shafarevich groups have trivial positive theory?
    \end{question}
    
    On the other extreme of the spectrum, one can consider the class of amenable groups. Chou proved that the class of elementary amenable groups can be characterised as the smallest class of groups containing finite and abelian groups and being closed under extensions and direct limits.
    
    As we show in Section \ref{sec:preservation extensions}, the class of groups with non-trivial positive theory is closed under extensions. However we believe that the class may not be closed under direct limits. More precisely, we ask:
    
    \begin{question}
    	Does the direct limit of $UT_n(\mathbb Z)$ (with the natural embeddings) have trivial positive theory?
    \end{question}
    
    This example would provide an elementary amenable group with trivial positive theory and would show that indeed the class of groups with non-trivial positive theory is not  closed under direct limits.
    
    Although we do not believe that the class is closed under direct limits, we do not know of any finitely generated elementary amenable or amenable group whose positive theory is trivial.
    
    \begin{question}
    	Is there a finitely generated (elementary) amenable group with trivial positive theory?
    \end{question}
    
    In this direction, the following examples are of special interest to study: the Olshanskii-Osin-Sapir example of a lacunary hyperbolic amenable group
    and the Basilica group. 
    
    \medskip
    
    Although we do not believe that containing a non-cyclic free group is going to be a necessary condition in order to have trivial positive theory, we do believe that there may be a strong relation with the growth of the group and the triviality of its positive theory. More specifically, we ask:
    
    \begin{question}
    	Does trivial positive theory imply exponential growth?
    \end{question}
    
  \subsection{Trivial positive theory and verbally parabolic groups}
    
    We say that a group is verbally elliptic for the word $w$ if the verbal subgroup defined by a non-trivial word $w(x)$ has finite width.
    If $G$ is not verbally elliptic for any word, we say that it is verbally parabolic.
    
    Hence if a group $G$ is verbally elliptic for a word $w$, its positive theory is non-trivial or conversely, if a group has trivial positive theory, then it is verbally parabolic.
    
    Note that all groups we know to have non-trivial positive theory are verbally elliptic for some word (note that finitely many conjugacy classes implies finite commutator width). Conversely, all verbally parabolic groups known have trivial positive theory.
    
    \begin{question}
    	Does the class of groups with trivial positive theory coincide with the class of verbally parabolic groups?
    \end{question}
   
    \begin{question}
    	Can one use small-cancellation techniques to construct an infinite simple group verbally parabolic so that the group does not contain free groups?
    \end{question}
    
    If such a group exists, we would get a non-amenable group without free groups and trivial positive theory.
    
    On the other hand, if $G$ has non-trivial positive theory, then we would get a verbally parabolic group with non-trivial positive theory.

  \bibliographystyle{alpha}
  \bibliography{references}
\end{document}